\def\date{4.4.2017} 
\documentclass{bcp}
\usepackage{hyperref}
\usepackage{stmaryrd}
\usepackage{amsmath}
\usepackage{latexsym}
\usepackage{amssymb}
\usepackage{amsthm}

\newtheorem{theorem}{Theorem}[section]
\newtheorem{proposition}[theorem]{Proposition}
\newtheorem{lemma}[theorem]{Lemma}
\newtheorem{corollary}[theorem]{Corollary}

\theoremstyle{definition}
\newtheorem{definition}[theorem]{Definition}

\newtheorem{remark}[theorem]{Remark}

\def\examplesname{Examples}
\def\examplename{Example}

\newtheorem{exas}[theorem]{\examplesname}
\newenvironment{exs}{\begin{exas}\rm}{\end{exas}}
\newtheorem{exa}[theorem]{\examplename}
\newenvironment{ex}{\begin{exa}\rm}{\end{exa}}

\newcommand{\cJ}{\mathcal J}
 
\newcommand{\cY}{\mathcal Y}

\newcommand{\Pin}{\mathop{{\rm Pin}}\nolimits}
\newcommand{\Out}{\mathop{{\rm Out}}\nolimits}
\newcommand{\grp}{\mathop{{\rm grp}}\nolimits}
\newcommand{\mc}{\mathop{{\rm mc}}\nolimits}
\newcommand{\ms}{\mathop{{\rm ms}}\nolimits}
\newcommand{\Conj}{\mathop{{\rm Conj}}\nolimits}
\newcommand{\Inv}{\mathop{{\rm Inv}}\nolimits}

\newcommand{\AAut}{\mathop{{\rm AAut}}\nolimits}
\renewcommand{\phi}{\varphi} 

\newcommand{\AdS}{\mathop{{\rm AdS}}\nolimits}

\newcommand{\Stand}{\mathop{{\rm Stand}}\nolimits}
\newcommand{\cs}{\mathop{{\rm cs}}\nolimits}
\newcommand{\Comm}{\mathop{{\rm Comm}}\nolimits}
\newcommand{\PO}{\mathop{{\rm PO}}\nolimits}

\newcommand\lbr{\llbracket}
\newcommand\rbr{\rrbracket}

\font\tengoth=eufm10 at 10pt
\font\sevengoth=eufm7 at 6pt
\newfam\gothfam
\textfont\gothfam=\tengoth
\scriptfont\gothfam=\sevengoth

\newcommand{\mlabel}[1]{\marginpar{#1}\label{#1}}

\newcommand{\fS}{{\mathfrak S}}

\newcommand{\g}{{\mathfrak g}}

\newcommand{\fb}{{\mathfrak b}}

\newcommand{\fh}{{\mathfrak h}}

\newcommand{\fp}{{\mathfrak p}}

\newcommand{\fu}{{\mathfrak u}}

\renewcommand{\:}{\colon}
\newcommand{\1}{\mathbf{1}}

\newcommand{\cA}{\mathcal{A}}

\newcommand{\cD}{\mathcal{D}}

\newcommand{\cF}{\mathcal{F}}
\newcommand{\cG}{\mathcal{G}}
\newcommand{\cH}{\mathcal{H}}
\newcommand{\cK}{\mathcal{K}}

\newcommand{\cM}{\mathcal{M}}
\newcommand{\cN}{\mathcal{N}}
\newcommand{\cO}{\mathcal{O}}
\newcommand{\cP}{\mathcal{P}}

\newcommand{\cR}{\mathcal{R}}
\newcommand{\cS}{\mathcal{S}}
\newcommand{\cT}{\mathcal{T}}

\newcommand{\cV}{\mathcal{V}}
\newcommand{\cW}{\mathcal{W}}
\newcommand{\cZ}{\mathcal{Z}}

\newcommand\bx{{\bf{x}}}
\newcommand\by{{\bf{y}}}

\newcommand\bp{{\bf{p}}}

\newcommand{\eset}{\emptyset}

\newcommand{\dd}{{\tt d}}

\newcommand{\trile}{\trianglelefteq}
\newcommand{\subeq}{\subseteq}
\newcommand{\supeq}{\supseteq}

\newcommand{\into}{\hookrightarrow}
\newcommand{\eps}{\varepsilon}

\newcommand{\N}{{\mathbb N}}
\newcommand{\Z}{{\mathbb Z}}
\newcommand{\R}{{\mathbb R}}
\newcommand{\C}{{\mathbb C}}

\newcommand{\bP}{{\mathbb P}}

\renewcommand{\H}{{\mathbb H}}
\newcommand{\T}{{\mathbb T}}

\newcommand{\bH}{{\mathbb H}}
\newcommand{\bS}{{\mathbb S}}

\renewcommand{\hat}{\widehat}

\renewcommand{\tilde}{\widetilde}

\renewcommand{\L}{\mathop{\bf L{}}\nolimits}

\newcommand{\Aff}{\mathop{{\rm Aff}}\nolimits}
\newcommand{\CAR}{\mathop{{\rm CAR}}\nolimits}

\newcommand{\GL}{\mathop{{\rm GL}}\nolimits}
\newcommand{\SL}{\mathop{{\rm SL}}\nolimits}

\newcommand{\PGL}{\mathop{{\rm PGL}}\nolimits}
\newcommand{\PAU}{\mathop{{\rm PAU}}\nolimits}
\newcommand{\AU}{\mathop{{\rm AU}}\nolimits}

\newcommand{\PSL}{\mathop{{\rm PSL}}\nolimits}
\newcommand{\SO}{\mathop{{\rm SO}}\nolimits}
\newcommand{\SU}{\mathop{{\rm SU}}\nolimits}
\newcommand{\OO}{\mathop{\rm O{}}\nolimits}

\newcommand{\U}{\mathop{\rm U{}}\nolimits}

\newcommand{\aff}{\mathop{{\mathfrak{aff}}}\nolimits}

\newcommand{\fsl} {\mathop{{\mathfrak{sl} }}\nolimits}

\newcommand{\so}  {\mathop{{\mathfrak{so} }}\nolimits}

\newcommand{\Exp}{\mathop{{\rm Exp}}\nolimits}
\newcommand{\Fix}{\mathop{{\rm Fix}}\nolimits}

\newcommand{\ad}{\mathop{{\rm ad}}\nolimits}
\newcommand{\Ad}{\mathop{{\rm Ad}}\nolimits}

\renewcommand{\Re}{\mathop{{\rm Re}}\nolimits}
\renewcommand{\Im}{\mathop{{\rm Im}}\nolimits}

\newcommand{\tr}{\mathop{{\rm tr}}\nolimits}

\newcommand{\Hom}{\mathop{{\rm Hom}}\nolimits}

\newcommand{\Herm}{\mathop{{\rm Herm}}\nolimits}

\newcommand{\Heis}{\mathop{{\rm Heis}}\nolimits}

\newcommand{\hsm}{\mathop{\rm hsm}\nolimits}

\newcommand{\Car}{\CAR}

\newcommand{\Aut}{\mathop{{\rm Aut}}\nolimits}
\newcommand{\Conf}{\mathop{\rm Conf{}}\nolimits}

\newcommand{\Diff}{\mathop{{\rm Diff}}\nolimits}

\newcommand{\Vir}{\mathop{{\rm Vir}}\nolimits}
\newcommand{\diag}{\mathop{{\rm diag}}\nolimits}

\newcommand{\id}{\mathop{{\rm id}}\nolimits}

\renewcommand{\dim}{\mathop{{\rm dim}}\nolimits}

\newcommand{\im}{\mathop{{\rm im}}\nolimits}

\newcommand{\Inn}{\mathop{{\rm Inn}}\nolimits}

\newcommand{\sgn}{\mathop{{\rm sgn}}\nolimits}

\newcommand{\ev}{\mathop{{\rm ev}}\nolimits}

\newcommand{\dS}{\mathop{{\rm dS}}\nolimits}
\newcommand{\PSO}{\mathop{{\rm PSO}}\nolimits}

\newcommand{\oline}{\overline}

\newcommand{\la}{\langle}
\newcommand{\ra}{\rangle}

\newcommand{\up}{\mathop{\uparrow}}

\newcommand{\res}{\vert}

\newcommand{\Spec}{{\rm Spec}}

\newcommand{\ssssarr}{\hbox to 15pt{\rightarrowfill}}
\newcommand{\sssarr}{\hbox to 20pt{\rightarrowfill}}
\newcommand{\ssarr}{\hbox to 30pt{\rightarrowfill}}
\newcommand{\sarr}{\hbox to 40pt{\rightarrowfill}}
\newcommand{\arr}{\hbox to 60pt{\rightarrowfill}}
\newcommand{\larr}{\hbox to 60pt{\leftarrowfill}}
\newcommand{\Arr}{\hbox to 80pt{\rightarrowfill}}
\newcommand{\mapdown}[1]{\Big\downarrow\rlap{$\vcenter{\hbox{$\scriptstyle#1$}}$}}

\newcommand{\mapright}[1]{\smash{\mathop{\arr}\limits^{#1}}}

\newcommand{\pmat}[1]{\begin{pmatrix} #1 \end{pmatrix}}
\newcommand{\mat}[1]{\begin{matrix} #1 \end{matrix}}

\renewcommand{\mlabel}{\label}

\begin{document}

\keywords{antiunitary representation, modular 
operator, modular conjugation, von Neumann algebra, 
local observables, half-sided modular inclusion 
quantum field theory} 
\mathclass{Primary 22E45; Secondary 81R05, 81T05.}

\abbrevauthors{K.-H. Neeb and G. \'Olafsson} 
\abbrevtitle{Antiunitary representations} 

\title{Antiunitary representations\\ and modular theory}

\author{Karl-Hermann Neeb} 
\address{Department Mathematik, Universit\"at Erlangen-N\"urnberg\\
Cauerstrasse 11, 91058 Erlangen, Germany\\ 
E-mail: neeb@math.fau.de}

\author{Gestur \'Olafsson} 
\address{Department of mathematics, Louisiana State University \\ 
Baton Rouge, LA 70803, USA \\ 
E-mail: olafsson@math.lsu.edu} 

\maketitlebcp


\begin{abstract}
Antiunitary representations of Lie groups take 
values in the group of unitary and antiunitary 
operators on a Hilbert space $\cH$. 
In quantum physics, antiunitary operators implement 
time inversion or a PCT symmetry, and in the modular 
theory of operator algebras they arise as modular conjugations 
from cyclic separating vectors of von Neumann algebras. 
We survey some of the key concepts 
at the borderline between the theory of local observables 
(Quantum Field Theory (QFT) in the sense of Araki--Haag--Kastler) 
and modular theory of operator algebras from the perspective 
of antiunitary group representations. Here a central point is 
to encode modular objects in  standard 
subspaces $V\subeq \cH$ which in turn are in one-to-one correspondence 
with antiunitary representations of the multiplicative group $\R^\times$. 
Half-sided modular inclusions and modular intersections of standard subspaces 
correspond to antiunitary representations of $\Aff(\R)$, and these provide 
the basic building blocks for a general theory started in the 
90s with the ground breaking work of Borchers and Wiesbrock and 
developed in various directions in the QFT context. The emphasis of these notes 
lies on the translation between configurations of standard 
subspaces as they arise in the context of modular localization 
developed by Brunetti, Guido and Longo, and the more classical 
context of von Neumann algebras with cyclic separating vectors. 
Our main point is that configurations of standard subspaces 
can be studied from the perspective of antiunitary Lie group 
representations and the geometry of the corresponding spaces, 
which are often fiber bundles over ordered symmetric spaces. 
We expect this perspective to provide new and systematic insight 
into the much richer configurations of nets of local observables in QFT. 
\end{abstract}

\tableofcontents

\section{Introduction} 
\mlabel{sec:1} 

One of the core ideas of quantum theory is that the states 
of a quantum system correspond to one-dimensional subspaces 
of a complex Hilbert space~$\cH$, i.e., the elements 
$[v] = \C v$, $v \not=0$, of its projective space $\bP(\cH)$. 
This set carries a geometric structure defined by 
the transition probability 
\[  \tau([v],[w]) := \frac{ |\la v,w \ra|^2}{\la v,v\ra \la w,w\ra} 
\in [0,1]\] 
between two states $[v]$ and $[w]$, 
where $d([v],[w]) = \arccos \sqrt{\tau([v],[w])} \in [0,\pi/2]$ 
is the corresponding Riemannian metric (the Fubini--Study metric),  
turning it into a Riemann--Hilbert manifold. {\it Wigner's Theorem} 
characterizes the automorphisms of $(\bP(\cH),\tau)$, 
resp., the isometries for the metric, as those bijections 
induced on $\bP(\cH)$ by elements of the {\it antiunitary group}~$\AU(\cH)$ 
of all linear and antilinear surjective isometries of $\cH$ 
(\cite{Ba64}).  
Accordingly, we have an isomorphism 
\[ \Aut(\bP(\cH),\tau) \cong \AU(\cH)/\T \1 =: \PAU(\cH)\] 
of $\Aut(\bP(\cH),\tau)$ with the {\it projective antiunitary group $\PAU(\cH)$.}
So any action of a group $G$ by symmetries of a quantum 
system leads to a homomorphism $\oline\pi \: G \to \PAU(\cH)$ and further to a 
homomorphism of the pullback 
extension $G^\sharp = \oline\pi^*\AU(\cH)$ of $G$ by the circle group $\T$ to $\AU(\cH)$, i.e., an antiunitary representation. 
More precisely, for a pair $(G, G_1)$, where $G_1 \subeq G$ is a subgroup 
of index $2$, a homomorphism $U \: G \to \AU(\cH)$ is called an 
{\it antiunitary representation of $(G,G_1)$} if 
$U_g$ is antiunitary for $g \not\in G_1$. 
If $G$ is a topological group with two connected components, then 
we obtain a canonical group pair by $G_1 := G_0$ (the identity component). 
In this case an {\it antiunitary representation of $G$} is a continuous 
homomorphism $U \: G \to \AU(\cH)$ mapping $G\setminus G_0$ into antiunitary operators. 

In the mathematical literature on 
representations, antiunitary operators have never been in the focus, 
whereas in quantum physics one is forced to consider antiunitary operators 
to implement a time-reversal symmetry (\cite{Wig59}). 
If the dynamics of a quantum system is 
described by a unitary one-parameter group $U_t = e^{itH}$, where 
the Hamiltonian $H$ is unbounded and bounded from below, 
then a unitary time reversal operator $\cT$ would 
lead to the relation $\cT H \cT = - H$, which is incompatible with 
$H$ being bounded from below. This problem is overcome 
by implementing time reversal by an antiunitary operator because it imposes no 
restrictions on the spectrum of the Hamiltonian. 
In particular, the PCT Theorem in Quantum Field Theory (QFT) 
which concerns the implementation of a symmetry 
reversing parity (P), charge (C) and time (T), leads to an 
extension of a unitary representation of the 
Poincar\'e group $P(d)^\uparrow_+ = \R^d \rtimes \SO_{1,d-1}(\R)^\uparrow$ 
to an antiunitary representation of 
the larger group $P(d)_+ \cong \R^d \rtimes\SO_{1,d-1}(\R)$ 
(\cite[Thm.~II.5.1.4]{Ha96}). 

In the modular theory of operator algebras 
one studies pairs $(\cM,\Omega)$ consisting 
of a von Neumann algebra $\cM \subeq B(\cH)$ and 
a cyclic separating unit vector $\Omega\in\cH$. 
Then $S(M\Omega) := M^*\Omega$ for $M \in \cM$ defines an unbounded 
antilinear involution, and the polar decomposition 
of its closure $\oline S = J \Delta^{1/2}$ leads to a positive 
selfadjoint operator $\Delta = S^*\oline S$, an antiunitary involution $J$ 
satisfying the {\it modular relation} $J\Delta J = \Delta^{-1}$,  
and $\alpha_t(M) := \Delta^{it}M\Delta^{-it}$ defines automorphisms 
of $\cM$ (see \cite{BR87} and \S\ref{subsec:4.1}). 
In particular, we are naturally led to antiunitary symmetries. 
We say that $(\Delta, J)$ 
is a pairs of {\it modular objects} if $J$ is a conjugation 
and $\Delta$ a positive selfadjoint operator satisfying the modular relation. 

To connect this with QFT, we recall the notion of a 
{\it Haag--Kastler net} of $C^*$-sub\-algebras $\cA(\cO)$ 
of a $C^*$-algebra $\cA$,  
associated to (bounded) regions $\cO$ in $d$-dimensional Minkowski space. 
The algebra $\cA(\cO)$ is interpreted as observables  
that can be measured in the ``laboratory'' $\cO$. 
Accordingly, one requires {\it isotony}, i.e., that $\cO_1 \subeq \cO_2$ implies 
$\cA(\cO_1) \subeq \cA(\cO_2)$ and that the $\cA(\cO)$ generate $\cA$. 
Causality enters by the {\it locality} assumption that 
$\cA(\cO_1)$ and $\cA(\cO_2)$ commute if 
$\cO_1$ and $\cO_2$ are space-like separated, i.e., cannot correspond 
with each other (cf.~Example~\ref{ex:caus-comp}). 
Finally one assumes an action 
$\sigma \: P(d)_+^{\up} \to \Aut(\cA)$ 
of  the connected Poincar\'e group such that 
$\sigma_g(\cA(\cO)) = \cA(g\cO)$. Every Poincar\'e invariant state $\omega$ 
of the algebra $\cA$ now leads by the GNS construction to a covariant 
representation $(\pi_\omega, \cH_\omega, \Omega)$ of $\cA$, and hence 
to a net $\cM(\cO) := \pi_\omega(\cA(\cO))''$ of von Neumann algebras 
on $\cH_\omega$. Whenever $\Omega$ is cyclic and separating for 
$\cM(\cO)$,  we obtain modular objects 
$(\Delta_\cO, J_\cO)$. This connection between the 
Araki--Haag--Kastler theory of local observables  
and modular theory leads naturally to antiunitary group representations 
(cf.\ Section~\ref{sec:5}). 

The starting point for the recent development that led to 
fruitful applications of modular theory in QFT 
was the Bisognano--Wichmann Theorem, asserting that, 
the modular automorphisms $\alpha_t(M) = \Delta^{it}M \Delta^{-it}$ 
corresponding to the algebra $\cM(W)$ of observables corresponding 
to a wedge domain $W$ in Minkowski space (cf.~Definition~\ref{def:wedges}) 
are implemented by the unitary action of a 
one-parameter group of Lorentz boosts preserving $W$ (\cite{BW76}). 
This geometric implementation of modular automorphisms in terms of 
Poincar\'e transformations was an important first step in a 
rich development based on the work of Borchers and Wiesbrock 
in the 1990s \cite{Bo92, Bo95, Bo97, Wi92, Wi93, Wi93c}.
They managed to distill the abstract essence from the Bisognano--Wichmann 
Theorem which led to a better understanding of the 
basic configurations of von Neumann algebras 
in terms of half-sided modular inclusions 
and modular intersections. 
This immediately led to very tight connections between 
the geometry of homogeneous spaces and modular theory \cite{BGL93}. 
In his survey \cite{Bo00}, Borchers described how 
these concepts have revolutionized quantum field theory. 
Subsequent developments can be found in \cite{Tr97, Sch97, Ar99, BGL02, Lo08, JM17}; 
for the approach to Quantum Gravity based on 
Non-commutative Geometry and Tomita--Takesaki Theory, see in particular~\cite{BCL10}. 

A key insight that simplifies matters considerably is that 
modular objects 
$(\Delta, J)$ associated to a pair $(\cM, \Omega)$ of a von Neumann algebra 
$\cM$ and a cyclic separating vector $\Omega$ are completely determined by 
the real  subspace 
\[ V_\cM := \oline{\cM_h\Omega}, \quad \mbox{ where } \quad 
\cM_h = \{ M \in \cM \: M^* = M\}.\] 
It satisfies 
$V_\cM \cap i V_\cM = \{0\}$ and $V_\cM +  i V_\cM$ is dense in $\cH$. 
Closed real subspaces $V \subeq \cH$ with these two properties 
are called {\it standard}. 
Every standard subspace $V$ determines by the polar decomposition 
of the closed operator $S_V$ defined on $V + i V$ by 
$S_V(x + i y) = x- iy$ a pair $(\Delta_V, J_V)$ of modular objects 
and, conversely, any such pair $(\Delta, J)$ determines a standard subspace 
as the fixed point space of $J\Delta^{1/2}$ (see Section~\ref{sec:3}).
We refer to \cite{Lo08} for an excellent survey on this correspondence. 
In QFT, standard subspaces provide 
the basis for the technique of modular localization, developed 
by Brunetti, Guido and Longo in \cite{BGL02}. For some applications 
we refer to \cite{Sch97, MSY06, Sch06, LW11, Ta12, LL14, Mo17}. 

From the perspective of antiunitary representations,
standard subspaces $V$ with modular objects $(\Delta, J)$ 
are in one-to-one correspondence with antiunitary representations 
\begin{equation}
  \label{eq:antiuni-corres}
U \: \R^\times \to \AU(\cH)\quad \mbox{  by  } \quad 
U_{-1} = J \quad \mbox{ and } \quad U_{e^t} = \Delta^{-it/2\pi}
\end{equation}
(Proposition~\ref{prop:3.2}). 
Accordingly, antiunitary representations $(U,\cH)$  
of the affine group $\Aff(\R) \cong \R \rtimes \R^\times$ 
correspond to one-parameter families of standard subspaces 
$(V_x)_{x \in \R}$, where $V_x$ corresponds to the affine stabilizer group of~$x$. 
Borchers' key insight was that the positive energy condition 
on the representation of the translation group is intimately related to 
inclusions of these subspaces. More precisely, 
$U_{(t,1)}= e^{itP}$ satisfies $P \geq 0$ if and only 
if $U_{(t,1)} V_0 \subeq V_0$ holds for all $t \geq 0$ 
(\S \ref{subsec:3.3}). This leads to Borchers 
pairs $(V,U)$ of a standard subspace~$V$ and a unitary one-parameter group 
$(U_t)_{t \in \R}$, a concept that is equivalent to the so-called half-sided modular 
inclusions $V_1 \subeq V_2$ of pairs of standard subspaces, 
which was condensed from the corresponding concept of a 
half-sided modular inclusion of von Neumann algebras 
(\S\S\ref{subsec:3.3},\ref{subsec:4.2}). 

The main objective of this article is to describe 
certain structures arising in QFT, such as nets of von Neumann algebras 
and standard subspaces, from the perspective of antiunitary group 
representations. Since any standard subspace $V$ 
corresponds to a representation of $\R^\times$ 
and inclusions of standard subspaces correspond to antiunitary 
positive energy representations of $\Aff(\R)$, it is very likely that a better 
understanding of antiunitary representations and corresponding 
families of standard subspaces provides new insight into the geometric 
structures underlying QFT. This article is written 
from a mathematical perspective and we are rather brief on the 
concrete physical aspects mentioned in \S \ref{subsec:5.2}. 
We tried to describe the mathematical side of the theory 
as clearly as  possible to make it easier for mathematicians 
to understand the relevant aspects without going to much into physics. 
For more details of the physical side, we recommend 
\cite{BDFS00, BGL02, Lo08, LL14}. 
In particular, the programs outlined by Borchers and Wiesbrock, 
see f.i.,  \cite{Bo97, Bo00}, \cite{Wi93c}, leave much potential 
for an analysis from the representation theoretic perspective.

The structure of this paper is as follows. 
In Section~\ref{sec:2} we discuss antiunitary representations of group 
pairs $(G, G_1)$ and criteria for a unitary representation 
of $G_1$ to extend to an antiunitary representation of $G$. 
An interesting simplifying feature is that, whenever antiunitary 
extensions exist, they are unique up to equivalence 
(Theorem~\ref{thm:equiv}). We show that 
irreducible unitary representations of $G_1$ fall into three types 
(real, complex and quaternionic) with respect to their extendability 
behavior to antiunitary representations of $G$. 
We also take a closer look at antiunitary representations 
of one-dimensional Lie groups (\S \ref{subsec:one-par}). 
Here $\R^\times$ plays a central role because 
its antiunitary representations encode modular objects $(\Delta, J)$ 
as in \eqref{eq:antiuni-corres}. 
We conclude Section~\ref{sec:2} with a discussion of antiunitary 
representations of the affine group $\Aff(\R)$,  the projective 
group $\PGL_2(\R)$ and the $3$-dimensional Heisenberg group $\Heis(\R^2)$. 

Section~\ref{sec:3} is devoted to various aspects of standard subspaces 
as a geometric counterpart of antiunitary representations of~$\R^\times$. 
In particular, we discuss how the embedding $V \subeq \cH$ can be 
obtained from the orthogonal one-parameter group 
$\Delta^{it}\res_V$ on $V$ (\S \ref{subsec:orthog}),  
and in \S \ref{subsec:3.3} we discuss half-sided modular inclusions 
of standard subspaces and how they are related to 
antiunitary representations of $\Aff(\R)$, $P(2)_+$ and $\PGL_2(\R)$. 

In Section~\ref{sec:4} we first recall some of the key features of
Tomita--Takesaki Theory. \S \ref{subsec:4.2} is of key importance
because it is devoted to the translation between pairs $(\cM,\Omega)$ 
of von Neumann algebras with cyclic separating vectors 
and standard subspaces~$V$. We have already seen how to obtain a standard 
subspace $V_\cM = \oline{\cM_h\Omega}$ from $(\cM,\Omega)$. 
Conversely, one can use Second Quantization (see Section~\ref{sec:6} for details) 
to associate to each standard subspace $V \subeq \cH$ 
pairs $(\cR_\pm(V),\Omega)$, where $\cR_\pm(V)$ is a von Neumann algebra 
on the (bosonic/fermionic) Fock space $\cF_\pm(\cH)$. 
This method has been invented and 
studied thoroughly by Araki and Woods in the 1960s and 1970s in the context of 
free bosonic quantum fields (\cite{Ar64, Ar99, AW63, AW68}); 
some of the corresponding fermionic results are more recent 
(cf.\ \cite{EO73}, \cite{BJL02}) and other statistics (anyons) 
are discussed in \cite{Sch97}. 

A central point  is that these correspondences permit to translate between results 
on configurations of standard subspaces and configurations of von Neumann algebras 
with common cyclic vectors. We explain this in detail for half-sided modular 
inclusions and Borchers pairs (\S\S\ref{subsec:4.2} and \ref{subsec:4.3}) 
but we expect it to go much deeper. Keeping in mind that 
standard subspaces are in one-to-one correspondence with antiunitary 
representations of $\R^\times$ and half-sided modular inclusions with 
antiunitary positive energy representations of $\Aff(\R)$, we expect that many interesting 
results on von Neumann algebras can be obtained from a better understanding 
of antiunitary representations of Lie group pairs $(G, G_1)$ and 
configurations of homomorphisms $\gamma \: (\R^\times, \R^\times_+) \to (G,G_1)$. 
The construction of free fields by second quantization associates to 
an antiunitary representation $(U,\cH)$ of $G$ 
on the one-particle spaces $\cH$, resp., to the corresponding 
standard subspaces $V_\gamma$, a net of Neumann algebras on Fock space. 
However, there is also a converse aspect which is probably  more important, 
namely that the passage from pairs 
$(\cM,\Omega)$ to the standard subspaces $V_\cM$ is not restricted to free fields 
and can be used to attach geometric structure to nets of von Neumann algebras, 
all encoded in the subgroup of $\AU(\cH)$ generated by all operators 
$\Delta_\cM^{it}$ and $J_\cM$. 
To substantiate this remark, we discuss in Section~\ref{sec:5} 
several aspects of 
nets of standard subspaces and von Neumann algebras as they arise in QFT. 
In particular, we consider nets of standard subspaces 
$(V_\ell)_{\ell \in L}$ arising from antiunitary representations $(U,\cH)$, which 
leads to the covariance relation $U_g V_\ell  = V_{g.\ell}$ for $g \in G_1$, and one expects 
geometric information 
to be encoded in the $G$-action on the index set $L$. 
A~common feature of the natural examples is that $L$ has a fibration 
over a symmetric space that corresponds to the projection 
$(\Delta_\ell, J_\ell) \mapsto J_\ell$, forgetting the modular operator. 
For details we refer to the discussion of several examples in 
Section~\ref{sec:5}. Typical index sets $L$ arise as 
conjugation orbits 
$\{ \gamma^g, (\gamma^\vee)^g \: g \in G \} \subeq \Hom(\R^\times, G)$, 
where $\gamma^g(t) = g\gamma(t)g^{-1}$ and $\gamma^\vee(t) = \gamma(t^{-1})$. 
In this picture, the above 
projection simply corresponds to the evaluation map 
$\ev_{-1} \: \Hom(\R^\times, G) \to \Inv(G)$ and the set $\Inv(G)$ of involutions 
of $G$ is a symmetric space (\cite{Lo69}; Appendix~\ref{app:a.2}). 
In many concrete situations, the centralizer of $\gamma(-1)$ in $G$ 
coincides with the centralizer of the whole subgroup $\gamma(\R^\times)$, so that 
the conjugacy class $C_\gamma = \{\gamma^g \: g\in G\}$ can be identified 
with the conjugacy class $C_{\gamma(-1)}$ of the involution $\gamma(-1)$, and 
this manifold is a symmetric space. We are therefore led to 
index sets which are ordered symmetric spaces, and these objects have been 
studied in detail in the 90s. We refer to the monograph \cite{HO96} 
for a detailed exposition of their structure theory. 

Section~\ref{sec:6} presents the second quantization process from 
standard subspaces $V \subeq \cH$ to pairs $(\cR^\pm(V), \Omega)$ 
in a uniform way, stressing in particular the similarity between the 
bosonic and the fermionic case. 

In the final Section~\ref{sec:7} we briefly describe some 
perspectives and open problems. 
Antiunitary representations occur naturally for interesting classes of 
groups such as the Virasoro group, 
conformal and affine groups related to euclidean Jordan algebras,
and automorphism groups of bounded symmetric domains. For detailed 
results we refer to the forthcoming paper \cite{NO17}. 
In \S\ref{subsec:7.4} we also explain how second quantization 
leads to interesting dual pairs in the Heisenberg 
group $\Heis(\cH)$: Any standard subspace 
$V\subeq \cH$ satisfying the factoriality condition $V \cap V' = \{0\}$, 
where $V'$ is the symplectic orthogonal space,  
leads by restriction of the irreducible Fock representation 
of $\Heis(\cH)$ to a factor representation 
of the subgroup $\Heis(V)$, which forms a dual pair with 
$\Heis(V')$ in $\Heis(\cH)$ (both subgroups are their mutual centralizers). 
So far, such dual pairs have not been exploited systematically 
from the perspective of unitary representations of infinite dimensional 
Lie groups. 

Some basic auxiliary lemmas and definitions have been 
collected in the appendix. \\

{\bf Notation and conventions:} 
As customary in physics, the 
scalar product $\la \cdot,\cdot\ra$ on a complex Hilbert space $\cH$ 
is linear in the second argument. \\
$\lbr S \rbr$ denotes the closed subspace 
of a Hilbert space $\cH$ generated by the subset $S$. \\
$\{a,b\} := ab + ba$ is the anti-commutator of two elements 
of an associative algebra. \\ 
For the cyclic group of order $n$ we write $\Z_n = \Z/n\Z$.\\ 

For $\bx, \by \in \R^{d-1}$, we 
write $\bx \by = \sum_{j = 1}^{d-1} x_j y_j$ for the scalar product and, 
for $x = (x_0, \bx) \in \R^{d}$, we write
$[x,y] = x_0 y_0 - \bx \by$ for the Lorentzian scalar product on the 
$d$-dimensional Minkowski space $\R^{1,d-1}\cong \R^d$. 
The light cone in Minkowski space 
is denoted 
\[ V_+ = \{ x \in \R^{1,d-1} \: x_0 > 0, [x,x] > 0\}.\] 

Here is our notation for some of the groups arising in physics: 
\begin{itemize}
\item the {\it Poincar\'e group}
$P(d) \cong \R^{1,d-1} \rtimes \OO_{1,d-1}(\R)$ of affine isometries 
of $\R^{1,d-1}$, 
\item $P(d)_+ = \R^{1,d-1} \rtimes \SO_{1,d-1}(\R)$ 
 is the subgroup of orientation preserving maps, and 
\item $P(d)^\uparrow = \R^{1,d-1} \rtimes \OO_{1,d-1}(\R)^\uparrow$ 
with $\OO_{1,d-1}(\R)^\uparrow = \{ g \in \OO_{1,d-1}(\R) \: gV_+ = V_+\}$ 
the subgroup preserving the causal structure. 
\item The corresponding {\it conformal group} 
is $\OO_{2,d}(\R)$, acting on the conformal compactification 
$\bS^1 \times \bS^{d-1}$ of $M^d$ with the kernel $\{\pm \1\}$ 
(see \cite[\S17.4]{HN12}). 
\end{itemize}
If not otherwise states, all Lie groups in this paper are finite dimensional.

\section{Antiunitary representations} 
\mlabel{sec:2}

In this section we discuss antiunitary representations of group 
pairs $(G, G_1)$ and criteria for a unitary representation 
of $G_1$ to extend to an antiunitary representation of $G$. 
We start in \S \ref{subsec:2.1} with some general remarks on 
group pairs $(G,G_1)$ and how to classify twists in this context. 
We also take a closer look at antiunitary representations 
of one-dimensional Lie groups in \S \ref{subsec:one-par} 
and discuss antiunitary 
representations of the affine group $\Aff(\R)$, the projective 
group $\PGL_2(\R)$ and the $3$-dimensional Heisenberg group 
in \S \ref{subsec:2.4}.

\begin{definition} An {\it antiunitary representation} $(U,\cH)$ of a 
group pair $(G, G_1)$, where $G_1 \subeq G$ is a subgroup of index $2$, 
is a homomorphism $U$ of $G$ into the group $\AU(\cH)$ 
of unitary or antiunitary operators on a 
complex Hilbert space $\cH$ for which $G_1 = U_G^{-1}(\U(\cH))$, i.e., 
$G_1$ is represented by unitary operators and the coset 
$G\setminus G_1$ by antiunitary operators. 

If $G$ is a Lie group, then $(G,G_1)$ is called a {\it Lie group pair}. 

If $G$ is a topological group with two connected components, then 
we obtain a canonical group pair by $G_1 := G_0$ (the identity component). 
In this case an {\it antiunitary representation of $G$} is a continuous 
homomorphism 
$U \: G \to \AU(\cH)$ mapping $G\setminus G_0$ into antiunitary operators. 
\end{definition}

We start this section with a discussion of the natural class of 
group pairs that will show up in the context of antiunitary representations.

\subsection{Involutive group pairs} 
\mlabel{subsec:2.1} 

\begin{definition} An {\it involutive group pair} is a 
pair $(G, G_1)$ of groups, where $G_1 \subeq G$ is a subgroup of index $2$ 
and there exists an element $g \in G \setminus G_1$ with 
$g^2 \in Z(G_1)$. Then $\tau(g_1) := gg_1 g^{-1}$ defines an involutive 
automorphism of $G_1$. 
\end{definition}

In most examples that we encounter below $G$ is a Lie 
group with two connected components and $G_1$ is its identity component. 

\begin{remark} (a) If $g^2 \in Z(G_1)$, then other elements 
$gh\in g G_1$ need not have central squares. 
From $(gh)^2 = ghgh = g^2 \tau(h)h$ it follows that 
$(gh)^2$ is central if and only if $\tau(h)h \in Z(G_1)$, which is 
in particular the case if $\tau(h) = h^{-1}$. 

(b) If $G$ is a Lie group, then any conjugacy class $C_g$ of $g \in G \setminus G_1$ 
with $g^2 \in Z(G)$ carries a natural symmetric space 
structure (Appendix~\ref{app:a.2}). In fact, the stabilizer 
of $g$ in $G_1$ is $G_1^\tau$, so that we obtain a diffeomorphism 
\[ G_1/G_1^\tau \to C_g, \qquad h G_1^\tau \mapsto h g h^{-1} = h \tau(h)^{-1} g. \] 
\end{remark}

\begin{exs} \mlabel{exs:conj} 
(a) Let $\cH$ be a complex Hilbert space 
and $(G, G_1) := (\AU(\cH), \U(\cH))$. 
An anti\-unitary operator $J \in \AU(\cH)$ is called a {\it conjugation} 
if $J^2 = \1$ and an {\it anticonjugation} if $J^2 = - \1$. 
Conjugations always exist and define a {\it real structure} on $\cH$ in 
the sense that $\cH^J= \Fix(J) := \ker(J-\1)$ is a real Hilbert space whose complexification 
is~$\cH$.\begin{footnote}{For the existence, fix an orthonormal 
basis $(e_j)_{j \in I}$ of $\cH$ and defined $J$ to be antilinear with 
$Je_j = e_j$ for every $j \in I$.} \end{footnote}
Anticonjugations define on $\cH$ a {\it quaternionic structure}, 
hence do not exist if $\cH$ is of finite odd dimension. 

Any (anti-)conjugation $J$ on $\cH$ is contained in $G\setminus G_1$ and 
satisfies $J^2 \in \{\pm \1\} \subeq Z(\U(\cH))$. 

(b) If $G_1$ is a group and $\tau \in \Aut(G_1)$ is an involutive automorphism, 
then \break $G := G_1 \rtimes \{\1,\tau\}$ defines an involutive group pair.
\end{exs}

\begin{ex} (A non-involutive group pair) 
Let $\sigma \: C_4 \to \Aut(\C)$ denote the natural action of 
the subgroup $C_4 = \{ \pm 1, \pm i\}\subeq \T$ by 
multiplication and form the semidirect product group 
$G := \C \rtimes_\sigma C_4$. Then $G_1 := \C \rtimes_\sigma \{\pm 1\}$ 
is a subgroup of index $2$ but no element $g \in G\setminus G_1$ satisfies 
$g^2 \in Z(G_1)$ because $g^2$ acts on $\C$ as $-\id_\C$. 
\end{ex}

\begin{remark} (Classification of involutive group pairs) 
\mlabel{rem:2.1} 
(a) Suppose we are given a group $G$ and an involutive automorphism 
$\tau$ of $G$. We want to classify all group extensions 
\[ \1 \to G \to G^\sharp \to \Z_2 \to \1, \] 
where the corresponding involution in the group $\Out(G)$ of outer automorphisms of $G$ 
is represented by~$\tau$. 
In view of \cite[Thm.~18.1.13]{HN12}, the equivalence classes of these extensions 
are parametrized by the cohomology group $H^2_\tau(\Z_2,Z(G))$, 
where $\oline 1$ acts on $Z(G)$ by $\tau\res_{Z(G)}$. 
As any cocycle 
$f \: \Z_2 \times \Z_2 \to Z(G)$ normalized by 
$f(\oline 0,g) = f(g,\oline 0) = e$ is determined by the element 
$z := f(\oline 1,\oline 1) \in Z(G)$ because all other values vanish, 
the group structure on the corresponding 
extension is given by an element $\hat\tau \in G^\sharp \setminus G$ satisfying 
\[ \hat\tau^2 = z \quad \mbox{ and } \quad \hat\tau g \hat\tau^{-1} = \tau(g)
\quad \mbox{ for }\quad g\in G.\]
This description shows in particular that $\tau(z) = z$, and a closer 
inspection of the cohomology groups yields 
\begin{equation}
  \label{eq:h2-form}
 H^2_\tau(\Z_2, Z(G)) \cong Z(G)^\tau/Z(G)_\tau, 
\quad \mbox{ where } \quad Z(G)_\tau := \{ \tau(z)z \: z \in Z(G)\}
\end{equation}
(\cite[Ex.~18.3.5]{HN12}). 

(b) For $\tau\res_{Z(G)} = \id_{Z(G)}$ we have 
\[ Z(G)_\tau = \{ z^2 \: z \in Z(G)\} \quad \mbox{ and }\quad 
H^2_\tau(\Z_2, Z(G)) \cong Z(G)/Z(G)_\tau.\] 
  
(c) For $\tau\res_{Z(G)} = -\id_{Z(G)}$, we have 
\[ H^2_\tau(\Z_2, Z(G)) \cong Z(G)^\tau 
= \{ z \in Z(G) \: z^2 = e\},\] 
the subgroup of central involutions.
\end{remark}

\begin{remark}
Although by \eqref{eq:h2-form} 
the cohomology groups $H^2_\tau(\Z_2,Z(G))$ are elementary abelian 
two groups, one cannot expect any bound on the order of an element 
$g \in G^\sharp \setminus G$. In the cyclic group 
$G^\sharp = \Z_{2^n}$ with $G \cong \Z_{2^{n-1}}$, any 
element of $G^\sharp \setminus G$ is of order $2^n$.
\end{remark}

\begin{ex} \mlabel{app:b.1} (a) For $G = \R$, Remark~\ref{rem:2.1} implies that 
$H^2_\tau(\Z_2,\R) = \{0\}$ for any involutive automorphism $\tau$. 
This implies that $G^\sharp \cong \R \rtimes_\tau \Z_2$. 

(b) For $G = \T$, the cohomology is trivial for $\tau = \id_\T$, 
but for $\tau(z) = z^{-1}$ the group  
\[ H^2_{\tau}(\Z_2, \T) \cong  \{ z \in \T \: z^2 = 1\} = \{ \pm 1 \}\] 
is non-trivial. A concrete model for the non-trivial extension with 
$\hat\tau^2 = -1$ is given by the subgroup 
\[  \Pin_2(\R) = \exp(\R I) \cup J \exp(\R I) \subeq \bH^\times, \] 
where $I$ and $J$ are the two generators of the skew-field $\bH$ of 
quaternions satisfying $I^2  = J^2 = - \1$ and $IJ = - JI$ 
(\cite[Ex.~B.3.24]{HN12}). 
This is a $1$-dimensional Lie group without a simply connected covering 
group (\cite[Ex.~18.2.4]{HN12})
\end{ex}

\begin{exs} \mlabel{ex:2.11}
Here are some concrete involutive group pairs $(G, G_1)$ that we shall 
be dealing with. 

(a) $G = \Aff(\R)\cong \R \rtimes \R^\times$ with $G_1 \cong \R \rtimes \R^\times_+$,  
the identity component. Here $r_x^2 = \1$ holds for the reflections 
$r_x = (2x,-1)$ in $x \in \R$.

(b) The automorphism group $G = \PGL_2(\R)$ of the real projective line 
$\bP_1(\R) \cong  \R \cup \{\infty\}$, where $G_1 = \PSL_2(\R)$ 
is the identity component and reflections in $\GL_2(\R)$ lead to orientation 
reversing involutions of $\bS^1$. 

(c) The {\it Poincar\'e group} $P(d) = \R^{1,d-1} \rtimes \OO_{1,d-1}(\R)$ 
of $d$-dimensional Minkowski space $\R^{1,d-1}$ contains the subgroup 
$P(d)_+ = \R^{1,d-1} \rtimes \SO_{1,d-1}(\R)$ of orientation preserving 
affine isometries. Then we obtain the involutive group pair $(G, G_1)$ with 
$G := P(d)_+$ and $G_1 := P(d)_+^\uparrow$. 
In the following the involution 
$R_{01} := \diag(-1,-1,1,\ldots, 1) \in G \setminus G_1$ plays an important role 
(cf.\ Lemma~\ref{lem:4.17}). 

(d) For a (bounded) symmetric domain $\cD \subeq \C^n$, the group $\Aut(\cD)$ of 
biholomorphic automorphisms is an index $2$-subgroup of the hermitian 
group $\AAut(\cD)$ of all bijections of $\cD$ that are either 
holomorphic or antiholomorphic. There always exist 
antiholomorphic involutions $\sigma$ in $\AAut(\cD)$ 
(see \cite{Ka97} for a classification covering even the infinite dimensional case). 
For any such involution $\sigma$, we obtain by 
$G_1 := \Aut(\cD)_0$ and $G := G_1 \rtimes \{\1,\sigma\}$ 
an involutive group pair (cf.~\cite{NO17} and \S\ref{subsec:7.3}). 
\end{exs}

\subsection{Extending unitary representations} 
\mlabel{subsec:2.2}

Suppose  that $G_1$ is an index two subgroup of the group $G$ 
and $(U,\cH)$ is a unitary representation of~$G_1$. 
In this subsection we discuss extensions of $U$ to antiunitary 
representations of $G$. In particular, we show that, 
in analogy to the classical case $G = G_1 \times \Z_2$, 
irreducible antiunitary representations fall into three types that 
we call real, complex and quaternionic, according to their commutant. 

We start with the following lemma on a situation 
where extensions always exist because the representation 
has been doubled in a suitable way. 

\begin{lemma} \mlabel{lem:extlem} {\rm(Extension Lemma)} 
Let $G_1 \subeq G$ be a subgroup of index two and 
$(U,\cH)$ be a unitary representation of $G_1$. 
Fix $r \in G \setminus G_1$ and consider the automorphism 
$\tau(g) := rgr^{-1}$ of $G_1$. Then the unitary 
representation $V := U \oplus U^*\circ \tau$ on $\cH \oplus \cH^*$ 
extends to an antiunitary 
representation of~$G$. 
\end{lemma}

\begin{proof} Let $\Phi \: \cH \to \cH^*, \Phi(v)(w) := \la v, w \ra$ denote the canonical 
antiunitary operator and note that $U_g^* \circ \Phi = \Phi \circ U_g$ for 
$g \in G_1$. We consider the antiunitary operator 
\[ J \: \cH \oplus \cH^* \to \cH \oplus \cH^*, \quad 
J(v,\lambda) := (\Phi^{-1}\lambda, \Phi U_{r^2} v). \] 
It satisfies 
\[ J^2(v,\lambda) 
= J (\Phi^{-1}\lambda, \Phi U_{r^2} v) 
= (U_{r^2} v, \Phi U_{r^2} \Phi^{-1}\lambda) 
= (U_{r^2} v, U^*_{r^2} \lambda) = V_{r^2}(v,\lambda),\] 
where we have used $\tau(r^2) = r^2$ for the last equality. 
This proves that $J^2 = V_{r^2}$. 
We now show that $J V_g J^{-1} = V_{\tau(g)}$ for $g \in G$: 
\begin{align*}
J V_g(v,\lambda) 
&= J (U_g v, U^*_{\tau(g)} \lambda)  
=  (\Phi^{-1} U^*_{\tau(g)} \lambda, \Phi U_{r^2} U_g v)  
=  (U_{\tau(g)} \Phi^{-1}\lambda, \Phi U_{\tau^2(g)} U_{r^2} v)\\  
&=  (U_{\tau(g)} \Phi^{-1}\lambda, U^*_{\tau^2(g)} \Phi U_{r^2} v) 
=  V_{\tau(g)} (\Phi^{-1}\lambda,\Phi U_{r^2} v) 
=  V_{\tau(g)} J(v,\lambda). 
\end{align*}
The relations $J^2 = V_{r^2}$ and $J V_g J^{-1} = V_{\tau(g)}$ 
now imply by direct calculation that the assignment 
$V_{gr} := V_g J$ for $g \in G_1$ defines an extension of 
$V$ to an antiunitary representation of $G$ 
(Lemma~\ref{lem:ext-homo}). 
\end{proof}

The following theorem implies 
that extensions of unitary representations of $G_1$ to antiunitary 
representations $(U,\cH)$ of $G$ are always unique up to isomorphism. 
It also describes the situation for irreducible representations. 
Note that the commutant 
\[ U_G' = \{ A \in B(\cH) \: (\forall g \in G)\, A U_g = U_g A \}\] 
is not a complex subalgebra of $B(\cH)$ because some $U_g$ are antilinear. 

\begin{theorem} \mlabel{thm:equiv}
Let $G_1 \subeq G$ be a subgroup of index two, 
$r \in G\setminus G_1$ and $\tau(g) := rgr^{-1}$ for $g \in G_1$. 
\begin{itemize}
\item[\rm(a)] For two antiunitary representation $(U^j, \cH_j)_{j =1,2}$, we then have 
\[ U^1 \cong U^2 \quad \Longleftrightarrow \quad U^1\res_{G_1} \cong U^2\res_{G_1}.\] 
\item[\rm(b)] For any antiunitary representation $(U, \cH)$ of $(G,G_1)$, 
the von Neumann algebra $U_{G_1}'$ is the 
complexification of the real algebra $U_G'$. 
\item[\rm(c)] An antiunitary representation $(U, \cH)$ of $(G,G_1)$ 
is irreducible if and only if its commutant $U_G'$ 
is  isomorphic to $\R$, $\C$ or $\H$. More specifically: 
\begin{itemize}
\item[\rm(i)] If $U_G' \cong \R$, then $U_{G_1}' \cong \C$ and $U\res_{G_1}$ is irreducible.\item[\rm(ii)] If $U_G' \cong \C$, then $U_{G_1}' \cong \C^2$ 
and $U\res_{G_1}$ is a direct sum of two inequivalent irreducible representations 
which do not extend to an antiunitary representation of~$G$. 
\item[\rm(iii)] If $U_G' \cong \H$, then $U_{G_1}' \cong M_2(\C)$ 
and $U\res_{G_1}$ is a direct sum of two equivalent irreducible representations 
which do not extend to an antiunitary representation of~$G$. 
\end{itemize}
\item[\rm(d)] 
For an irreducible unitary representation $(U,\cH)$ of $G_1$, 
either 
\begin{itemize}
\item[\rm(i)] $U$ extends to an antiunitary representations $\oline U$ of $G$,  
and then $\oline U$ is  irreducible with $\oline U_{G}' \cong \R$; or 
\item[\rm(ii)] $U$ does not extend to an antiunitary representation of $G$. 
Then $V := U \oplus U^*\circ \tau$ extends to an irreducible 
antiunitary representation of $G$ and $V_G' \cong \C$ if \break 
$U^* \circ \tau \not\cong U$ and 
$V_G' \cong \H$ if $U^* \circ \tau \cong U$. 
\end{itemize}
\end{itemize}
\end{theorem}

\begin{proof} (a)
  \begin{footnote}
{In the finite dimensional context, this was already known to E.~Wigner; see \cite[p.~344]{Wig59}.} 
\end{footnote}
Let $\Phi \: \cH_1 \to \cH_2$ be a unitary intertwining operator 
for the representations $U^j\res_{G_1}$. Pick $r \in G \setminus G_1$ and consider 
the antiunitary operators $J_j := U^j_r \in \AU(\cH_j)$. 
Then the unitary operator 
$U := J_1^{-1} \Phi^{-1} J_2 \Phi \in \U(\cH_1)$ commutes with $U^1_{G_1}$. 
The map $j_1(M) := J_1MJ_1^{-1}$ defines an antilinear automorphism of the von Neumann 
algebra $(U^1_{G_1})'$ satisfying 
\[ j_1(U) 
= \Phi^{-1} J_2 \Phi J_1^{-1} 
= \Phi^{-1} J_2^{-1} U^2_{r^2} \Phi J_1^{-1} 
= \Phi^{-1} J_2^{-1} \Phi U^1_{r^2}  J_1^{-1} 
= \Phi^{-1} J_2^{-1} \Phi J_1 = U^{-1}.\] 
Therefore Lemma~\ref{lem:app.1}(c) implies the existence 
of a unitary operator $V \in (U^1_{G_1})'$ with $V^2 = U^{-1}$ and $j_1(V) = V^{-1}$. 
With $\Psi := \Phi \circ V$, this leads to 
\[ \Psi^{-1} J_2 \Psi 
= V^{-1} \Phi^{-1} J_2 \Phi V 
= V^{-1} J_1 U  V 
= V^{-1} U^{-1} J_1 V 
= V J_1 V = VV^{-1} J_1 = J_1.\] 
We conclude that the antiunitary representations $U^1$ and $U^2$ are equivalent. 

(b) Let $J := U_r$. Then $U_{G_1}'$ is invariant under the antilinear 
automorphism $j(M) := JMJ^{-1}$. Since $J^2$ commutes with $U_{G_1}'$, it is involutive. 
As $U_G'$ is the set of fixed points of $j$, it is 
a real form of the complex vector space $U_{G_1}'$. This implies the assertion. 

(c) The closed complex subspaces invariant under $U_G$ are precisely the 
closed real subspaces of the underlying real space $\cH^\R$ invariant under 
the group $\T \cdot U_G$. Therefore $U_G$ is irreducible if and only if 
the real representation of $\T \cdot U_G$ on $\cH^\R$ 
is irreducible, which is equivalent to its commutant 
being isomorphic to $\R, \C$ or $\H$ (\cite[Thm.~1]{StVa02}). 
Next we observe that the real linear commutant of $\T\1$ consists of the complex linear 
operators. Therefore the real linear commutant of $\T\cdot U_G$ equals the complex 
linear commutant $U_G'$. Now (b) implies that $U_G' \cong \R,\C,\H$ leads to 
$U_{G_1}' \cong \C, \C^2,M_2(\C)$, respectively. 

In the first case $U\res_{G_1}$ is irreducible. 
In the second case $\cH \cong \cH_+ \oplus \cH_-$, where $\cH_\pm$ 
are $G_1$-invariant subspaces on which the $G_1$-representations are 
irreducible and non-equivalent. As $U_r$ permutes the $G_1$-isotypical 
subspaces, $U_r \cH_\pm = \cH_\mp$. 
For the representations $U^\pm$ of $G_1$ on $\cH_\pm$, this implies that 
$U^- \cong (U^+)^*\circ \tau$. If $U^+$ or $U^-$ extends to an 
antiunitary representation of $G$, then 
$U\res_{G_1}$ has an extension to a reducible representation 
$\tilde U$ of $G$. As $U$ is irreducible, this contradicts (a).  
In the third case we have a similar decomposition with 
$U^+ \cong U^-$. Again, the irreducibility of $U$, combined with (a),  
implies that $U^\pm$ do not extend to~$G$. 

(d) 
If $U$ extends to an antiunitary representation $\oline U$ 
of $G$ on the same space, then this representation is obviously 
irreducible and (c) implies that $\oline U_G' \cong \R$. 
If such an  extension does not exist, then 
the Extension Lemma~\ref{lem:extlem} 
provides an extension of 
$V := U \oplus U^* \circ \tau$ to an antiunitary representation of~$G$ by 
\[ V_r := J, \quad \mbox{ where } \quad 
J(v,\lambda) := (\Phi^{-1}\lambda, U^*_{r^2} \Phi v).\] 
If $U^* \circ \tau \not\cong U$, then 
$V_{G_1}' \cong \C^2$, and if 
$U^* \circ \tau \cong U$, then $V_{G_1}' \cong M_2(\C)$. 

In the first case the algebra 
$V_{G_1}' \cong \C^2$ acts by diagonal operators 
$T_{(a,b)}(v,\lambda) := (av, b \lambda)$. Such an operator commutes with 
$J$ if and only if $b = \oline a$. Therefore $V_G' \cong \C$, and thus 
the representation $V$ of $G$ is irreducible. 

In the second case, $V_G'$ is a real form of $V_{G_1}' \cong M_2(\C)$. 
We show that the representation $(V,\cH \oplus \cH^*)$ of $G$ is irreducible. 
If this is not the case, there exists a proper $G$-invariant subspace 
$\cK \subeq \cH \oplus \cH^*$. As $V\res_{G_1} \cong U \oplus U$, 
the $G_1$-representation on $\cK$ must be irreducible 
and equivalent to~$U$. This contradicts the non-extendability of $U$ 
to an antiunitary representation of $G$. Therefore $V$ is irreducible  
and $V_G'$ is isomorphic to $\H$. 
\end{proof}

\begin{definition} \mlabel{def:3types} 
(Three types of irreducible representations
\begin{footnote}{In a special context, this 
classification by three types can already be found in 
Wigner's book \cite[\S 26, p.~343]{Wig59}.}\end{footnote}
) 
We keep the notation of the preceding theorem. 
If $(U,\cH)$ is an irreducible unitary representation 
of $G_1$ with $U \cong U^* \circ \tau$, then there exists a 
$\Phi \in \AU(\cH)$ with 
$\Phi U_g \Phi^{-1} = U_{rgr^{-1}}$ for $g \in G_1$. 
By Schur's Lemma, such an operator $\Phi$ is unique up to a scalar 
factor in $\T$, so that $\Phi^2$ does not depend on the concrete choice of 
$\Phi$. Therefore an antiunitary extension to $G$ exists if and only if 
$\Phi^2 = U_{r^2}$. Then we call $(U,\cH)$ of {\it real type 
(with respect to $\tau$)}. 
If this is not the case, but $U \cong U^* \circ \tau$, 
then $(U,\cH)$ is said to be of {\it quaternionic type 
(with respect to $\tau$)}, 
 and otherwise we say that it is of {\it complex type (with respect to $\tau$)}. 
This terminology matches the type of the commutant of the corresponding 
irreducible antiunitary representation of~$G$. 
\end{definition}

\begin{ex} (a) If $\cH = \C$ is one-dimensional, then $\AU(\cH) =\T \{\1,J\} 
\cong \OO_2(\R)$ for any conjugation $J$. 
We conclude in particular that all antiunitary operators are involutions. 

(b) If $\cH = \C^2$ is two-dimensional, we can already see all types of situations 
for groups generated by a single antiunitary operator, i.e., for antiunitary 
representations of the pair $(G,G_1) = (\Z, 2\Z)$. 

Let $J \in \AU(\cH)$ be antiunitary and $J^2 \in \U(\cH)$ be its square. 
If $J^2 = \1$, then $J$ is a conjugation, so that there are proper $J$-invariant 
subspaces. 
If $J^2 = - \1$, then $J$ is an anticonjugation defining a quaternionic structure 
on $\C^2 \cong \H$. In particular, the representation is irreducible with 
$U_G' \cong \H$ and $U_{G_1}' = B(\cH) \cong M_2(\C)$. 

Assume that $J^4 \not=\1$. Then $J^2$ is not an involution, so that 
it has an eigenvalue $\lambda \not=\pm 1$. If
$\cH^\lambda$ is the corresponding eigenspace, then 
$J \cH^\lambda = \cH^{\oline\lambda}$, so that $\Spec(J^2) = \{\lambda, \oline\lambda\}$. 
Choosing an orthonormal basis $e_1, e_2$ such that 
$e_1 \in \cH^\lambda$ and $e_2 := J e_1 \in \cH^{\oline\lambda}$, 
we obtain $Je_2 = J^2 e_1 = \lambda e_1$, so that $J$ is determined up to 
equivalence. The corresponding representation on $\C^2$ is irreducible 
with  $U_{G_1}' \cong \C^2$ and $U_G' \cong \C$ 
(Theorem~\ref{thm:equiv}(c)). 
\end{ex}

\begin{ex} (a) For $G = G_1 \times \Z_2$, the concepts of 
real/complex/quaternionic type coincides with the classical definition 
for $G_1$, as the characterization in Theorem~\ref{thm:equiv} shows. 

(b) For $G = G_1 \rtimes \{\1,\tau\}$ and $\tau^2 = \id_{G_1}$, 
the extendability of an irreducible unitary representation 
$(U,\cH)$ of $G_1$ is equivalent to the existence of a conjugation $J \in \AU(\cH)$ 
satisfying $JU_g J = U_{\tau(g)}$ for $g \in G_1$. 

If $U \cong U^* \circ \tau$, then a $J \in \AU(\cH)$ 
satisfying $JU_g J = U_{\tau(g)}$ for $g \in G_1$ exists 
and $J^2 \in U_{G_1}' = \C \1$, together with $J J^2 J = J^2$ imply 
$J^2 \in \{\pm \1\}$. 
Accordingly,  $U$ is of real, resp., quaternionic type if 
$J^2 = \1$, resp., $J^2 = - \1$. 
\end{ex}

\begin{ex} (a) For $G = \Z_2$ and $G_1 = \{e\}$, 
Theorem~\ref{thm:equiv})(a) reproduces the fact 
 that all conjugations on $\cH$ are conjugate under $\U(\cH)$. 
  
(b) For $G = \Z_4 = \Z/4\Z$ and $G_1 = \{\oline 0, \oline 2\}$, the case of 
antiunitary representations 
with $U_{\oline 2} = -\1$ likewise implies that all anticonjugations 
are conjugate under $\U(\cH)$. 

(c) The irreducible unitary representation of 
$G = \SU_2(\C)$ on $\C^2 \cong \bH$ (by left multiplication)  
is of quaternionic type. The complex structure on $\bH$ 
is defined by the right multiplication with~$I$. 
Then $\Phi(a) = a J$ defines a $G$-equivariant anticonjugation on $\C^2$. 
Therefore the representation is of quaternionic type. 

(d) For any compact connected Lie group $G_1$, the irreducible unitary 
representations $(U_\lambda, \cH_\lambda)$ 
are classified in terms of their highest weights 
$\lambda$ with respect to a maximal torus $T \subeq G_1$, resp., 
by the orbits $\cW\lambda$ under the Weyl group~$\cW$.  
As $-\cW\lambda$ is the Weyl group orbit of the dual representation, 
$U_\lambda$ is self-dual if and only if $-\lambda \in \cW\lambda$ 
(\cite[Prop.~VI.4.1]{BtD85}). 
It is of real, resp., quaternionic type if and only if an invariant 
symmetric, resp., skew-symmetric bilinear form exists 
(\cite[Prop.~II.6.4]{BtD85}), 
and this can also be read from the highest weight 
(\cite[Prop.~VI.4.6]{BtD85}). 

Further, for any automorphism $\sigma \in \Aut(G_1)$, there exists an 
inner automorphism $\sigma'$ such that $\tau := \sigma\sigma'$ preserves $T$. 
Then $\lambda^\tau := \lambda \circ \tau\res_T$ is an extremal weight 
of $U_\lambda \circ \tau \cong U_\lambda \circ \sigma$, so that 
$U_\lambda^* \circ \tau \cong U_\lambda$ if and only if 
$-\lambda^\tau \in \cW\lambda$. 
\end{ex}

The following lemma shows that, if only $G_1$ and an involutive 
automorphism $\tau$ of $G_1$ are given, 
then there always exists an extension to a group of the type 
$G_1 \rtimes_\alpha \Z_4$, where $\alpha_{\oline 1} = \tau$. 
This issue is already discussed in Wigner's book  
\cite[\S 26, p.~329]{Wig59}, where $J^2 = \pm \1$ is related to 
spin being integral or half-integral. 

\begin{lemma} \mlabel{lem:2.9}
Let $(U,\cH)$ be a unitary representation of the group $G_1$ and 
$\tau \in \Aut(G_1)$ be an involution. 
If $U \circ \tau \cong U^*$, then there exists a 
$J \in \AU(\cH)$ with $J^4 = \1$ and $J U_g J^{-1} = U_{\tau(g)}$ for $g \in G_1$. 
\end{lemma}

\begin{proof} From $U \circ \tau \cong U^*$ we obtain a $J \in \AU(\cH)$ 
with $J U_g J^{-1} = U_{\tau(g)}$ for $g \in G_1$. 
As $\tau^2 = \id_{G_1}$, the unitary operator $J^2$ commutes with $U_{G_1}$. 
We therefore have a $G_1$-invariant orthogonal decomposition 
$\cH =\cH_+  \oplus \cH_-$, where $\cH_- = \ker(J^2 + 1)$ and 
$\cH_+ = \cH_-^\bot$. 
Since both subspaces are invariant under $G_1$ and $J$, 
we may w.l.o.g.\ assume that $\cH_-= \{0\}$ 
and show that there exists a conjugation commuting with $G_1$. 

Conjugating with $J$ defines an antilinear 
automorphism of the von Neumann algebra $\cM := U_{G_1}'$ 
fixing the unitary element $J^2$. Therefore Lemma~\ref{lem:app.1} implies 
the existence of a unitary $A \in U_{G_1}'$ with $JAJ = A$ and $A^2 = J^2$. 
Replacing $J$ by $\tilde J := A^{-1}J$, we obtain $\tilde J^2 = \1$.   
\end{proof}

\begin{lemma} \mlabel{lem:ext-abelian}
Let $G_1$ be an abelian group, $\tau(g) = g^{-1}$ 
and $G := G_1 \rtimes \{\1,\tau\}$. Then every unitary representation 
of $G_1$ extends to an antiunitary representation of $G$. 
\end{lemma}

\begin{proof} We consider $G_1$ as a discrete 
group, so that any unitary representation $(U,\cH)$ of $G_1$ 
is a direct sum of cyclic representations of the form 
$(V,L^2(\hat{G_1},\mu))$, where 
$(V_g f)(\chi) = \chi(g)f(\chi)$. 
Then $Jf := \oline f$ defines a conjugation on $L^2(\hat A,\mu)$ 
with $J V_g J = V_{g^{-1}}$, so that we obtain an extension of $V$ to an 
antiunitary representation of $G$. 
\end{proof}

\begin{lemma} \mlabel{lem:2.19} 
Suppose that $G \cong G_1 \rtimes \{\id,\tau\}$, 
where $\tau \in \Aut(G)$ is an involution. 
\begin{itemize}
\item[\rm(i)] If $(U,\cH)$ is an irreducible antiunitary 
representations of $G$ and $x \in \g^\tau$ satisfies 
$-i\dd U(x)\geq 0$, then $\dd U(x) = 0$. 
\item[\rm(ii)] If $(U,\cH)$ is an irreducible unitary 
representations of $G_1$ and $x \in \g^\tau$ satisfies 
$-i\dd U(x)\geq 0$ and $\dd U(x) \not=0$, 
then $U^* \circ \tau \not\cong U$, i.e., $U$ is of complex type 
with respect to $\tau$. 
\end{itemize}
\end{lemma}

\begin{proof} (i) The conjugation $U_\tau$ on $\cH$ satisfies 
$U_\tau i\dd U(x) U_\tau = - i\dd U(\tau x) = - i \dd U(x),$ 
so that the positivity assumption implies $\dd U(x) = 0$. 

(ii) From (i) it follows that $U$ does not extend to an antiunitary 
representation of $G$. 
By Theorem~\ref{thm:equiv}(d)(ii), 
$V := U \oplus U^* \circ \tau$ extends to an irreducible 
antiunitary representation of $G$. 
If $V_G' \cong \H$, then $U^* \circ \tau \cong U$ 
implies $-i \dd V(x) \geq 0$, so that $\dd V(x) = 0$ by (i), 
and this contradicts $\dd U(x) \not=0$. We conclude that 
$V_G' \cong \C$ and $U^* \circ \tau \not\cong U$.
\end{proof}

\begin{remark} \mlabel{rem:doublecone}
In \cite{OM16} the authors study a concept of 
a ``Wigner elementary relativistic system'' which 
is defined as a faithful irreducible orthogonal representation  
$(U,\cK)$ of the proper orthochronous Poincar\'e group 
$G := P(4)_+^\uparrow$ on a real Hilbert space $\cK$. 
Writing $(\tilde P_j)_{0 \leq j \leq 3}$ for the skew-adjoint generators 
of the unitary representation 
of the translations groups $U_{t e_j}= e^{t \tilde P_j}$, 
the {\it mass squared operator} is defined as 
\[ M^2 := -\tilde P_0^2 + \sum_{j = 1}^3 \tilde P_j^2.\] 
One of the main results in \cite{OM16} is that if 
$M^2 \geq 0$, then $\cK$ carries a complex structure $I$ commuting 
with the image of $U$ (\cite[Thm.~4.3, Thm.~5.11]{OM16}). 

This result can be obtained quite directly in our context. 
We consider the complexification $(U_\C,\cK_\C)$ of the representation 
on $\cK$ by extending all operators $U_g$ to unitary operators on $\cK_\C$. 
Then the operators $P_j := -i \tilde P_j$ are selfadjoint with 
\[ M^2 = P_0^2 - \sum_{j = 1}^3 P_j^2 \geq 0.\] 
Since $(U,\cK)$ is irreducible, its commutant 
is isomorphic to $\R, \C$ or $\H$ (\cite[Thm.~1]{StVa02}). 
We claim that it is isomorphic to $\C$. 
If this is not the case, then 
$U_\C$ is either irreducible (if the commutant is $\R$) 
or a direct sum of two copies of 
the same irreducible unitary representation $(\hat U,\hat\cH)$ of $G$ 
(if the commutant is $\bH$). As $M^2 \geq 0$, the spectrum 
of the translation group is contained in the set 
\[ D := \{ (x_0, \bx) \in \R^{1,3} \: x_0^2 \geq \bx^2 \}.\] 
The decomposition 
$D = D_+ \dot\cup \{0\} \dot\cup D_-$ with $D_\pm := \{ x \in D \: \pm x_0 > 0\}$ 
is invariant under $\SO_{1,3}(\R)^\uparrow$, so that we obtain a corresponding 
decomposition $\hat U = \hat U^+ \oplus \hat U^0 \oplus \hat U^-$, 
where the spectrum of $\hat U_j\res_{\R^{1,3}}$ is supported by $D_j$. 
Since $\hat U$ is irreducible, only one summand is non-zero. 
Further, $\hat U = \hat U_0$ implies that the translation group acts 
trivially, which is ruled out by the assumption that $U$ 
is faithful. Hence we may w.l.o.g.\ assume that $\hat U = \hat U_+$, 
so that $P_0 > 0$ (i.e., $P_0 \geq 0$ and $\ker P_0 = \{0\}$) 
on $\hat\cH$ and therefore on $\cH_\C$. 
Next we observe that the conjugation $J$ of $\cH_\C$ with 
respect to $\cH$ commutes with $\tilde P_0$, hence 
satisfies $J P_0 J = - P_0$, which leads to the contradiction 
$P_0 = 0$ because it implies that the spectrum of $P_0$ is symmetric 
(cf.~Remark~\ref{rem:2.21} below). This shows that 
the commutant $U_G'$ is $\C$, so that there exists an, up to sign unique,
complex structure on $\cH$ commuting with $U_G$. 
\end{remark}

\subsection{One-parameter groups} 
\mlabel{subsec:one-par}

We have seen in Example~\ref{app:b.1}  that there are three types 
of one-dimensional Lie groups defining involutive group pairs: 
\begin{itemize}
\item[\rm(A)] $\R^\times$, resp., $(\R^\times, \R^\times_+)$, 
\item[\rm(B)] $\R \rtimes \{\pm \id\}$, and 
\item[\rm(C)] $\Pin_2(\R)$. 
\end{itemize}

Before we turn to the most important case (A), we take a brief look 
at the other two cases. 

\begin{remark} \mlabel{rem:2.20} Case (B): Here any antiunitary representation $(U,\cH)$ 
yields a conjugation $J := U_{(0,-1)}$ which defines a real structure on $\cH$ and 
satisfies $J U_t J = U_{-t}$ for $t \in \R$. 
Conversely, every unitary one-parameter group extends to an antiunitary representation of $G$ 
(Lemma~\ref{lem:ext-abelian}). 

Case (C): For the group $G = \T \{\1,J\} = \Pin_2(\R)$, we have  
$J^2 = -\1$ and $JzJ^{-1} = \oline z$ for $z \in \T$, so that 
antiunitary representations correspond to pairs $(H,I)$, where 
$I \in \AU(\cH)$ satisfies $I^4 = \1$ and $H$ is a selfadjoint operator 
satisfying $IHI =  H$ and $e^{\pi i H} = I^2$. This implies in particular that 
$\Spec(H) \subeq \Z$. For any such pair we put 
$U_J := I$ and $U_{e^{it}} := e^{itH}$ (see \cite[\S 4.5]{NO16} for a natural occurence 
of such representations). 
\end{remark}

The following simple observation is the fundamental link between modular 
theory and antiunitary representations. 

\begin{lemma} \mlabel{lem:funda}  
For every continuous antiunitary representation $(U,\cH)$ 
of $\R^\times$ and the infinitesimal generator $H$ defined by 
$U_{e^t} = e^{itH}$, we obtain by 
\[ \Delta := e^H \quad \mbox{ and }  \quad J := U_{-1} \] 
a pair $(\Delta, J)$, consisting of a positive operator $\Delta$ and a 
conjugation $J$ satisfying the modular relation 
\begin{equation}
  \label{eq:modrel1}
 J\Delta J =  \Delta^{-1}.
\end{equation}
Conversely, any such pair $(\Delta, J)$ defines an antiunitary 
representation of $\R^\times$ by 
\[ U_{e^t} := \Delta^{-it/2\pi} \quad \mbox{ and }  \quad 
U_{-1} := J.\] 
\end{lemma}

\begin{proof} The only point one has to observe here is that the 
antiunitarity of $J$ implies that 
$J U_t J = U_t$ corresponds to the relation $JHJ = -H$, which is equivalent to 
$J\Delta J = \Delta^{-1}$. 
\end{proof}

Lemma~\ref{lem:funda}  motivates the following definition from the perspective of 
antiunitary representations: 

\begin{definition} A {\it pair of modular objects} on a complex 
Hilbert space $\cH$ is a pair $(\Delta, J)$, 
where $J$ is a {\it conjugation}, i.e., an antilinear 
isometric involution and $\Delta > 0$ is a positive selfadjoint 
operator satisfying the {\it modular relation} \eqref{eq:modrel1}.
Then $J$ is called the {\it modular conjugation} and $\Delta$ the 
{\it modular operator}. 
\end{definition}

With this terminology, the preceding lemma immediately yields: 

\begin{corollary} \mlabel{cor:2.21}
For any continuous homomorphism $\gamma \: (\R^\times, \R^\times_+) \to (G,G_1)$ and 
any continuous antiunitary representation $(U,\cH)$ of $(G,G_1)$, 
we obtain a pair of modular objects $(\Delta_\gamma, J_\gamma)$ from the representation 
$U \circ \gamma$ of~$\R^\times$. 
\end{corollary}

\begin{remark} \mlabel{rem:2.21} For a selfadjoint operator~$H$, the existence 
of a conjugation $J$ satisfying $JHJ = - H$ 
is equivalent to the restriction $H(0,\infty)$ to the strictly positive 
spectral subspace being 
equivalent to the restriction $H(-\infty,0)$ to the strictly negative 
spectral subspace (\cite{Lo08}).  Only such 
operators $H$ arise as infinitesimal 
generators for antiunitary representations of~$\R^\times$. 
\end{remark}

\begin{ex} Let $(G,G_1)$ be an involutive pair of Lie groups 
and $r \in G\setminus G_1$ be such that $\tau :=c_r\res_{G_1}$ is an involution. 
Then $\Ad(\tau)$ is an involutive 
automorphism of $\g$ and if $\g$ is non-abelian, then $\g^\tau \not= \{0\}$. 

(A) If $r^2 = \1$, then any element $x \in \g^\tau$ leads to a homomorphism 
\[ \gamma_{r,x} \: \R^\times \to G, \quad 
\gamma_{r,x}(e^t) := \exp(tx), \quad 
\gamma_{r,x}(-1) := r.\] 

(B) If $r^2 = \1$, then any element $x \in \g^{-\tau}$ leads to a homomorphism 
\[ \gamma_{r,x} \: \R \rtimes \{\pm \id_\R\} \to G, \quad 
\gamma_{r,x}(t) := \exp(tx), \quad 
\gamma_{r,x}(-1) := r.\] 

(C) If $r^4 = \1$, then any element $x \in \g^{-\tau}$ with 
$\exp(\pi x) = r^2$ leads to a homomorphism 
\[ \gamma_{r,x} \: \Pin_2(\R) = \T \{\1,J\} \to G, \quad 
\gamma_{r,x}(e^{it}) := \exp(tx), \quad 
\gamma_{r,x}(J) := r.\] 
\end{ex}

\begin{definition} \mlabel{def:cplx-type} 
(One-parameter groups of complex type) 
Let $(G, G_1)$ be an involutive Lie group pair. 
We assume that $G$ is a subgroup of a complex Lie group $G_\C$ 
on  which there exists an antiholomorphic involution $\sigma$ such that 
$G \subeq (G_\C)^\sigma$. We consider the set 
\[  \cY_{(G,G_1)} := \{ x \in \g \: 
2\pi = \min \{ t > 0 \: \exp(ti x) =  e\}, \exp(\pi i x) \in G \setminus G_1\}.\] 
We associate to each $x \in  \cY_{(G,G_1)}$ the holomorphic homomorphism 
\[ \gamma_x \: \C^\times \to G_\C, \quad 
\gamma_x(e^z) := \exp(z x).\] 
Then 
$\sigma(\gamma_x(w)) = \gamma_x(\oline w)$ for $w \in \C^\times$ and thus 
$\gamma_x(\R^\times) \subeq (G_\C)^\sigma$ 
holds automatically and $r_x := \gamma_x(-1)$ 
is an involution. For $x \in \cY_{(G,G_1)}$, we thus obtain 
\[ \gamma_x \in \Hom((\R^\times,\R^\times_+), (G, G_1)).\] 
\end{definition}

In Section~\ref{sec:5} below we shall see that many geometric 
realizations of modular automorphism groups come  from elements of $\cY_{(G,G_1)}$, 
where $G = P(d)_+$ is the Poincar\'e group or the conformal 
group $\Conf(\R^{1,d-1}) \cong \OO_{2,d}(\R)/\{\pm \1\}$ of Minkowski space 
(cf.\ \cite[\S17.4]{HN12}). 
This motivates the following discussion of examples.

\begin{ex} \mlabel{ex:one-par} 
(a) For $(G, G_1) = (\R^\times, \R^\times_+)$ and $G_\C = \C^\times$ and 
$\exp(z) = e^z$,  we have 
$\cY_{(G,G_1)} = \{\pm 1\} \subeq \R = \g.$ 

(b) (Lorentz groups) 
For 
\[ G = \SO_{1,1}(\R) \subeq G_\C = \SO_{1,1}(\C) 
= \Big\{ \pmat{ a & b \\ b & a} \: a,b \in \C, 
 a^2 - b^2 = 1 \Big\}, \] 
we have $G \cong \R^\times$ and $G_\C \cong \C^\times$, so that we basically 
have the same situation as under~(a). Here a canonical 
generator of the Lie algebra is the boost generator 
\begin{equation}
  \label{eq:boostgen2}
 b_0 := \pmat{0 & 1 \\ 1 & 0} 
\quad \mbox{ with } \quad 
e^{z b_0} = \pmat{ \cosh z & \sinh z \\ \sinh z & \cosh z} \quad \mbox{ and } \quad 
r_{b_0} = e^{\pi i b_0} = - \1.
\end{equation}
We have $\cY_{(G,G_0)} = \{ \pm b_0\}$. 

This example embeds naturally into the higher dimensional 
Lorentz groups 
$G = \SO_{1,d}(\R) \subeq G_\C = \SO_{1,d}(\C)$, where 
\begin{equation}
  \label{eq:bostgen-d}
b_0 := E_{10} + E_{01} \in \cY_{(G,G_0)}\quad \mbox{ and } \quad 
 r_{b_0} = R_{01} = \diag(-1,-1,1,\ldots, 1). 
\end{equation}

Since the simple real Lie algebra $\g = \so_{1,d}(\R)$ (for $d \geq 2$)  
is of real rank $1$, all $\ad$-diagonalizable elements $x \in \g$ are conjugate 
to a multiple of $b_0$. All these elements $x$ are diagonalizable matrices 
and $\im(x)$ is a two-dimensional Minkowski plane in which the two eigenvectors 
are light-like. Conversely, every triple $(\beta, \ell_+, \ell_-)$ 
consisting of $\beta \in \R^\times$ and two linearly independent light-like vectors 
$\ell_\pm$ specifies such an element $x = x(\ell_+, \ell_-, \beta) \in \g$ by 
$x \ell_\pm = \pm\beta \ell_\pm$ and $\ker x = \{\ell_1, \ell_2\}^\bot$. 
We then have 
\[ \cY_{(G,G_0)} = \Ad(G)b_0 = 
\{ x(\ell_+, \ell_-, \beta) \: \beta = 1 \} 
\cong \SO_{1,d}(\R)/(\SO_{1,1}(\R) \times \SO_{d-1}(\R)),\] 
and this is a symmetric space because the centralizers of $b_0$ and 
the involution $r_{b_0}$ share the same identity component. 

(c) For the affine group $G := \Aff(\R)\cong \R \rtimes \R^\times$ of the real line,  
the coset $G \setminus G_0$ consists of the orientation reversing affine maps. 
Note that $G_\C \cong \C \rtimes \C^\times$ and $G_\C^\sigma = G$. 
Here $\cY_{(G,G_0)} \cong \R \times \{ \pm 1\}$ 
is the set of real affine vector fields $X$ 
for which the vector field $iX$ on $\C$ generates a $2\pi$-periodic flow 
(whose center lies on the real axis). 
\end{ex}

\begin{ex} \mlabel{ex:proj-grp} 
We consider the real projective group 
$G = \PGL_2(\R)\subeq G_\C = \PGL_2(\C)$ acting on the real projective line 
$\bS^1 \cong \R \cup \{ \infty\}$, 
resp., on the Riemann sphere $\C \cup \{\infty\} \cong \bP_1(\C)$. 

A subset $I \subset \bS^1$ is called an {\it interval} 
if it is connected, open, non-empty and not dense. Then the interior $I'$ 
of its complement also is an interval. For every interval there is a canonical 
involution $r_I \in \PGL_2(\R)$ fixing both endpoints and exchanging $I$ and~$I'$.
The centralizer of $r_I$ in $\PSL_2(\R)$ is isomorphic to $\PSO_{1,1}(\R) 
\cong \R$, hence connected, and there exists 
an element $x_I \in \cY_{(G,G_0)}$ which is up to sign unique. 
The corresponding homomorphism $\gamma^I := \gamma_{x_I} \: \R^\times \to G$ 
satisfies $\gamma^I_{-1} = \exp(\pi i x_I) = r_I$. 

For the interval $I = (0,\infty)$, we have 
\[ \gamma^I_t(z) = t z \quad \mbox{ and } \quad  r_I(z) = -z. \]  
For $I = (-1,1)$, we have $\gamma^I(\R^\times) = \PO_{1,1}(\R)$ and 
$\gamma^I_{2t}(z) := \frac{\cosh t \cdot z + \sinh t}{\sinh t \cdot z + \cosh t}.$
This leads to 
\[ \gamma^I_{2ti}(z) = \frac{\cos t \cdot z + i \sin t}{i \sin t \cdot z + \cos t}, 
\quad \mbox{ so that} \quad 
\gamma^I_{2\pi i}(z) = z \quad \mbox{ and } \quad 
r_I(z) = \gamma^I_{\pi i}(z) = \frac{1}{z}.\] 
\end{ex}

\subsection{Some low-dimensional groups}
\mlabel{subsec:2.4} 
\subsubsection{The affine group of the real line} 
\mlabel{subsubsec:2.5.1}

We consider the affine group $G := \Aff(\R) = \R \rtimes \R^\times$ 
and its identity component $G_1 = \R \rtimes \R^\times_+$. 
We say that a unitary representation $(U,\cH)$ of $G_1$ is of {\it positive 
energy} if $U_{(t,1)} = e^{itP}$ with $P \geq 0$, i.e., the restriction to the 
translation subgroup has non-negative spectrum. We speak of {\it strictly positive 
energy} if, in addition, $\ker P = \{0\}$. 

Up to unitary equivalence, $G_1$ has exactly one 
irreducible unitary representation with strictly positive energy 
and every unitary representation with strictly positive energy is a 
multiple of the irreducible one. The analogous statement holds for negative 
energy (\cite[Thm.~2.8]{Lo08}). Further, any 
unitary representation $U$ of $G_1$ decomposes uniquely as a direct 
sum $U = U^+ \oplus U^0 \oplus U^-$, where $U^\pm$ have strictly positive/negative 
energy and the translation group is contained in $\ker U^0$.   

The unique irreducible representation of strictly positive energy can be realized 
on $\cH := L^2(\R^+)$ by 
\begin{equation}
  \label{eq:pics-s1}
(U_{(t,e^s)}f)(x) = e^{itx} e^{s/2} f(e^sx).
\end{equation}
It  obviously extends by $U_{(0,-1)} f := \oline f$ to an irreducible 
antiunitary representation of~$G$. 
By Theorem~\ref{thm:equiv} we thus obtain up to equivalence 
precisely one irreducible antiunitary representation of $G$ 
with strictly positive energy. More generally, we have 
by \cite[Prop.~2.11]{Lo08} and 
Theorem~\ref{thm:equiv}: 

\begin{proposition}
Every unitary representation $(U,\cH)$ 
of $\Aff(\R)_0$ of strictly positive energy extends to an 
antiunitary representation $\oline U$ of $\Aff(\R)$ on the same Hilbert space 
which is unique up to equivalence.   
\end{proposition}

The representation theory of the affine group can be used to draw some 
general conclusions on spectra of one-parameter groups. 

\begin{proposition} \mlabel{prop:2.x} Let $G$ be a connected Lie group and 
$(U,\cH)$ be a unitary representation for which $\dd U$ is faithful 
and $x \in \g$. Then the following assertions hold: 
\begin{itemize}
\item[\rm(a)] If $\ad x$ has a non-zero real eigenvalue, then $\Spec(i\dd U(x)) = \R$. 
\item[\rm(b)] If $\g$ is semisimple and 
$0\not=y$ is nilpotent, then $\Spec(i\dd U(y)) \in \{ \R,\R_+, \R_-\}$. 
\item[\rm(c)] If $0\not=x \in \g$ is such that 
$\ad x$ is diagonalizable and $\fb \trile \g$ is the ideal generated by 
$\im(\ad x)$ and $x$, then $\ker\big(\dd U(x)\big) = \cH^B 
= \{ \xi \in \cH \: (\forall g \in B)\ U_g\xi = \xi\}$ holds 
for the corresponding integral subgroup $B \trile G$. 
\end{itemize}
\end{proposition}

\begin{proof} (a) Let $0\not=y \in \g$ with $[x,y] = \lambda y$ for some 
$\lambda \not=0$. Then $\fh := \R x + \R y$ is a $2$-dimensional non-abelian 
subalgebra and $\dd U(y) \not=0$. Therefore the assertion follows from the fact that, 
for all irreducible unitary representations of the corresponding 
$2$-dimensional subgroup isomorphic to $\Aff(\R)_0$, 
the spectrum of $i \dd U(x)$ coincides with $\R$.

(b) Using the Jacobson--Morozov Theorem, we find an $h \in \g$ with 
$[h,x] = x$, so that the Lie algebra $\fb := \R h + \R x$ is isomorphic to 
$\aff(\R)$. Then the result follows from the classification of the 
irreducible representations of the group $\exp(\fb) \cong \Aff(\R)$. 

(c) Let $\g = \oplus_{\mu\in \R} \g_\mu(\ad x)$ denote the eigenspace 
decomposition of $\g$ with respect to the diagonalizable operator $\ad x$. 
Then the representation theory of $\Aff(\R)_0$ implies that, 
for $\mu\not=0$, the operators in $\dd U(\g_\mu(\ad x))$ vanish 
on $\ker \dd U(x)$. This shows that the Lie subalgebra $\fh$ generated 
by $x$ and $[x,\g]$ acts trivially on $\ker \dd U(x)$. 
Since this subalgebra is invariant under $\ad(\g_0(\ad x))$ and 
contains the other eigenspaces of $\ad x$, it is an ideal of $\g$, 
hence coincides with~$\fb$. Therefore $\ker\big(\dd U(x)\big) = \cH^B$. 
\end{proof}

\subsubsection{The projective group of the real line} 

We consider the projective group $G = \PGL_2(\R)$ 
and its identity component $G_1 = \PSL_2(\R)$. 
We write $r(x) = -x$ for the reflection in $0$ which 
commutes with the dilation group $\R^\times \subeq \Aff(\R) \subeq \PGL_2(\R)$ 
(cf.~Example~\ref{ex:proj-grp}). Note that $r$ extends to an antiholomorphic 
automorphism $r(z) := -\oline z$ of the upper half plane $\C_+$, so that 
we obtain an identification of $G$ with the group $\AAut(\C_+)$ 
(Example~\ref{ex:2.11}(d)).

For the generators of $\fsl_2(\R)$, we write 
\[ T = \pmat{0 & 1 \\ 0 & 0}, \quad 
S = \pmat{0 & 0 \\ -1 & 0} \quad \mbox{ and } \quad 
E = \frac{1}{2} \pmat{ 1 & 0 \\ 0 & -1}.\]  
They satisfy the commutation relations 
\[ [E,T] = T, \quad [E,S] = -S \quad \mbox{ and }\quad 
[T,S] = - 2 E.\]
In the complexification $\fsl_2(\C)$, we have the basis 
\begin{align*}
  L_{\pm 1} := \frac{1}{2}\pmat{1 & \mp i \\ \mp i & -1} = E \mp \frac{i}{2}(T-S), 
\qquad   L_0 := -\frac{i}{2} \pmat{0 & 1 \\ -1 & 0} =- \frac{i}{2}(T+S).
\end{align*}
These elements satisfy the relations 
\[ [L_0, L_{-1}]= L_{-1}, \quad [L_0, L_1] = - L_1 
\quad \mbox{ and } \quad [L_1, L_{-1}] = -2 L_0.\] 
\begin{definition}
The element $L_0 \in i \fsl_2(\R)$ is called the {\it conformal Hamiltonian}. 
A~unitary representation  $(U,\cH)$ of $\tilde\SL_2(\R)$ is called a 
{\it positive energy representation} if \break ${\dd U(L_0) \geq 0}$. 
\end{definition}

The following result is well known; for a proof in the spirit of the 
present exposition, we refer to \cite[Cor.~2.9]{Lo08}. 

\begin{corollary} \mlabel{cor:psl2-restrict}
For every non-trivial irreducible positive energy representation $U$ of 
the simply connected covering group 
$\tilde\SL_2(\R)$, the restriction to $\Aff(\R)_0$ is also irreducible. 
If $U$ is non-trivial, then $U\res_{\Aff(\R)_0}$ is the unique irreducible 
representation with strictly positive energy.
\end{corollary}

\begin{remark} \mlabel{rem:lowei} 
In $\PSL_2(\R)$, we have $\exp(2\pi i L_0) = \1$, 
so that, for every 
irreducible positive energy representation of $\PSL_2(\R)$, 
the spectrum of $\dd U(L_0)$ is contained in $m + \N_0$ for some 
$m \in \N_0$. We call $m$ its 
{\it lowest weight} and write $\cH_m= \C \xi_m$ for the $m$-eigenspace of 
$L_0$ in $\cH$. Then $\dd U(L_1) \xi_m = 0$ and 
$\xi_{m + k} := \dd U(L_{-1})^k \xi_m$, $k \in \N_0$, is an orthogonal 
basis of $\cH$.
\end{remark}

\begin{theorem}{\rm(\cite[Thm.~2.10]{Lo08})}  \mlabel{thm:1.4}
Every unitary positive energy representation $U$ of $\PSL_2(\R)$ extends to an 
antiunitary representation $\oline U$ of $\PGL_2(\R)$ on the same Hilbert space. 
This extension is unique up to isomorphism 
and, if $U$ is irreducible, then $J := \oline U_r$ {\rm(for $r(x) = -x$)} 
is unique up to a multiplicative factor in $\T$. 
\end{theorem}

\begin{proof} In view of Theorem~\ref{thm:equiv}, it suffices 
to verify the first assertion. 

Here the main point is to define the antiunitary 
involution on the irreducible lowest weight representation $U^m$ of 
lowest weight~$m \in \N_0$. We 
specify an antiunitary involution $C$ on $\cH$ by $C \xi_n = \xi_n$ for 
$n \geq m$  (cf.\ Remark~\ref{rem:lowei}). 
Then $C$ commutes with $L_0$ and $L_{\pm 1}$ and 
\[ C E C = E, \quad C T C = - T \quad \mbox{ and }\quad 
C S C = - S.\] 
This implies that $C U^m_g C = U^m_{rgr}$ for $g \in \PSL_2(\R)$. 
\end{proof}

\begin{definition} (Positive energy representations) 
\mlabel{def:posen-rep}
(a) A unitary representation $(U,\cH)$ of the translation group 
$\R^d= \R^{1,d-1}$ of Minkowski space is said to be a {\it positive 
energy representation} if $-i \dd U(x) \geq 0$ for 
$x \in \oline{V_+}$. 

(b) A unitary representation $(U,\cH)$ of the Poincar\'e group 
$P(d)^\uparrow_+$ is said to be a {\it positive 
energy representation} if its restriction to the translation subgroup 
is of positive energy.  We likewise define antiunitary 
positive energy representations of $P(d)_+$. 
\end{definition}

\begin{remark} 
(a) For the group $G := P(2)_+ \cong \R^{1,1} \rtimes \SO_{1,1}(\R)$ 
and the reflection $r := (0,-\1)$ inducing on $G$ the involution 
$\tau(b,a)= (-b,a)$, there are similar results to Theorem~\ref{thm:1.4}  
(cf.~Theorem~\ref{thm:wies2-standard}). 
Here the main point is to see that the irreducible strictly positive energy 
representations $(U,\cH)$ of $G_0$ carry a natural conjugation  
that we can use for the extension. In the $L^2$-realization 
on the hyperbolas 
\[ \cO_m = \{ (\lambda, \mu) \in \R^2 \: \lambda^2 - \mu^2 = m^2 \}, \quad 
m > 0, \]
suggested by Mackey theory, we can extend the representation 
simply by $U_{(0,0,-\1)} f = \oline f$. 

(b) For the Poincar\'e group $P(d)_+$, the situation is more complicated. 
The irreducible strictly positive energy representations of $P(d)^\uparrow_+ 
\cong \R^d \rtimes \SO_{1,d-1}(\R)^\uparrow$ are induced from representations 
of the stabilizer group $\SO_{1,d-1}(\R)_{e_0} \cong \SO_{d-1}(\R)$ 
and realized in vector-valued $L^2$-spaces on the hyperboloids 
\[ \cO_m = \{ (p_0,\bp) \in \R^d \: 
p_0^2 - \bp^2 = m^2\}, \quad m > 0.\]
Since the stabilizer group is non-trivial for $d > 2$, the existence 
of an antiunitary extension to $P(d)_+$ depends on the 
existence of an antiunitary extension of the representation 
$(\rho, V)$ of $\SO_{d-1}(\R)$ to $\OO_{d-1}(\R)$.  We refer to 
\cite{NO17} for a detailed analysis of these issues; 
see also \cite[Thm.~9.10]{Va85} for a discussion concerning the 
Poincar\'e group. 
\end{remark}

\subsubsection{The Heisenberg group} 

In this subsection we recall the close connection between 
unitary representations of the $3$-dimensional 
Heisenberg group and positive energy representations of~$\Aff(\R)_0$ 
(cf.~\cite[Thm.~2.8]{Lo08}) which also extends to antiunitary extensions. 

We define the {\it Heisenberg group} $\Heis(\R^2)$ as the 
manifold $\T \times \R^2$, endowed with the group multiplication 
\[ (z,s,t)(z',s',t') = (zz' e^{is't}, s + s', t + t'). \] 
Note that 
\[ \Heis(\R^2) \cong (\T \times \R) \rtimes_\alpha \R 
\quad \mbox{ for } \quad 
\alpha_t(z,s) = (z e^{ist}, s).\] 
Extending the action of $\R$ on $\T \times \R$ 
to an action of $\R \times \Z_2 \cong \R^\times$ via 
\[ \beta_r(z,s) = (z r^{is} , s) \quad \mbox{ and }\quad 
\beta_{-1}(z,s) = (\oline z, -s), \] 
we obtain the larger group 
\[ \Heis(\R^2)_\tau 
\cong (\T\times \R) \rtimes_\beta \R^\times 
\cong \Heis(\R^2) \rtimes \{\1,\tau\}, \quad \mbox{ with }  \quad 
\tau(z,s,t) = (\oline z, -s,t).\] 

\begin{proposition} There is a natural one-to-one correspondence between 
unitary representations $(\tilde U, \cH)$ of $\Heis(\R^2)$ 
satisfying $\tilde U_{(z,0,0)} = z\1$ and unitary strictly positive 
energy representations $(U,\cH)$ of $\Aff(\R)_0$. It is established 
as follows: 
\begin{itemize}
\item[\rm(i)] If $U$ is given and $U_{(b,1)} = e^{ibP}$ with $P > 0$, then we put 
$W_s := e^{is \log P}$ and 
$\tilde U_{(z,s,t)} := z W_s U_{(0,e^t)}$. 
\item[\rm(ii)] If $\tilde U$ is given and 
$W_s := \tilde U_{(1,s,0)} = e^{isA}$, then we put 
$U_{(s,e^t)} := e^{is \exp A} \tilde U_{(1,0,t)}$. 
\item[\rm(iii)] This correspondence extends naturally 
to antiunitary representations of $\Heis(\R^2)_\tau$ and 
antiunitary positive energy representations of $\Aff(\R)$. 
\end{itemize}
\end{proposition}

\begin{proof} (i) Let $V_t := U_{(0,e^t)}$ and $A := \log P$. 
Then $V_t P V_{-t} = e^t P$ implies that $W_s = e^{isA}$ satisfies 
\[ V_t A V_{-t} = t \1 + A \quad \mbox{ and }\quad
V_t W_s V_{-t}  = e^{ist} W_s.\] 
Therefore $V$ and $W$ define a unitary representation of 
$\Heis(\R^2)$ via $\tilde U_{(z,s,t)} := z W_s V_t.$

(ii) With $V_t := \tilde U_{(1,0,t)}$ and $W_s = \tilde U_{(1,s,0)} = e^{isA}$,  
the positive operator $P := e^A$ satisfies 
$V_t P V_{-t}= e^t P$, so that we obtain a positive energy representation of 
$\Aff(\R)_0$  by $U_{(s,e^t)} := e^{is P} V_t$. 

(iii) If, in addition, $U$ is an antiunitary representation of 
$\Aff(\R)\cong \R \rtimes \R^\times$ and $J = U_{(0,-1)}$, then 
$J U_{(b,a)} J = U_{(-b,a)}$ leads to 
$J P J = P$ and thus to $JAJ = A$. We therefore obtain an antiunitary 
representation $\hat U$ of 
$\Heis(\R^2)_\tau \cong (\T\times \R) \rtimes_\alpha \R^\times$ by 
$\hat U_{(z,s,a)} := \hat U_{(z,s)} U_{(0,a)}.$ 
\end{proof}

\begin{remark}
Write $\Heis(\R^2)_\tau$ as the semidirect 
product $\Heis(\R^2) \rtimes \{\1,\tau\}$, 
where $\tau(z,s,t) = (\oline z, -s, t)$. 
Then the conjugacy class $C_\tau \subeq \Heis(\R^2)_\tau$ is a $2$-dimensional 
symmetric space diffeomorphic to  $\T \times \R$ 
and the centralizer of $\tau$ in $\Heis(\R^2)$ is the subgroup 
$\{\pm 1\} \times \{0\} \times \R$ which also commutes 
with the whole subgroup $\{(1,0)\} \times \R$. 
Therefore $C_\tau$ can be identified with the conjugacy class 
of the homomorphism $\gamma \: \R^\times \to \Heis(\R^2)_\tau 
\cong (\T \times \R) \rtimes_\beta \R^\times$ with 
$\gamma(t) = (1,0,t)$. 
\end{remark}

\section{Modular objects and standard subspaces} 
\mlabel{sec:3}

Besides antiunitary representations of $\R^\times$ 
(Lemma~\ref{lem:funda}), 
there are other interesting ways to encode modular objects $(\Delta, J)$. 
Below we discuss some of them. In particular, we introduce the 
concept of a standard subspace $V \subeq \cH$ 
which is a geometric counterpart of antiunitary representations of~$\R^\times$ 
(Proposition~\ref{prop:3.2}). 
We also discuss how the embedding $V \subeq \cH$ can be 
obtained from the orthogonal one-parameter group 
$\Delta^{it}\res_V$ on $V$ (\S \ref{subsec:orthog}),  
and in \S \ref{subsec:3.3} we introduce half-sided modular inclusions 
of standard subspaces and how they are related to 
antiunitary representations of $\Aff(\R)$, $P(2)_+$ and $\PGL_2(\R)$. 
Modular intersections are studied in \S\ref{subsec:modint}.  

\subsection{Standard subspaces} 

We now turn to the fundamental concept of a standard subspace $V$ of a 
complex Hilbert space $\cH$. The key structures on the set 
$\Stand(\cH)$ of standard subspaces is a natural action of the 
group $\AU(\cH)$, an order structure induced by inclusion, and 
an involution $V \mapsto V'= i V^{\bot_\R}$ defined by the symplectic orthogonal space. 

\begin{definition} A closed real subspace $V \subeq \cH$ is called a {\it standard 
real subspace} (or simply a {\it standard subspace}) 
if $V \cap i V = \{0\}$ and $V + i V$ is dense in $\cH$. 
We write $\Stand(\cH)$ for the set of standard subspaces of~$\cH$.
\end{definition}

For every standard subspace $V \subeq \cH$, we obtain an 
antilinear unbounded operator 
\[ S \: \cD(S) := V + i V \to \cH, \qquad 
S(v + i w) := v - i w \] 
and this operator is closed, so that $\Delta_V := S^*S$ is a positive 
selfadjoint operator. We thus obtain 
the polar decomposition 
\[ S = J_V \Delta_V^{1/2},\] 
where $J_V$ is an antilinear isometry, and 
$S = S^{-1} = \Delta_V^{-1/2} J_V^{-1} = J_V^{-1} (J_V \Delta_V^{-1/2} J_V^{-1})$ 
leads to $J_V^{-1} = J_V$ and the modular relation $J_V\Delta_V J_V = \Delta_V^{-1}$. 
If, conversely, $(\Delta, J)$ is a pair of modular objects, then 
$S := J \Delta^{1/2}$ is a densely defined antilinear involution and 
\[ \Fix(S) := \{ \xi \in \cD(S) \: S\xi = \xi \} \] 
is a standard subspace with $J_V = J$ and $\Delta_V = \Delta$.
The correspondence between modular objects and standard subspaces 
is the core of Tomita--Takesaki Theory (see Theorem~\ref{thm:tom-tak} below).

Combining the preceding discussion with 
Lemma~\ref{lem:funda}, we obtain:
\begin{proposition}\mlabel{prop:3.2} If $(U,\cH)$ is an antiunitary 
representation of $\R^\times$ with $U_{e^t} = \Delta^{-it/2\pi}$ for $t \in \R$ and 
$J := U_{-1}$, then $V := \Fix(J\Delta^{1/2})$ is a standard subspace. 
This defines a bijection $V \leftrightarrow U^V$ between antiunitary representations of 
$\R^\times$ and standard subspaces.
\end{proposition}

\begin{remark} \mlabel{rem:anaext} 
The parametrization of the one-parameter group 
in Proposition~\ref{prop:3.2} may appear artificial, but it turns out that 
it is quite natural. 
As $V \subeq \cD(\Delta^{1/2})$, for each $v \in V$, the orbit map 
$U^v(g) := U_g v$ has an analytic extension 
\[ \{ z \in \C \: 0 \leq \Im z \leq \pi \} \to \cH, 
\quad z \mapsto U^v_{e^z} := \Delta^{-iz/2\pi}v \] 
with $U^v(i\pi) = \Delta^{1/2} v = J v$. This fits with $U_{-1} = J$ 
and it is compatible with  the context of Definition~\ref{def:cplx-type}, where 
$\gamma(-1) = \exp(\pi ix)$ is obtained by analytic continuation from 
$\gamma(e^t) = \exp(tx)$.
\end{remark}

\begin{remark} \mlabel{rem:4.3} 
(a) If $V = \Fix(S)$ is a standard subspace with modular objects 
$(\Delta, J)$, then 
\begin{equation}
  \label{eq:delta-conj}
 \Delta^{1/4} S \Delta^{-1/4} = \Delta^{1/4} J \Delta^{1/4} 
= \Delta^{1/4}\Delta^{-1/4}J = J
\end{equation}
implies that $V = \Fix(S) = \Delta^{-1/4} \Fix(J) = \Delta^{-1/4}\cH^J$. 

(b) Write $\Stand_0(\cH)$ for the set of those standard subspaces $V$ 
for which $V + i V = \cH$, i.e., the antilinear involution~$S$ 
is bounded. 
Combining \eqref{eq:delta-conj} with the fact that the unitary group $\U(\cH)$ 
acts transitively on the set of all conjugations (=antiunitary involutions), 
it follows that the group $\GL(\cH)$ acts transitively on $\Stand_0(\cH)$. This leads
to the structure of a Banach symmetric space on this set  
\[ \Stand_0(\cH) \cong \GL(\cH)/\GL(\cH^J) \cong \GL(\cH)/\GL(\cH)^J,\] 
where $J$ is any conjugation on $\cH$ (cf.\ Appendix~\ref{app:a.2} and \cite{Kl11}). 
For $\cH = \C^n$, we obtain in particular 
\[ \Stand(\C^n) = \Stand_0(\C^n) \cong \GL_n(\C)/\GL_n(\R). \] 
For elements of $\Stand_0(\cH)$, there are no proper inclusions. As we shall see in 
\S \ref{subsec:3.3}, the order structure on $\Stand(\cH)$ is non-trivial if $\cH$ is infinite dimensional. 

(c) To extend (b) to arbitrary standard subspaces $V$, we note that  
a dense complex subspace $\cD\subeq \cH$ carries at most one 
Hilbert space structure (up to topological linear isomorphism) 
for which the inclusion 
$\cD \into \cH$ is continuous (Closed Graph Theorem). 
We consider the category $\cG$ whose objects are 
all dense subspaces $\cD \subeq \cH$ 
carrying such Hilbert space structures and whose morphisms are 
the topological linear isomorphisms $\cD_1 \to \cD_2$ with respect to the 
intrinsic Hilbert space structures. This defines a category 
in which all morphisms are invertible, so that we actually 
obtain a groupoid. As all these subspaces $\cD$ are isomorphic to 
$\cH$ as Hilbert spaces, this groupoid acts transitively. 

For each standard subspace $V \subeq \cH$, the dense subspace 
$V + iV$ carries the natural Hilbert structure obtained from the identification
 with the complex Hilbert space $V_\C$. Therefore the groupoid $\cG$ 
acts transitively on $\Stand(\cH)$ with stabilizer groups $\cG_V \cong \GL(V)$. 

(d) Write $\Conj(\cH)$ for the set of conjugations on $\cH$ 
(Examples~\ref{exs:conj}). 
Then the map $\Stand(\cH) \to \AU(\cH), V \mapsto J_V$ is surjective and 
$\AU(\cH)$-equivariant. The fiber in a fixed conjugation $J$ 
corresponds to the set of all positive 
operators $\Delta$ satisfying $J\Delta J = \Delta^{-1}.$ 
Passing to $D := i\log \Delta\res_{\cH^J}$, it follows that it can be parametrized 
by the set of all skew-adjoint operators on the real Hilbert space $\cH^J$ 
(see also Remark~\ref{rem:2.3}(b) for a different parametrization). 
\end{remark}

The problem to describe the set of pairs 
$(V,\cH)$, where $V \subeq \cH$ is a standard subspace, 
can be addressed from two directions. One could 
either start with a real Hilbert space $V$ and ask 
for all those complex Hilbert spaces into which $V$ embeds 
as a standard subspace, or start with the pair $(\cH,J)$, 
respectively the real Hilbert space $\cH^J$, and ask for all standard 
real subspaces $V \subeq \cH$ with $J_V = J$. Both problems 
have rather explicit answers that are easily explained 
(see \cite{NO16} for details). 

\begin{remark} \mlabel{rem:2.3} 
(a) Let $(V, (\cdot,\cdot))$ be a real Hilbert space. 
For any realization of $V$ as a standard subspace of $\cH$, the 
restriction of the scalar product of $\cH$ to $V$ is a complex-valued 
hermitian form 
\[ h(v,w) := \la v,w\ra  = (v,w) + i \omega(v,w),\] 
where $\omega \: V \times V \to \R$ is continuous and skew-symmetric, 
hence of the form $\omega(v,w) = (v,Cw)$ for a skew-symmetric operator 
$C = - C^\top$ on $V$ satisfying $\|Cv\| < \|v\|$ for any non-zero 
$v \in V$ (\cite[Lemma~A.10]{NO16}). Conversely, we obtain 
for every such operator $C$ on $V$ by completion of $V_\C$ 
with respect to $h$ a complex Hilbert space in which $V$ is a standard 
real subspace. Then $C$ extends to a bounded skew-hermitian operator 
$\hat C$ on $\cH$ satisfying
\[ \Delta = \frac{\1 - i \hat C}{\1 + i \hat C} 
\quad \mbox{ and } \quad 
\hat C = i \frac{\Delta - \1}{\Delta + \1}.\] 

(b) If we start with the conjugation $J$ on $\cH$, then 
the standard subspaces $V$ with $J_V = J$ are the subspaces of the form 
$V = (\1 + i C)\cH^J$, where $C \in B(\cH^J)$ is a skew-symmetric operator 
satisfying $\|Cv\| < \|v\|$ for $0\not= v \in \cH^J$ 
(\cite[Lemma~B.2]{NO16}). Writing also $C$ for its complex linear extension 
to $\cH$, we then have 
\[ \Delta^{1/2} = \frac{\1 - i C}{\1 + i C} \quad \mbox{ and }\quad 
C = i \frac{\Delta^{1/2}-\1}{\Delta^{1/2}+ \1}.\]
\end{remark}

\begin{remark} If $V$ is a standard subspace of $\cH$ 
and  $W \subeq V+ i V$ is a real subspace closed in $\cH$ 
such that $W$ corresponds to a standard subspace of the complex Hilbert space $V_\C$, then 
$W$ is also standard in $\cH$ because the closure of 
$W + iW$ contains $V + i V$, hence all of $\cH$. 
\end{remark}

\subsection{Symplectic aspects of standard subspaces} 
\mlabel{subsec:3.2} 

Let $V \subeq \cH$ be a standard subspace and 
consider the corresponding 
antiunitary representation $U^V \:\R^\times \to \AU(\cH)$ 
with $U^V_{-1} = J^V$ and $U^V_{e^t} = \Delta^{-it/2\pi}$ 
(Proposition~\ref{prop:3.2}). 
Since the operators $\Delta^{it}$ commute with 
$S = J \Delta^{1/2}$, they leave the closed subspace 
$V = \Fix(S)$ invariant. Further, the relation 
$JSJ = \Delta^{1/2} J = S^* =  J \Delta^{-1/2}$ 
implies that 
\[ JV = V',\quad \mbox{ where } \quad 
V' := \{ w \in \cH \: (\forall v \in V) \ 
\Im \la v, w \ra = 0\} = i V^{\bot_\R}  \] 
is the {\it symplectic orthogonal space of $V$}, and 
$V^{\bot_\R}$ denotes the orthogonal complement of $V$ in the underlying 
real Hilbert space $\cH^\R$ (\cite[Prop.~3.2]{Lo08}). 
In particular, the orbit $U^V_{\R^\times}V = \{V,V'\}$ consists of at most 
two standard subspaces. 

\begin{lemma}\mlabel{lem:stand-factorial} 
The following assertions hold: 
\begin{itemize}
\item[\rm(i)] $U^{V'}(t) = U^V(t^{-1})$ for $t \in \R^\times$, 
is the antiunitary representation corresponding to~$V'$. 
\item[\rm(ii)] $J_{V'} = J_V$ and $\Delta_{V'} = \Delta_V^{-1}$. 
\item[\rm(iii)] $V \cap V' = \cH^{U^V}$ 
is the fixed point space for the antiunitary representation $(U^V,\cH)$ of~$\R^\times$. 
\item[\rm(iv)] $V = V'$ is equivalent to $\Delta = \1$. 
\end{itemize}
\end{lemma}

\begin{proof} (i) and (ii) follow immediately 
from  $V' = \Fix(J \Delta^{-1/2})$. 

(iii) If $v \in V\cap V'$, then 
$v = Sv = S^*v$ implies $\Delta v = v$ and hence $Jv = v$. Conversely, 
these two relations imply $v \in V \cap V'$. 

(iv) follows from (ii). 
\end{proof}

\begin{remark} \mlabel{rem:dirsum} (Direct sums of standard subspaces) 

(a) Suppose that $V_j \subeq \cH_j$ are standard subspaces for $j =1,2$. 
Then $V := V_1 \oplus V_2 \subeq \cH_1 \oplus \cH_2$ is a standard subspace. 
We have $J_V = J_{V_1} \oplus J_{V_2}$ and 
$\Delta_V = \Delta_{V_1} \oplus \Delta_{V_2}$. 
In particular, the corresponding antiunitary representation 
$U^V$ of $\R^\times$ is the direct sum $U^{V_1} \oplus U^{V_2}$. 

(b) In particular, every standard subspace $V$ can be written 
as such a direct sum 
\[ V = (V \cap V') \oplus V_1, \quad \mbox{ where } \quad V_1' \cap V_1 = \{0\} \] 
and $(V \cap V')_\C$ is the set of fixed points of the unitary 
representation $U^V\res_{\R^\times_+}$ (Lemma~\ref{lem:stand-factorial}). 
\end{remark}

\begin{lemma} \mlabel{lem:split} Let $V$ be a standard subspace, 
$V_1 \subeq V$ be a closed subspace and $V_2 := V \cap V_1^{\bot_\R}$ be its orthogonal 
complement in~$V$. Then the following are equivalent: 
\begin{itemize}
\item[\rm(i)] $V = V_1 \oplus V_2$ is a direct sum of standard subspaces. 
\item[\rm(ii)] $V_1 \subeq V_2'$, i.e., $i V_1 \bot V_2$ in $\cH$. 
\item[\rm(iii)] $V_1$ is invariant under the modular automorphisms 
$(\Delta_V^{it})_{t \in \R}$. 
\end{itemize}
If these conditions are satisfied and $V_1$ is also standard, then $V  =V_1$.
\end{lemma}

\begin{proof} (i) $\Leftrightarrow$ (ii) is easy to verify. 

(i) $\Leftrightarrow$ (iii): Clearly, (i) implies (iii). To see the converse, 
consider the closed subspace $\cH_1 := \oline{V_1 + i V_1}$ of $\cH$. 
Then, for each $v \in V$, the curve $t \mapsto \Delta_V^{-it/2\pi}v$ is contained in $\cH_1$, 
hence the same is true for its analytic continuation to the strip  
\break $\{ z \in \C \: 0 \leq \Im z \leq \pi\}$ (Remark~\ref{rem:anaext}). Therefore
$\Delta^{1/2}v = Jv \in \cH_1$ and thus $\cH_1$ is invariant under the 
antiunitary representation $U^V$ of $\R^\times$ corresponding to $V$
(Proposition~\ref{prop:3.2}). Since the orthogonal decomposition 
$\cH = \cH_1 \oplus \cH_1^\bot$ reduces $U^V$, the standard subspace $V$ decomposes 
accordingly. 

If (i)-(iii) are satisfied and $V_1$ is also standard, then (i) implies 
that $V_1 = V$ (cf.~\cite[Prop.~3.10]{Lo08}). 
\end{proof}


\subsection{Orthogonal real one-parameter groups} 
\mlabel{subsec:orthog}

For any standard subspace $V$, the unitary operators 
$\Delta^{it}$ define on the real Hilbert space~$V$ 
a continuous orthogonal one-parameter group $(U,V)$ 
(\S \ref{subsec:3.2}). 

If, conversely, $(U_t)_{t \in \R}$ is a strongly continuous one-parameter 
group on the real Hilbert space $V$, then we can recover the corresponding 
embedding of $V$ as a standard subspace as follows. 
Let $V_0 := V^U$ be the subspace of $U$-fixed vectors and $V_1 := V_0^\bot$. 
Then $U_t = e^{tD}$ with a skew-symmetric infinitesimal generator 
$D = -D^\top$ satisfying $V_0 = \ker D$. On $V_1$ we have the 
polar decomposition $D = I|D|$, where $I$ is a complex structure 
and $|D| = \sqrt{-D^2}$. We now consider the bounded skew-symmetric 
operator $C$ on $V$ defined by $C\res_{V_0} = 0$ and 
\[ C\res_{V_1} = I \frac{\1 - e^{-|D|}}{\1 + e^{-|D|}}.\] 
Then $h(v,w) := (v,w) + i (v,Cw)$ 
leads to an embedding of $V$ as a standard subspace $V \subeq \cH$ 
as in Remark~\ref{rem:2.3}(a). The operator $D$ can be recovered directly from 
$C$ by $D\res_{V_0} = 0$ and 
\[ D\res_{V_1} = I \log\Big(\frac{\1 + |C_1|}{\1 - |C_1|}\Big) \] 
(cf.~\cite[Rem.~4.3]{NO16}, where different sign conventions are used). 

The orthogonal one-parameter group $(U_t)_{t \in \R}$ on $V$ is trivial 
if and only if $D = 0$, which corresponds to $\Delta = \1$, 
resp., to $C = 0$, resp., to $V = \cH^J$ (Lemma~\ref{lem:stand-factorial}(iv)).

\subsection{Half-sided modular inclusions of standard subspaces} 
\mlabel{subsec:3.3}

We have seen above that standard subspaces $V \subeq \cH$ 
are in one-to-one correspondence with antiunitary 
representations $U^V \: \R^\times \to \AU(\cH)$ (Proposition~\ref{prop:3.2}). 
In this subsection we shall see how certain inclusions of 
standard subspaces can be related to antiunitary positive energy representations 
of $\Aff(\R)$ (cf.~Section~\ref{subsubsec:2.5.1}). 
Here the positive energy condition for the translation group 
turns out to be the crucial link 
between the inclusion order on $\Stand(\cH)$ and the affine geometry 
of the real line. 

There are two ways to approach inclusions of standard subspaces. 
One is to consider the interaction of a unitary one-parameter group 
with a standard subspace, which leads to the concept of a Borchers pairs 
and the other considers the modular groups of 
two standard subspaces and leads to the concept of a half-sided modular inclusion. 
These perspectives have been introduced by  Borchers (\cite{Bo92}) 
and Wiesbrock (\cite{Wi93}), respectively, in the context of 
von Neumann algebras (see \S\ref{subsec:4.2} for the translation 
to standard subspaces and \cite{Lo08} for the results in the context of standard subspaces).

\begin{definition} \mlabel{def:3.8} (a) Let $(U_t)_{t \in \R}$ be a continuous unitary 
one-parameter group on $\cH$ and $V \subeq \cH$ be a standard subspace. 
We call $(U,V)$ a {\it (positive/negative) Borchers pair} if 
$U_t V \subeq V$ holds for $t \geq 0$ and 
$U_t = e^{itP}$ with $\pm P \geq 0$. 

(b) 
We call an inclusion $K \subeq H$ of standard subspaces of 
$\cH$ a {\it $\pm$half-sided modular inclusion} if 
\[ \Delta_H^{-it} K \subeq K \quad \mbox{ for } \quad \pm t \geq 0.\] 
\end{definition}

\begin{remark}\mlabel{rem:3.13e}
The inclusion $K \subeq H$ is positive half-sided modular 
if and only if the inclusion $H' \subeq K'$ is negative half-sided modular 
(\cite[Cor.~3.23]{Lo08}). 
\end{remark}

The following theorem provides a passage from Borchers pairs 
to  antiunitary representations of $\Aff(\R)$ 
(\cite[Thm.~3.2]{BGL02}, \cite[Thm.~3.15]{Lo08}). 

\begin{theorem}[Borchers' Theorem---one particle case]
  \mlabel{thm:bor-stand} 
If $(U,V)$ is a positive/negative Borchers pair, 
then 
\[ U^V(a) U(b) U^V(a)^{-1} = U(a^{\pm 1} b) 
\quad \mbox{ for } \quad  a \in \R^\times, b \in\R,\] 
i.e., we obtain an antiunitary 
positive energy representation $(\tilde U,\cH)$ of $\Aff(\R)$ by 
$\tilde U_{(b,a)} = U(b) U^V(a)$. 
\end{theorem}

We are now ready to explain how inclusions of 
standard subspaces are related to antiunitary representations of 
$\Aff(\R)$. The following result contains in particular a converse of Borchers' Theorem. 
For its formulation, we 
recall the one-to-one correspondence between standard subspaces and antiunitary 
representations of $\R^\times$ from Proposition~\ref{prop:3.2}. 

\begin{theorem}{\rm(Antiunitary positive energy 
representations of $\Aff(\R)$ and standard subspaces)} 
\mlabel{thm:3.8}  
Let $(U,\cH)$ be an antiunitary representation 
of $\Aff(\R)$. For each  $x \in \R$, we consider the homomorphism 
\[ \gamma_x \: \R^\times \to \Aff(\R), \qquad 
\gamma_x(s)  := (x,1)(0, s)(-x,1) = ((1-s)x, s)\] 
whose range is the stabilizer group $\Aff(\R)_x$ 
and the corresponding family $(V_x)_{x \in \R}$ 
of standard subspaces determined by $U^{V_x} = U \circ \gamma_x$. 
Then the following assertions hold: 
\begin{itemize}
\item[\rm(i)] $U_{(t,s)} V_x = V_{t + sx}$ and $U_{(t,-s)} V_x = V_{t-sx}'$ for 
$t,x \in \R, s > 0.$ 
\item[\rm(ii)] The following are equivalent: 
  \begin{itemize}
  \item[\rm(a)] $U$ is a positive energy representation. 
  \item[\rm(b)] $V_s \subeq V_0$ for $s \geq 0$. 
  \item[\rm(c)] $V_s \subeq V_t$ for $s \geq t$. 
  \item[\rm(d)] $(W, V_0)$ with $W_t := U_{(t,1)}$ is a positive Borchers pair. 
  \item[\rm(e)] $V_1 \subeq V_0$ is a +half-sided modular inclusion. 
  \end{itemize}
\item[\rm(iii)] $V_x = V_0$ for every $x \in \R$ is equivalent to 
$U_{(b,1)} = \1$ for every $b \in \R$. 
\item[\rm(iv)] $V_\infty := \bigcap_{t \in \R} V_t = \{ v \in V_0 \: (\forall b \in \R)\ 
U_{(b,1)}v = v\}$ is the fixed point space for the translations. 
\item[\rm(v)] $V_0 \cap V_0' 
= \cH^{\Aff(\R)} =  \{ v \in \cH \: (\forall g\in \Aff(\R))\ U_g v = v\}$. 
\end{itemize}
\end{theorem}

\begin{proof} (i) follows from $(t,s)\gamma_x (t,s)^{-1}  = \gamma_{t+sx}$, 
$U_{(0,-1)} V_0 = V_0'$ and $V_x' = U_{(x,1)}V_0'$. 

(ii) (a) $\Leftrightarrow$ (b): For $W(s) := U_{(s,1)}$ we have 
\[ \Delta_{V_0}^{-it/2\pi} W(s)  \Delta_{V_0}^{it/2\pi} 
= U_{(0,e^{t})} W(s) U_{(0,e^{-t})}  = W(e^{t}s),\] 
so that the assertion follows from 
the converse of Borchers' Theorem \cite[Thm.~3.17]{Lo08}. 

(b) $\Leftrightarrow$ (c) follows from $V_t = W(t) V_0$ for $t \in \R$. 

By definition,  (d) is equivalent to (a) and (b). 

(b) $\Leftrightarrow$ (e): From (b) we derive (e) by 
\[ U^{V_0}_{e^t} V_1 
= U_{(0,e^t)} V_1 = V_{e^t} 
= U_{(1,1)} V_{e^t-1}\subeq U_{(1,1)} V_0 = V_1. \] 
From (e) we obtain, conversely, for $t \geq 0$ 
\[ U_{(1,1)} V_0 = V_1 \supeq U^{V_0}_{e^t} V_1 = V_{e^t} = U_{(1,1)} V_{e^t-1},\] 
and thus $V_{e^t-1} \subeq V_0$, which implies (b). 

(iii) If $W(x) := U_{(x,1)} = \1$ for every $x \in \R$, then $V_x = W(x)V_0 = V_0$. 
If, conversely, $W(x) V_0 = V_x = V_0$ for every $x \in \R$, 
then every $W(x)$ commutes with $\Delta_{V_0}$ and $J_{V_0}$, 
so that Theorem~\ref{thm:bor-stand} yields $W(x) = \1$ for every $x \in \R$. 

(iv) By (i), the closed real subspace $V_\infty$ of $V_0$ 
is invariant under $\Aff(\R)_0$. Hence Lemma~\ref{lem:split} implies 
that it is a direct summand of the standard subspace~$V_0$ and therefore 
also invariant under $J := U_{(0,-1)}$. Now (iii) implies that 
the translation group fixes $V_\infty$ pointwise. Conversely, 
every fixed vector $v\in V_0$ of the translations is contained in 
each subspace $V_x = W(x)V_0$, hence also in $V_\infty$. 

(v)  From Lemma~\ref{lem:stand-factorial}(iv) we know that 
$V_0 \cap V_0'$ is the space of fixed vector for the dilation group 
$(U_{(0,a)})_{a \in \R^\times}$. Proposition~\ref{prop:2.x}(c) 
implies the translations also act trivially on this space. This proves (v).
\end{proof}

\begin{remark} \mlabel{rem:3.11} (a) If the momentum operator $P$ from 
Theorem~\ref{thm:bor-stand} is strictly positive,  
then the space of fixed points for the dilation subgroup is trivial 
and Theorem~\ref{thm:3.8}(iv) implies $V_\infty = \{0\}$. 

(b) If $(U,V)$ is a Borchers pair for which 
$U_t V \subeq V$ for all $t \in \R$, 
then $U_t V = V$ for every $t \in \R$ because $V = U_0 V = U_t U_{-t}V\subeq U_t V$. 
Now Theorem~\ref{thm:3.8}(iii) entails $U_t = \1$ for every $t \in \R$. 
Therefore non-trivial representations of the translation group 
lead to proper inclusions. 

(c) For a Borchers pair $(U,V)$, the operators $(U_t)_{t \geq 0}$,  and 
the modular operators $(\Delta^{it})_{t \in \R}$ act by isometries on the real 
Hilbert space $V$, so that we obtain a representation of the semigroup
$[0,\infty) \rtimes \R^\times_+$ by isometries on $V$. 
In this sense we may consider Borchers' Theorem~\ref{thm:bor-stand} 
as a higher dimensional analog of the Lax--Phillips Theorem which 
provides a normal form for one-parameter semigroups of 
isometries on real Hilbert spaces as translations acting on spaces like 
$L^2(\R^+,\cK)$, where $\cK$ is a Hilbert space 
(cf.~Remark~\ref{rem:3.17}(b) and \cite{NO15}). 

The connection with the Lax--Phillips Theorem can also be made more direct as 
follows. The subspace $H := U_1 V$ is invariant under the 
modular automorphisms $(\Delta^{-it})_{t \geq 0}$. 
More precisely, 
$\Delta^{-it} H = \Delta^{-it} U_1 V = U_{e^{2\pi t}} V 
= V_{e^{2\pi t}}\subeq V_1 = H$ for $t \geq 0$, in the notation of Theorem~\ref{thm:3.8}. 
This shows that $\bigcup_{t \in \R} \Delta^{-it} H$ is dense in $V$ and 
that $\bigcap_{t > 0} \Delta^{-it} H = V_\infty$ is the fixed point 
set for $(U_t)_{t\in \R}$ in $V$  (Theorem~\ref{thm:3.8}(iv)). 
Assuming that $U$ has no non-zero fixed vectors (as in (a) above), we obtain 
$V_\infty = \{0\}$. This means that the subspace 
$H \subeq V$ is outgoing in the sense of Lax--Phillips 
for the orthogonal one-parameter group $(\Delta^{-it})_{t \in \R}$. 
\end{remark}

The group $\Aff(\R)$ is generated by translations and dilations, 
which is the structure underlying Borchers pairs. 
But we can also generate 
it by the subgroups  $\gamma_0(\R^\times)$ and $\gamma_1(\R^\times)$. 
For every antiunitary representation $(U,\cH)$ of $\Aff(\R)$, the corresponding 
modular objects lead to two standard subspaces $V_0$ and $V_1$ 
and we have already seen above that $V_1 \subeq V_0$ is a positive 
half-sided modular inclusion if $U$ is of positive energy. 
The following theorem provides a converse 
(see \cite[Thm.~3.21]{Lo08}). 

\begin{theorem}[Wiesbrock Theorem---one particle case]  \mlabel{thm:wiesbrock} 
An inclusion $K \subeq H$ of standard subspaces is positive 
half-sided modular if and only if there exists an antiunitary 
positive energy representation $(U,\cH)$ of $\Aff(\R)$ with 
$K = V_1$ and $H = V_0$. 
\end{theorem}

\begin{proof} In view of Theorem~\ref{thm:3.8}(ii)(e), it remains to 
show the existence of $U$ if the inclusion is +half-sided modular. 
In view of \cite[Thm.~3.21]{Lo08}, there exists a unitary positive energy 
representation $(U,\cH)$ of the connected affine group $\Aff(\R)_0$ 
such that $U_{\gamma_0(e^t)} = U^H(e^t)$ and $U_{\gamma_1(e^t)} = U^K(e^t)$ for $t \in \R$. 
Further, the translation unitaries 
$W_t := U_{(t,1)}$ satisfy $W_1 H = K$ and $W_t H \subeq H$ for $t \geq 0$. 
Therefore $(W,H)$ is a Borchers pair, and thus 
$\tilde U_{(b,a)} := W_b U^H_a$ defined an extension of $U$ to an antiunitary 
representation of $\Aff(\R)$ (Theorem~\ref{thm:bor-stand}). 
The corresponding subspaces are 
$V_0 = H$ by construction, and $V_1 = W_1 V_0 = W_1 H = K$.  
\end{proof}

\begin{exs} \mlabel{ex:hardy}
Below we provide an explicit description of a 
positive Borchers pair in a concrete model of the irreducible antiunitary 
positive energy representation of $\Aff(\R)$ 
(cf.~Theorem~\ref{thm:bor-stand} and \cite[\S4]{LL14}). 
A slight variation of \eqref{eq:pics-s1} leads to the 
antiunitary representation of $\Aff(\R)$ on 
$L^2\big(\R_+, \frac{dp}{p}\big)$ by 
\[ (U_{(b,e^t)}\psi)(p) = e^{ibp} \psi(e^tp), \qquad 
(U_{(0,-1)}\psi)(p) = \oline{\psi(p)}.\] 
Transforming it with the unitary operator 
$\Gamma \: L^2(\R_+, \frac{dp}{p}) \to L^2(\R,d\theta), 
\Gamma(\psi)(\theta) = \psi(e^\theta),$ transforms it into 
the representation 
\begin{equation}
  \label{eq:strip}
(U_{(b,e^t)}\psi)(\theta) := e^{ib e^\theta} \psi(\theta+t), \qquad 
(U_{(0,-1)}\psi)(\theta) := \oline{\psi(\theta)}.
\end{equation}
On the strip $\cS_\pi := \{ z \in \C \: 0 < \Im z < \pi\}$ we have the Hardy space 
\begin{equation}
  \label{eq:hardy}
 \cH^2(\cS_\pi) := \Big\{ \psi \in \cO(\cS_\pi) \: \sup_{0 < \lambda < \pi} 
\int_\R |\psi(\theta + i \lambda)|^2\, d\theta < \infty\Big\},
\end{equation}
and in these terms, the standard subspace $V_0$ 
corresponding to $\gamma_0(t) = (0,t)$ is given by 
\[ V_0 = \{ \psi \in \cH^2(\cS_\pi) \: 
(\forall z \in \cS_\pi)\ \oline{\psi(i \pi + \oline z)} = \psi(z)\}. \] 
On the strip $\cS_\pi$, the functions $B(z) := e^{ib e^z}$ satisfy 
\[ |B(x+ iy)| = e^{-b \Im(e^{x+iy})} = e^{-b e^x \sin y}\leq 1\quad \mbox{ because }  \quad 
\sin y \geq 0\]  
and $\oline{B(i\pi + \oline z)} = B(z).$ 
This shows hat, for $b\geq 0$, multiplication with $B$ 
defines an isometry of the Hardy space $\cH^2(\cS_\pi)$ and also of the real subspace 
$V_0$ into itself (cf.\ Remark~\ref{rem:3.11}(c)). 
One can show that all unitary operators commuting with the representation 
of the one-parameter group $(U_{(b,1)})_{b \in \R}$ and mapping $V_0$ into itself 
are multiplications with bounded holomorphic functions $\phi$ on $\cS_\pi$ 
satisfying $\phi(i\pi + \oline z) = \oline{\phi(z)}$ and whose boundary 
values in $L^\infty(\R,\C)$ satisfy $|\phi(x)| = 1$ for almost every $x \in \R$ 
(cf.~Remark~\ref{rem:inner}(c)). 

For explicit descriptions of standard 
subspaces related to free fields, we refer to \cite[p.~422ff]{FG89}.
\end{exs}

\begin{remark} \mlabel{rem:3.17}
(a) If $K \subeq H$ is a proper positive half-sided modular inclusion and 
$V$ is a closed real subspace with $K \subeq V \subeq H$, then $V$ is clearly standard. 
However, neither the inclusion $K \subeq V$ nor the inclusion $V \subeq H$ has to be 
half-sided modular. In fact, the existence of the unitary one-parameter group 
$(U_t)_{t \in \R}$ with $U_1 H = K$ implies that all the inclusions 
$U_t H \subeq U_s H$ for $0 \leq s < t \leq 1$ are proper 
(Theorem~\ref{thm:3.8}(ii)). Therefore $K$ has infinite 
codimension in $H$. So subspaces $V$ for which 
$V/K$ or $H/V$ is finite dimensional yield counterexamples. 

(b) Let $V$ be a standard subspace. We write $\hsm_+(V)$ for the set of all 
standard subspaces $H \subeq V$ for which the inclusion $H\subeq V$ is 
positive half-sided modular. To obtain a description of this set, one can proceed as 
follows. First we can split off the maximal direct summand 
$H_1 := \bigcap_{t \in \R} \Delta_V^{it} H$ of $V$ contained in $H$ 
(Lemma~\ref{lem:split}). This leaves us with the situation where $H_1 = \{0\}$. 

Decomposition of the corresponding antiunitary representation $(U,\cH)$ of $\Aff(\R)$ 
(Theorem~\ref{thm:wiesbrock}) into a subspaces $\cH^0$ on which the translations act trivially 
and an orthogonal space $\cH^+$ on which the representation is of strict 
positive energy, we accordingly obtain 
the direct sum $V = V^0 \oplus V^+$ of standard subspaces and 
$H = V^0 \oplus H^+$. Hence our assumption implies $V^0 = \{0\}$ and $\cH = \cH^+$. 
Now \cite[Thm.~2.8]{Lo08} implies that $U$ is a multiple of the unique irreducible 
positive energy representation of $\Aff(\R)$, so that we may assume that 
\[ \cH = L^2(\R_+, \cK) \quad \mbox{ and }\quad 
(U_{(t,e^s)}f)(x) = e^{itx} e^{s/2} f(e^sx), \qquad 
U_{(0,-1)} f = J_K f,\] 
where $J_K$ is a conjugation on $\cK$ (see \S\ref{subsubsec:2.5.1}). 

As all antiunitary rerpresentations of $\Aff(\R)$ with strictly positive 
energy and the same multiplicity are equivalent, we obtain all such standard subspaces 
$H$ by applying elements of the group 
\begin{align*}
 K &:= \{ U \in \U(\cH) \: UV = V \} 
= \{ U \in \U(\cH) \: (\forall a \in \R^\times)\ U U_{(0,a)} U^{-1} = U_{(0,a)}\} \\
&\cong \{ U \in \OO(V) \: (\forall a \in \R^\times)\ 
U \Delta_V^{it}\res_V = \Delta_V^{it}U\}.
\end{align*}
If $\cK = \C$, then the representation of $\R^\times$ on $\cH$ is (by Fourier transform) 
equivalent 
to the representation of $\R$ on $L^2(\R,\cK)$ by 
\[ (V_{e^x} \xi)(p) = e^{ixp} \xi(p) \quad \mbox{ and } \quad 
(V_{-1} \xi)(p) = \oline{\xi(-p)}.\]  
Therefore any unitary operator $M$ on $\cH$ commuting with $V_{\R^\times}$ is of the form 
$(M\xi)(p) = m(p)\xi(p)$, where $m \: \R \to \T$ is a measurable function satisfying
$m(-p) = \oline{m(p)}$. It would be interesting to see how this relates to the 
inner functions corresponding to endomorphisms of one-dimensional standard pairs 
(see Remark~\ref{rem:inner}(c)). 
\end{remark}

Combining the preceding results with the fact that 
the infinite dimensional irreducible positive energy representation of 
$\Aff(\R)_0$ extends to an antiunitary positive energy representation of 
$\PGL_2(\R)$ with lowest weight $1$ 
(Theorem~\ref{thm:1.4} and Corollary~\ref{cor:psl2-restrict}), we obtain 
(\cite[Cor.~4.15]{Lo08}): 

\begin{theorem} 
There exists a one-to-one correspondence between 
\begin{itemize}
\item[\rm(i)] Positive half-sided modular inclusions $K \subeq H$. 
\item[\rm(ii)] Antiunitary positive energy representations of $\Aff(\R)$. 
\item[\rm(iii)] Positive Borchers pairs $(V,U)$. 
\item[\rm(iv)] Unitary representations of $\PSL_2(\R)$ which are direct sums of 
representations with lowest weights $0$ or $1$. 
\end{itemize}
\end{theorem}

An important aspect of the last item in 
the preceding theorem is that it leads to a considerable enrichment 
of the geometry. Starting with a positive half-sided modular inclusion $K \subeq H$,  
we obtain an antiunitary representation of $\PGL_2(\R)$. 
Accordingly, for every interval $I \subeq \bS^1$, the corresponding 
homomorphism $\gamma^I \: \R^\times \to \PGL_2(\R)$ 
(Example~\ref{ex:proj-grp}) determines 
a standard subspace $V_I$, whereas the representation of 
$\Aff(\R)$ only leads to standard subspaces $V_I$ indexed by the 
open half-lines $I \subeq \R$. 

The following theorem is another result in this direction. 
It relates pairs of half-sided modular inclusions 
via the corresponding antiunitary representations of $\Aff(\R)$ 
to representations 
of the two-dimensional Poincar\'e group $P(2)_+$, resp., $\PGL_2(\R)$. 

\begin{theorem} \mlabel{thm:wies2-standard} {\rm(a)}
Let $H_1 \subeq V$ be a $-$half sided modular inclusion and 
$H_2 \subeq V$ be a $+$half sided modular inclusion such that 
\begin{equation}
  \label{eq:J-rel}
 J_{H_1}J_{H_2} = J_VJ_{H_2} J_{H_1} J_V.  
\end{equation}
Then the corresponding three modular one-parameter groups 
combine to a faithful continuous antiunitary representation 
of the proper Poincar\'e group $P(2)_+ \cong \R^{1,1} \rtimes \SO_{1,1}(\R)$. 

{\rm(b)} Let $H,V$ be standard subspaces of $\cH$ 
such that $H \cap V \subeq V$ and $H\cap V \subeq H$ are 
$-$, resp., $+$half-sided modular inclusions satisfying 
$J_H V = V.$ 
Then the corresponding three representation 
$U^H, U^V$ and $U^{H \cap V}$ 
generate a faithful antiunitary positive energy representation of $\PGL_2(\R)$. 
\end{theorem}

\begin{proof} (a) The version for von Neumann algebras 
is contained in \cite[Lemma~10]{Wi98} and we shall see in 
\S \ref{subsec:4.2} below how the present version follows from 
this one. Here are some comments on the proof. 
Clearly, the two half-sided modular inclusion defines 
two antiunitary representations $U^{1/2}$ 
of $\Aff(\R) \cong \R \rtimes \R^\times$ 
that coincide on the subgroup of dilations. Now the main point is to verify 
that the corresponding images of the translation groups commute. 

That \eqref{eq:J-rel} is necessary can be seen as follows. 
If we have a unitary representation 
of $P(2)_0$ as required, then 
\[ U_{(b_1, b_2,e^t)} := U^1_{b_1} U^2_{b_2} \Delta_V^{-it/2\pi} 
\quad \mbox{ and } \quad U_{(0,0,-1)} := J_V\] 
defines an extension to an antiunitary representation of~$P(2)_+$. 
Here the modular conjugations 
\[ J_{H_1} = U_{(2,0,-1)} \quad \mbox{ and } \quad   J_{H_2} = U_{(0,2,-1)} \] 
satisfy \eqref{eq:J-rel} because 
$(0,0,-1) (0,2,-1) (2,0,-1) (0,0,-1) = (2,-2,1)$ 
holds in $P(2)$. 

(b) The von Neumann version is \cite[Thm.~3]{Wi93b} 
(see also \cite[Lemma~2]{Wi93c}). 
\end{proof}

\subsection{Half-sided modular intersections} 
\mlabel{subsec:modint}

\begin{definition}
We consider two standard subspaces $H_1, H_2 \subeq \cH$ and their 
modular objects $(\Delta_{H_j}, J_{H_j})_{j=1,2}$. We say that 
the pair $(H_1, H_2)$ has a {\it $\pm$modular intersection} if 
the following two conditions are satisfied
\begin{itemize}
\item[\rm(MI1)] The intersection $H_1 \cap H_2$ is a standard subspace and the 
inclusions $H_1 \cap H_2 \subeq H_j$, $j =1,2$, are 
$\pm$half-sided modular. 
\item[\rm(MI2)] The strong limit $S := \lim_{t \to \pm\infty} 
\Delta_{H_1}^{it} \Delta_{H_2}^{-it}$ (which always exists by Remark~\ref{rem:mi-1} below) 
satisfies $J_{H_1} S J_{H_1} = S^{-1}.$ 
\end{itemize}
\end{definition}

\begin{remark} \mlabel{rem:mi-1} 
(a) In $\Aff(\R)$ the two multiplicative one-parameter groups 
$\gamma_0(r) := (0,r)$ and $\gamma_1(r) := (1-r,r)$ (stabilizing the points 
$0$ and $1$, resp.) satisfy 
\[ \gamma_0(r)\gamma_1(r^{-1}) = (0,r)(1-r^{-1}, r^{-1}) = (r-1,1),\] 
so that $\lim_{r \to 0} \gamma_0(r)\gamma_1(r^{-1}) = (-1,1)$ 
exists. 
As a consequence, for every continuous unitary representation $(U,\cH)$ of $\Aff(\R)_0$, the limit 
\begin{equation}
  \label{eq:limit2}
\lim_{r \to 0} U_{\gamma_0(r)}U_{\gamma_1(r)}^{-1} = U_{(-1,1)}  
\end{equation}
exists. 

(b) Suppose that the $+$-variant of 
(MI1) is satisfied and let $(U^1, \cH)$, $(U^2, \cH)$ 
be the corresponding positive energy representations of 
$\Aff(\R)$ satisfying 
\[ H_1 \cap H_2 = V^1_1 = V^2_1,  \qquad 
  H_1 = V^1_0 \quad \mbox{ and } \quad 
  H_2 = V^2_0,\] 
where $(V^j_x)_{x \in \R}$ are the corresponding families of standard subspaces 
(Theorem~\ref{thm:3.8}). We write $W^j(t) := U^j_{(t,1)}$ for the representations of the 
translation group. 
Then (a) implies that 
\[ W^1(-1) = U^1_{(-1,1)} 
= \lim_{t \to \infty} U^1_{\gamma_0(e^{-t})}U^1_{\gamma_1(e^t)} 
= \lim_{t \to \infty} \Delta_{H_1}^{it/2\pi} \Delta_{H_1 \cap H_2}^{-it/2\pi} 
= \lim_{t \to \infty} \Delta_{H_1}^{it} \Delta_{H_1 \cap H_2}^{-it}\]
and likewise 
$W^2(-1) = \lim_{t \to \infty} \Delta_{H_2}^{it} \Delta_{H_1 \cap H_2}^{-it}.$ 
This leads to 
\begin{equation}
  \label{eq:limit3}
S := \lim_{t \to \infty} \Delta_{H_1}^{it} \Delta_{H_2}^{-it} 
=  W^1(-1) W^2(1), 
\end{equation}
so that the limit in (MI2) exists whenever (MI1) is satisfied. 
From the relation 
\begin{equation}
  \label{eq:s-rel}
 S H_2 
= W^1(-1) W^2(1) H_2 
= W^1(-1) (H_1 \cap H_2) = H_1
\end{equation}
it follows that $S J_{H_2} S^{-1} = J_{H_1}$, i.e., 
\begin{equation}
  \label{eq:s12-rel}
S J_{H_2} = J_{H_1} S.
\end{equation}
As condition (MI2) means that $J_{H_1} S$ is an involution, and this is equivalent 
to $S J_{H_1} = S(J_{H_1} S)S^{-1}$ being an involution, \eqref{eq:s12-rel} 
shows that (MI2) is equivalent to the relation ${J_{H_2} S J_{H_2} = S^{-1}}$. 

(c) If the negative variant of (MI1) is satisfied, 
then $H_j' \subeq (H_1 \cap H_2)'$ are $+$half-sided modular inclusions 
by Remark~\ref{rem:3.13e}. Let $(U^1, \cH)$, $(U^2, \cH)$ 
be the corresponding positive energy representations of 
$\Aff(\R)$ satisfying 
\[ (H_1 \cap H_2)' = V^1_0 = V^2_0,  \qquad 
  H_1' = V^1_1 \quad \mbox{ and } \quad 
  H_2' = V^2_1\] 
(Theorem~\ref{thm:3.8}). With the same notation as in (b), 
we obtain 
\[ W^1(1) 
= \lim_{t \to \infty} U^1_{\gamma_1(e^{-t})}  U^1_{\gamma_0(e^{t})}
= \lim_{t \to -\infty} \Delta_{H_1'}^{-it/2\pi} \Delta_{(H_1 \cap H_2)'}^{it/2\pi} 
= \lim_{t \to -\infty} \Delta_{H_1}^{it} \Delta_{H_1 \cap H_2}^{-it} \] 
and likewise 
$W^2(1) = \lim_{t \to -\infty} \Delta_{H_2}^{it} \Delta_{H_1 \cap H_2}^{-it}.$ 
This leads to 
\begin{equation}
  \label{eq:limit3b}
S := \lim_{t \to -\infty} \Delta_{H_1}^{it} \Delta_{H_2}^{-it} =  W^1(1) W^2(-1), 
\end{equation}
so that the limit in (MI2) exists. Here $S H_2' = H_1'$ shows that 
(MI2) is equivalent to ${J_{H_2} S J_{H_2} = S^{-1}}$. 
\end{remark} 

The following theorem extends Wiesbrock's Theorem~\ref{thm:wiesbrock} 
from half-sided modular inclusions to general modular intersections. 

\begin{theorem}[Wiesbrock's Theorem for modular intersections---one particle version] 
\mlabel{thm:wies-modinter}
For a pair $(H_1, H_2)$ of standard subspaces, the following 
assertions hold: 
  \begin{itemize}
  \item[\rm(a)] $(H_1, H_2)$ has a $+$modular intersection 
if and only if there exists an antiunitary representation 
$(U,\cH)$ of $\Aff(\R)$ such that the corresponding family of standard subspaces 
$(V_x)_{x \in\R}$ from {\rm Theorem~\ref{thm:wiesbrock}} satisfies 
$V_0 = H_1$ and $V_1 = H_2$. 
  \item[\rm(b)] $(H_1, H_2)$ has a $-$modular intersection 
if and only if there exists an antiunitary representation 
$(U,\cH)$ of $\Aff(\R)$ such that $V_0 = H_1'$ and $V_1 = H_2'$. 
  \end{itemize}
\end{theorem}

\begin{proof} (a) Suppose first that $(H_1, H_2)$ has the $+$modular 
intersection property. In the context of Remark~\ref{rem:mi-1}(b) we then have 
\[ J_{H_1} J_{H_2} 
=  J_{H_1} J_{H_1 \cap H_2} J_{H_1 \cap H_2} J_{H_2} 
=  U^1_{(0,-1)} U^1_{(2,-1)}  U^2_{(2,-1)}  U^2_{(0,-1)} 
=  W^1(-2) W^2(2), \] 
and this operator commutes with $S$ by (MI2). From these relations Wiesbrock derives in 
\cite{Wi97} that the two one-parameter groups $W^1$ and $W^2$ commute, so that 
$W(t) := W^1(t) W^2(-t)$ defines a unitary one-parameter group and 
$U_{(b,a)} := W(b)U^1_{\gamma_0(a)}$ 
defines an antiunitary representation of $\Aff(\R)$ 
for which the corresponding standard subspaces $(V_x)_{x \in \R}$ 
satisfy $V_0= H_1$ and $V_{1} = W(1)V_0= S^{-1} H_1 = H_2.$. 

Suppose, conversely, that $(U,\cH)$ is an antiunitary representation 
of $\Aff(\R)$ and let $(V_x)_{x \in \R}$ be the corresponding family of 
standard subspaces. We decompose $(U,\cH)$ as a direct sum 
\[ (U,\cH) = (U^+,\cH^+) \oplus (U^0,\cH^0) \oplus(U^-,\cH^-), \] 
where the representation $U^+$ has strictly positive energy, 
$U^-$ has strictly negative energy and the translation group acts 
trivially on $\cH^0$. Then the subspaces $V_x$ decompose accordingly 
as orthogonal direct sums $V_x = V_x^+ \oplus V_x^0 \oplus V_x^-$, 
where $V_x^0 = V^0_0$ does not depend on~$x$. 

Theorem~\ref{thm:3.8} now implies that $V_x^\pm \subeq V_y^\pm$ for $\pm (x-y) \geq 0$. 
To see that $V_0$ and $V_1$ have a $+$modular intersection, we first observe that 
\[ V_0 \cap V_1 
= (V_0^+ \cap V_1^+) \oplus V_0^0 \oplus (V_0^- \cap V_1^-) 
= V_1^+ \oplus V_0^0 \oplus V_0^-\] 
is an orthogonal direct sum of three standard subspaces, hence a standard subspace. 
Its invariance under the modular operators 
$\Delta_{V_0}^{-it/2\pi} = U_{\gamma_0(e^t)}$ for $t \geq 0$ follows from the 
invariance of $V_1^+$ under $U^+_{\gamma_0(e^t)}$ and 
the invariance of $V_0^+$ and $V_0^0$ under the corresponding modular group. 

For the invariance of $V_0 \cap V_1$ under $\Delta_{V_1}^{-it/2\pi} = U_{\gamma_1(e^t)} 
= U_{((1-e^t, e^t)}$, we likewise use that 
\[ U^-_{(1-e^t, e^t)} V^-_0 
= U^-_{(1-e^t, 1)} V^-_0 = V^-_{1 - e^t} \subeq V^-_0 \quad \mbox{ for } \quad t \geq 0.\] 
This shows that $(V_0, V_1)$ has a $+$modular intersection. 

(b) If $(H_1, H_2)$ is a $-$modular 
intersection, we likewise obtain with Remark~\ref{rem:mi-1}(c) 
that $U_{(b,a)} := W(b) U^1_{\gamma_1(a)}$ and $W(t) := W^1(t) W^2(-t)$  
define an antiunitary representation of $\Aff(\R)$ with 
$S = W(-1)$, $V_0 = V^1_1 = H_1'$ and $V_1 = W(1)V_0 = S^{-1}H_1' =  H_2'$. 

Suppose, conversely, that $(U,\cH)$ is an antiunitary representation 
of $\Aff(\R)$. We use the notation from (a). 
To see that $V_0'$ and $V_1'$ have a $-$modular intersection, we first observe that 
\[ V_0' \cap V_1' 
= ( (V_0^+)' \cap (V_1^+)') \oplus (V_0^0)' \oplus ((V_0^-)' \cap (V_1^-)') 
= (V_0^+)' \oplus V_0^0 \oplus (V_1^-)'\] 
to see that $V_0' \cap V_1'$ is standard. 
Its invariance under the modular operators 
$\Delta_{V_0'}^{it/2\pi} = \Delta_{V_0}^{-it/2\pi} = U_{\gamma_0(e^{t})}$ for $t \geq 0$ follows from the 
invariance of $(V_1^-)'$ under $U^-_{\gamma_0(e^t)}$ (Remark~\ref{rem:3.13e}) and 
the invariance of $V_0^+$ and $V_0^0$ under the corresponding modular group. 

For the invariance of $V_0' \cap V_1'$ under $\Delta_{V_1'}^{it/2\pi} = U_{\gamma_1(e^t)} 
= U_{(1-e^t, e^t)}$, we likewise use that 
\[ U^+_{(1-e^t, e^t)} (V^+_0)' 
= U^+_{(1-e^t, 1)} (V^+_0)' = (V^+_{1 - e^t})' \subeq (V^+_0)'\quad \mbox{ for }\quad t \geq 0.\] 
Therefore $(V_0', V_1')$ has a $-$modular intersection. 
\end{proof}

The key point of modular intersections is that they no longer require 
any spectral condition on the corresponding representations of $\Aff(\R)$. 
The preceding theorem shows that $\pm$modular intersections 
are characterized as pairs of standard subspaces that can be obtained from arbitrary 
antiunitary representations of $\Aff(\R)$. This is of particular relevance for 
representations of Lorentz groups $\SO_{1,d-1}(\R)$ which, for $d > 3$, never 
satisfy any positive energy condition. 

With the same method that we used to obtain Theorem~\ref{thm:wies2-standard},  
we now obtain by transcribing \cite[Thm.~6]{Wi97} from the context 
of von Neumann algebras the following theorem.  

\begin{theorem} \mlabel{thm:wies3-standard} 
Let $H_1, H_2, H_3$ be three standard subspaces such that 
$(H_1, H_2)$ and $(H_3, H_1')$ are $-$modular intersections 
and $(H_2, H_3)$ is a $+$modular intersection. 
Then the corresponding antiunitary representations $U^{H_j}$, $j =1,2,3$, of 
$\R^\times$ generate an anti-unitary representation of $\PGL_2(\R)$.
\end{theorem}

\begin{proof} For the proof we only has to observe that the 
group $G$ generated by the three modular one-parameter groups $(\Delta_{H_j}^{it})_{t \in \R}$, 
$j = 1,2,3$ is invariant under $J_{H_1}$. 
For the subgroup $G_{12}$ generated by the operators 
$\Delta_{H_1}^{it}$ and $\Delta_{H_2}^{is}$, this follows from Theorem~\ref{thm:wiesbrock}, 
and we likewise obtain the invariance of the subgroup 
$G_{13}$ generated by the operators 
$\Delta_{H_1}^{it}$ and $\Delta_{H_3}^{is}$. As $G$ is generated by $G_{12} \cup G_{13}$, 
the assertion follows. 
\end{proof}

\section{A glimpse of modular theory} 
\mlabel{sec:4}

We now recall some of the key features of
Tomita--Takesaki Theory. In \S \ref{subsec:4.2} we discuss 
the translation between pairs $(\cM,\Omega)$ 
of von Neumann algebras with cyclic separating vectors 
and standard subspaces~$V$. More specifically, we discuss 
this translation for half-sided modular inclusions in 
\S \ref{subsec:4.3}, and in \S \ref{subsec:4.4} 
we take a closer look at the space of modular conjugations 
of a von Neumann algebra.

\subsection{The Tomita--Takesaki Theorem} 
\mlabel{subsec:4.1}

Let $\cH$ be a Hilbert space and $\cM \subeq B(\cH)$ be a von Neumann algebra. 
We call a unit vector $\Omega \in \cH$ 
\begin{itemize}
\item {\it cyclic} if $\cM\Omega$ is dense in $\cH$.
\item {\it separating} if the map $\cM \to \cH, M \mapsto M\Omega$ is injective. 
\end{itemize}
It is easy to see that $\Omega$ is separating if and only if it is cyclic for the 
commutant $\cM'$. 

\begin{definition}
We write $\cs(\cM)$ for the set of cyclic and separating unit vectors for $\cM$. 
\end{definition}

\begin{theorem}[Tomita--Takesaki Theorem] \mlabel{thm:tom-tak}
Let $\cM \subeq B(\cH)$ be a von Neumann algebra and 
$\Omega \in\cH$ be a cyclic separating vector for $\cM$. 
Write $\cM_h := \{ M \in \cM \: M^* = M\}$ for the real subspace of hermitian 
elements in $\cM$. 
Then $V := \oline{\cM_h \Omega}$ is a standard subspace. 
The corresponding modular objects $(\Delta, J)$ satisfy
\begin{itemize}
\item[\rm(a)] $J \cM J = \cM'$ and $\Delta^{it} \cM \Delta^{-it} = \cM$ for $t \in \R$.
\item[\rm(b)] $J \Omega = \Omega$, $\Delta \Omega = \Omega$ and 
$\Delta^{it}\Omega = \Omega$ for all $t \in \R$.
\item[\rm(c)] For $M \in \cM \cap \cM'$, we have 
$JMJ = M^*$ and $\Delta^{it} M \Delta^{-it} = M$ for $t \in \R$. 
\end{itemize}
\end{theorem}

\begin{proof} 
We only show that $V$ is a standard subspace and refer to 
\cite[Thm.~2.5.14]{BR87} for the other assertions. 
Clearly, $V$ is a closed real subspace for which $V + i V$ is dense because
it contains $\cM_h \Omega + i \cM_h\Omega = \cM\Omega$.
The same holds for $W  :=  \oline{\cM_h' \Omega}$ because $\Omega$ is also 
cyclic for $\cM'$. For $M \in \cM_h$ and $M' \in \cM'_h$, we have 
\[ \la M\Omega, M'\Omega \ra = \la M'M\Omega, \Omega \ra = \la MM'\Omega, \Omega \ra 
= \la M'\Omega, M\Omega \ra \in \R,\]
which implies that $\omega(V,W) = \{0\}$. 
Therefore $V \cap iV$ is a complex subspace of $W^\bot = \{0\}$, hence trivial. 
Now the main point is to show that the modular objects $(\Delta, J)$ associated 
to $V$ satisfy (a)-(c). 
\end{proof}

The key point of the Tomita--Takesaki Theorem is that it provides 
for each cyclic separating vector $\Omega \in \cs(\cM)$ 
a pair $(\Delta, J)$ of modular objects. 
The modular operators $\Delta$ and their spectra are the key tool in 
the classification of factors and in the characterization 
of von Neumann algebras by their natural cones by A.~Connes \cite{Co73, Co74}. 
Here we emphasize that
$(\Delta, J)$ is encoded in an antiunitary representation $U^V$ of $\R^\times$. 

We first take a closer look at the antiunitary operators that come directly 
from $\cM$ and its commutant. The picture will be refined in \S\ref{subsec:4.4} below.  

\begin{ex} Let $\cM \subeq B(\cH)$ be a von Neumann algebra, 
$G_1 := \U(\cM) \times \U(\cM)$ and 
$\tau \in \Aut(G_1)$ be the flip automorphism. We consider the 
group $G := G_1 \rtimes \{\1,\tau\}$. 

If $\Omega$ is a cyclic separating vector for $\cM$ and 
$J$ the corresponding modular involution, then 
\[ U_{(g,h,\tau^\eps)} := g J h J J^\eps \] 
defines an antiunitary representation of the pair $(G, G_1)$. 

Any other conjugation $\tilde J$ on $\cH$ that we can use 
to extend the unitary representation $U\res_{G_1}$ is of the form 
$\tilde J = J  g$ for some central unitary element 
$g \in \U(\cM \cap \cM')$. 
\end{ex}

\begin{definition} \mlabel{def:symform}
A von Neumann algebra $\cM \subeq B(\cH)$ is said to be 
in {\it symmetric form} if there exists a conjugation $J$ on $\cH$ with 
\[ J\cM J = \cM' \quad \mbox{ and }\quad JZJ = Z^* \quad \mbox{ for } \quad Z \in \cM\cap \cM'.\]
\end{definition}

According to the Tomita--Takesaki Theorem, the existence of a 
cyclic separating vector 
implies that $\cM$ is in symmetric form. According to \cite[Thm.~III.4.5.6]{Bla06}, 
any two realizations of $\cM$ in symmetric form are unitarily equivalent.

Let $\fS_n(\cM)$ denote the set of {\it normal states} of the von Neumann algebra 
$\cM$. By the Gelfand--Naimark--Segal construction, any state $\omega$ corresponds 
to a cyclic normal representation 
$(\pi_\omega, \cH_\omega, \Omega_\omega)$ with 
$\omega(M) = \la \Omega_\omega, \pi_\omega(M) \Omega_\omega\ra$, 
which is uniquely determined up to unitary equivalence of cyclic representations. 
By construction $\Omega_\omega$ is cyclic, it leads to a faithful representation 
for which $\Omega_\omega$ is separating if and only if the state $\omega$ is 
{\it faithful}, i.e., $\omega(M^*M) > 0$ for any non-zero $M \in \cM$.

\begin{remark} (Existence of cyclic separating vectors) 
A von Neumann algebra $\cM$ possesses a faithful normal state if and only 
if it is {\it $\sigma$-finite} (also called {\it countably decomposable}) 
in the sense that every family of mutually orthogonal 
projections in $\cM$ is at most countable (\cite[Prop.~III.4.5.3]{Bla06}). 
This is always the case if $\cM$ can be realized 
on a separable Hilbert space, but not in general. Therefore one has to generalize 
the concept of a state to that of a {\it normal weight}. This is an 
additive positively homogeneous weakly lower semicontinuous 
functional $\omega \: \cM^+ \to [0,\infty]$ on the positive cone $\cM^+$ of $\cM$ 
that may also take the value $\infty$. A weight $\omega$ is called {\it semifinite} 
if the subset $\{M\in\cM_+\: \omega(M)<\infty\}$ 
generates $\cM$ as a von Neumann algebra. 
Every von Neumann algebra has a faithful normal semifinite weight 
(cf. \cite[III.2.2.26]{Bla06}) and the GNS construction as well as 
Tomita--Takesaki theory extend naturally to normal semifinite weights. 
In particular, any such weight leads to a symmetric form realization of~$\cM$. 
\end{remark}

\begin{ex} \mlabel{ex:3.2} 
(a) Let $\cH = L^2(X,\fS,\mu)$ for a $\sigma$-finite measure space 
$(X,\fS,\mu)$ and $\cM = L^\infty(X,\fS,\mu)$, 
acting on $\cH$ by multiplication operators. 
Then the normal states of $\cM$ are of the form 
$\omega_h(f) = \int_X fh\, d\mu$, where $0 \leq h$ satisfies 
$\int_X h\, d\mu = 1$. Such a state is faithful if and only if 
$h\not=0$ holds $\mu$-almost everywhere. Then $\Omega := \sqrt{h} \in \cH$ 
is a corresponding cyclic separating unit vector. 
From $S(f\Omega) = \oline f \Omega$, we obtain 
$S(f) = \oline f$, which is isometric and therefore 
$S = J$ and $\Delta = \1$. 

(b) Let $\cH = B_2(\cK)$ be the space of Hilbert--Schmidt operators on the complex 
separable Hilbert space $\cK$ and consider the von Neumann algebra 
$\cM = B(\cK)$ acting on $\cH$ by left multiplications. 
Then $\cM' \cong B(\cK)^{\rm op}$ acts by right multiplications. 
Normal states of $\cM$ are of the form 
$\omega_D(A) = \tr(AD)$, where $0 \leq D$ satisfies $\tr D = 1$. 
Such a state is faithful if and only if 
$\ker D = \{0\}$ (which requires $\cK$ to be separable), 
and then $\Omega := \sqrt{D} \in \cH$ 
is a cyclic separating unit vector. 
Then $S(M\Omega) = M^*\Omega = (\Omega M)^*$ implies that 
\[ JA = A^* \quad \mbox{ and } \quad 
\Delta(A) = \Omega^{2} A \Omega^{-2}= D A D^{-1} 
\quad \mbox{ for } \quad A \in B_2(\cK).\] 

(c) The prototypical pair $(\Delta, J)$ of a modular operator 
and a modular conjugation arises from the regular representation 
of a locally compact group $G$ on the Hilbert space $\cH = L^2(G, \mu_G)$ 
with respect to 
a left Haar measure~$\mu_G$. 
Here the modular operator is given by the multiplication 
\[ \Delta f = \Delta_G \cdot f,\] 
where $\Delta_G \: G \to \R^\times_+$ is the modular function of $G$ 
and the modular conjugation is given by 
\[ (Jf)(g) = \Delta_G(g)^{-\frac{1}{2}}  \oline{f(g^{-1})}.\] 
Accordingly, we have for $S = J \Delta^{1/2}$: 
\[ (Sf)(g) = \Delta_G(g)^{-1} \oline{f(g^{-1})} = f^*(g).\] 

The corresponding von Neumann algebra is the algebra $\cM \subeq B(L^2(G,\mu_G))$ 
generated by the left regular representation. 
If $M_f h =f * h$ is the left convolution with $f \in C_c(G)$, then the value of 
the corresponding normal weight $\omega$ on $\cM$ is given by 
$\omega(M_f) = f(e),$ so that $\omega$ corresponds to evaluation in~$e$, 
which is defined on a weakly dense subalgebra of~$\cM$.
\end{ex}

\subsection{Translating between standard subspaces and von Neumann 
algebras} 
\mlabel{subsec:4.2}

We have already seen that cyclic separating vectors of a von Neumann 
algebra $\cM$ lead to standard subspaces. In this subsection we explore 
some properties of this correspondence and describe 
how half-sided modular inclusions of standard subspaces translate into 
corresponding inclusions of von Neumann algebras. This correspondence 
shows that antiunitary representations of groups generated by 
modular one-parameter groups and conjugations from cyclic vectors 
of von Neumann algebras can already be studied in terms of 
standard subspaces and their inclusions, and all this can be encoded 
in antiunitary representations of pairs $(G,G_1)$, 
and homomorphisms $\R^\times \to G$, resp., $\Aff(\R) \to G$ 
(Corollary~\ref{cor:2.21} and Theorems~\ref{thm:3.8}, \ref{thm:wies-modinter}).

\begin{lemma} \mlabel{lem:4.14}
If $\cM \subeq B(\cH)$ is a von Neumann algebra and 
$\Omega \in \cH$ a separating vector for $\cM$, then we associate to every 
von Neumann subalgebra $\cN \subeq \cM$ the closed real subspace 
$V_\cN := \oline{\cN_h \Omega}$. This assignment is injective. 
\end{lemma}

Note that the subspace $V_\cN$ is standard if $\Omega$ is also cyclic for 
$\cN$. 

\begin{proof} (cf.~\cite[Prop.~3.24]{Lo08}) 
We have to show that 
$M \in \cM_h$ and $M \Omega \in V_\cN$ implies $M \in \cN$. 
First we find a sequence $A_n \in \cN$ such  that 
$A_n\Omega \to M\Omega$. For any $B \in \cM'$, this leads to 
$A_n B\Omega = BA_n \Omega \to B M\Omega = MB\Omega$, 
so that $A_n \to M$ holds pointwise on the dense subspace 
$\cD := \cM'\Omega$. Since the hermitian operators $A_n$ and $M$ are bounded, 
$\cD$ is a common core for all of them. 
With \cite[Thm.~VIII.25]{RS73} it now follows that $A_n \to M$ 
holds in the strong resolvent sense, i.e., that 
$(i\1 + A_n)^{-1} \to  (i\1 + M)^{-1}$ in the strong operator topology. 
This implies that $(i\1 + M)^{-1} \in \cN$, which entails $M \in \cN$. 
\end{proof}

The concept of a half-sided modular inclusion was originally conceived 
on the level of von Neumann algebras with cyclic separating 
vectors, where it takes the following form 
(\cite{Wi93,Wi97}). 

\begin{definition}  
Let $\Omega$ be a cyclic separating vector for the von Neumann algebra $\cM$ 
and $\cN \subeq \cM$ be a von Neumann subalgebra for which 
$\Omega$ is also cyclic. 
The triple $(\cM, \cN,\Omega)$ is called a {\it $\pm$half-sided modular inclusion} 
\begin{footnote}{Here we switched signs, compared to 
\cite{Bo97, Wi93}, to make the concept compatible with 
the sign convention in the context of standard 
subspaces \cite{Lo08}.} 
\end{footnote}
if 
\begin{equation}
  \label{eq:modular-inc}
\Delta_\cM^{-it} \cN \Delta_\cM^{it} \subeq \cN \quad \mbox{ for }\quad 
\pm t \geq 0.
\end{equation}
Note that $\Omega$ is also separating for $\cN$ because $\cN \subeq \cM$, 
so that we obtain two pairs of modular objects 
$(\Delta_\cM, J_\cM)$ and $(\Delta_\cN, J_\cN)$.
\end{definition}

\begin{lemma} Let $\cN \subeq \cM \subeq B(\cH)$ be von Neumann algebras
with the common cyclic separating vector $\Omega\in\cH$. 
Then $(\cM, \cN, \Omega)$ is a $\pm$half-sided modular inclusion if and only 
if the corresponding standard subspaces 
$V_\cN := \oline{\cN_h \Omega} \subeq V_\cM := \oline{\cM_h \Omega}$ 
define a $\pm$half-sided modular inclusion.
\end{lemma} 

\begin{proof} Since 
$\oline{\Delta_\cM^{-it} \cN_h \Delta_\cM^{it} \Omega}
= \Delta_\cM^{-it} V_\cN,$ 
relation \eqref{eq:modular-inc} implies 
\begin{equation}
  \label{eq:modinc-std}
 \Delta_\cM^{-it} V_\cN \subeq V_\cN \quad \mbox{ for } \quad \pm t \geq 0.
\end{equation}
If, conversely, the latter condition is satisfied, then 
$\Delta_\cM^{-it} \cN_h \Delta_\cM^{it} \Omega 
\subeq \Delta_\cM^{-it} V_\cN \subeq V_\cN,$ 
so that Lemma~\ref{lem:4.14} implies \eqref{eq:modular-inc}.
\end{proof}

The preceding lemma has a very interesting consequence 
because it translates directly between half-sided modular inclusions 
of von Neumann algebras and half-sided modular inclusions of the corresponding 
standard subspaces. It immediately implies that a triple 
$(\cM,\cN,\Omega)$ consisting of two von Neumann algebras 
$\cM$ and $\cN$ with a common cyclic separating vector $\Omega$ 
defines a modular intersection in the sense of \cite{Wi97} if and only if 
$V_\cM$ and $V_\cN$ have a modular intersection.  

Clearly, every result on half-sided modular inclusions on 
standard subspaces, such as Borchers' Theorem~\ref{thm:bor-stand} 
(\cite[Thms.~II.5.2, VI.2.2]{Bo00}), 
Wiesbrock's Theorem~\ref{thm:wiesbrock} 
(\cite{Wi93, AZ05}), and Theorem~\ref{thm:wies2-standard} 
(\cite[Lemma~10]{Wi98}) yield corresponding results on half-sided 
modular inclusions of von Neumann algebras which 
preceded the corresponding results on standard subspaces. 

It is remarkable that this transfer also works in the other direction: 
every result on half-sided modular inclusions of von Neumann algebras 
can be used to obtain a corresponding result on standard subspaces. 
For this transfer one can use the second quantization procedure 
described in some detail in Section~\ref{sec:6} below. It associates 
to every standard subspace $V \subeq \cH$ 
a von Neumann algebra $\cR(V) \subeq B(\cF_+(\cH))$ on the bosonic 
Fock space $\cF_+(\cH)$ for which the vacuum $\Omega$ is a cyclic separating vector 
and for which the modular objects are obtained by second quantization. 
Here we consider the  antiunitary representation 
\[ \Gamma \: \AU(\cH) \to \AU(\cF_+(\cH)),\quad 
\Gamma(U)(v_1 \vee \cdots \vee v_n) 
:= Uv_1 \vee \cdots \vee Uv_n\] 
obtained by second quantization. 
If $\gamma_V \: \R^\times \to \AU(\cH)$ is the antiunitary representation 
associated to $V$, then $\tilde\gamma_V := \Gamma \circ \gamma_V$ 
is the corresponding antiunitary representation on the Fock space $\cF_+(\cH)$ 
(cf.~Proposition~\ref{prop:3.2}). 

If $\Delta_V^{-it/2\pi}H = \gamma_V(e^t)H \subeq H$ holds for $t \geq 0$, then 
\[ \cR(\gamma_V(e^t)H) 
= \tilde\gamma_V(e^t) \cR(H) \tilde\gamma_V(e^{-t})\subeq \cR(H)\]  
implies that $(\cR(V), \cR(H),\Omega)$ is a $\pm$half-sided modular 
inclusion whenever $H \subeq  V$ is. 

If, conversely, $(\cR(V), \cR(H),\Omega)$ is a $\pm$half-sided modular 
inclusion, then 
\[ \cR(\gamma_V(e^t)H) 
= \tilde\gamma_V(e^t) \cR(H) \tilde\gamma_V(e^{-t})\subeq \cR(H) \] 
implies that $\gamma_V(e^t)H \subeq H$ by Theorem~\ref{thm:araki-1}(i). 
Therefore $H \subeq V$ is a half-sided modular inclusion of the same type. 
As the subgroup of $\AU(\cF_+(\cH))$ generated by the corresponding 
one-parameter groups $\tilde\gamma_V(\R^\times)$ is contained in 
the subgroup $\Gamma(\AU(\cH))$ which is the range of the 
second quantization homomorphism $\Gamma \: \AU(\cH) \to \AU(\cF_+(\cH))$, 
anything that we can say about subgroups generated by these groups 
and conditions relating to modular objects can be translated into 
a corresponding result on standard subspaces and the antiunitary 
one-parameter groups $\gamma_V$ on~$\cH$. 

According to this principle, any result on half-sided modular inclusions 
of von Neumann algebras has a ``one-particle version'' concerning 
standard subspaces and vice versa (cf.~\S\S\ref{subsec:3.3}, \ref{subsec:modint}). 
The advantage of the one-particle 
version is that it has a simpler formulation and that 
standard subspaces are completely encoded in 
the antiunitary representations $\gamma_V$ of $\R^\times$, hence in 
an antiunitary representation of a group $(G, G_1)$ generated by 
the image of homomorphism $(\R^\times, \R^\times_+) \to (G,G_1)$. 
Therefore one can hope that any results on standard subspaces, 
half-sided modular inclusions and the corresponding groups can be 
expressed in terms of antiunitary representations of suitable 
involutive pairs of Lie groups $(G, G_1)$. This was one of the key 
motivations for us to write this note. 

\begin{remark} (a) A typical result of this type is Wiesbrock's Theorem 
on half-sided modular inclusions (cf.~Theorem~\ref{thm:wiesbrock} and 
\cite{Wi93, Wi97, AZ05}). On the level of modular inclusions of 
von Neumann algebras $(\cM,\cN,\Omega)$, 
Wiesbrock provides the additional information that, 
if $\cM$ is a factor, then it is of type III$_1$ 
(see \cite[Thm.~12]{Wi93} which uses \cite{Lo82}). 
It would be interesting 
to see if and how this can be formulated and derived on the level of standard 
subspaces and antiunitary representations. 
The discussion of modular nuclearity in \cite[\S 6.3]{Lo08} may 
indicate a possible way how this can be done. 

(b) In \cite[Thm.~4.11]{GLW98} 
(see also \cite[Lemmas 3,4ff]{Wi93c} and \cite[Thm.~2]{Wi93b}), 
similar structures related to 
multiple modular inclusions are studied, namely 
quadruples $(\cM_0, \cM_1, \cM_2, \Omega)$, where 
the $\cM_j$ are von Neumann algebras with the common separating cyclic vector $\Omega$ 
such that the $\cM_j$ commute pairwise and, in cyclic order, 
$\cM_j \subeq \cM_{j+1}'$ is a half sided modular inclusion. 
From this structure, which arises from partitions of $\bS^1$ into 
three intervals, one  derives antiunitary positive energy representations 
of $\PGL_2(\R)$ as in Theorem~\ref{thm:wies2-standard} (\cite[Thm.~1.2]{GLW98}). 

(c) In \cite{Wi93b} it is shown that 
the von Neumann version of Theorem~\ref{thm:wies2-standard}(b) characterizes 
conformal quantum fields on the circle in terms of modular data associated to three intervals. 

(d) In \cite{KW01} configurations of 6 von Neumann algebras 
$(\cM_{ij})_{1 \leq i < j \leq 4}$ are used to generate unitary representations of the 
group $\SO_{1,3}(\R)^\uparrow$ and further of 
the connected Poincar\'e group $P(4)^\uparrow_+$. 
\end{remark}

\subsection{Borchers triples} 
\mlabel{subsec:4.3} 

In this subsection we briefly discuss generalization of Borchers pairs 
to higher dimensional situations, where the semigroup 
$\R_+$ acting on a standard subspace is replaced by a 
wedge $W$ in Minkowski space or by the subsemigroup of $P(d)_+$ 
mapping such a wedge into itself. 

\begin{definition} \mlabel{def:wedges}
In $d$-dimensional Minkowski space $\R^{1,d-1}$, we consider the {\it right wedge} 
\[ W_R:=\big\{ x = (x_0, \ldots, x_{d-1}) \in \R^{d} \: x_1 > |x_0|\big\}.\] 
To fix notation for the following, we write 
$W_R = W_R^2 \oplus E_R,$ where 
\[ E_R = \{ (x_0, \bx) \: x_0 = x_1 = 0\}\cong \R^{d-2} \] 
is the {\it edge of the wedge} 
and $W_R^2$ is the standard right wedge in $\R^2$. 

A subset of the form $W = gW_R$, $g \in P(d)$, is called a {\it wedge}. 
We write $\cW$ for the set of wedges in~$\R^{1,d-1}$. 
\end{definition}

The following lemma contains some details on $\cW$ as an 
ordered homogeneous space. For item (iii), we 
recall the generator $b_0$ of the Lorentz boost from Example~\ref{ex:one-par} 
(see also \cite[\S 2]{BGL02}). 

\begin{lemma}
  \mlabel{lem:4.17} 
The wedge space $\cW$ has the following properties: 
  \begin{itemize}
  \item[\rm(i)] The stabilizer $P(d)_{W_R} = \{ g \in P(d)\:gW_R = W_R\}$ 
of the standard right wedge has 
the form 
\[ P(d)_{W_R} \cong E(d-2) \times \OO_{1,1}(\R)_{W_R^2},\] 
where $E(d-2)$ denotes the euclidean group on $E_R \cong \R^{d-2}$.
  \item[\rm(ii)] $r_W  := g R_{01} g^{-1}$ for $W = gW$  and 
$R_{01} = \diag(-1,-1,1,\ldots, 1)$ 
yields a consistent definition of  wedge reflections $(r_W)_{W \in \cW}$. 
  \item[\rm(iii)] The subgroup $P(d)^\uparrow$ acts transitively on $\cW$, 
and the following are equivalent for $g \in P(d)^\uparrow$: 
\begin{itemize}
\item[\rm(a)] $gW_R = W_R$. 
\item[\rm(b)] $\Ad_g b_0 = b_0$ and 
$g$ commutes with the wedge reflection $r_{W_R} = R_{01}$.
\item[\rm(c)] $\Ad_g b_0 = b_0$. 
\end{itemize}
The set of all elements satisfying these conditions is 
\begin{equation}
  \label{eq:stabgrp}
P(d)^\uparrow_{W_R} \cong E(d-2) \times \SO_{1,1}(\R)^\uparrow.
\end{equation}
In particular, the subgroup $\SO_{1,1}(\R)^\uparrow$ is central in 
$P(d)^\uparrow_{W_R}$. 
For $d > 2$, even the identity component acts transitively on $\cW$ with 
stabilizer 
\begin{equation}
  \label{eq:stabgrp2}
 P(d)^\uparrow_{+,W_R} \cong E(d-2)_+ \times \SO_{1,1}(\R)^\uparrow.
\end{equation}
\item[\rm(iv)] For $\gamma_{W_R} \: \R^\times \to P(d)_+$,  
defined by $\gamma_{W_R}(e^t) := e^{tb_0}$ and $\gamma_{W_R}(-1) := r_{W_R}$, 
we have a bijection 
\[ \cW \to C_{\gamma_{W_R}}, \quad g W_R \mapsto \gamma_W := \gamma_{W_R}^g \quad \mbox{ for } \quad 
g \in P(d)^\uparrow_+, \quad 
\gamma_{W_R}^g(t) = g\gamma_{W_R}(t)g^{-1}\] 
and the map 
\[ \cW \to C_{r_{W_R}} = \{ r_W \: W \in \cW\}, \quad 
W \mapsto r_W \] 
corresponding to evaluation in $-1$ is a two-fold covering map. 
\item[\rm(v)] We have a bijection 
$\cW \to \Ad(P(d)^\uparrow)b_0, gW_R \mapsto \Ad(g)b_0$ 
of $\cW$ onto an adjoint orbit of $P(d)^\uparrow$. 
\item[\rm(vi)] The stabilizer $P(d)_{W_R}$ is open in the centralizer 
of $r_{W_R}$ in $P(d)$. In particular 
$(P(d), P(d)_{W_R})$ is a symmetric pair and $\cW$ is a symmetric space. 
\item[\rm(vii)] The semigroup 
$S_{W_R} := \{ g \in P(d) \: gW_R \subeq W_R\}$ is given by 
$\oline{W_R} \rtimes \OO_{1,d-1}(\R)_{W_R}$. 
  \end{itemize}
\end{lemma} 

\begin{proof} (i) The stabilizer group 
contains the translation group corresponding to the edge $E_R$ 
and $gW_R = W_R$ implies $g(0) \in E_R$, so that 
\[ P(d)_{W_R} \cong E_R \rtimes \OO_{1,d-1}(\R)_{W_R}.\] 
Further, 
\[ \OO_{1,d-1}(\R)_{W_R} 
= \OO_{d-2}(\R) \times \OO_{1,1}(\R)_{W_R^2} 
= \OO_{d-2}(\R) \times (\SO_{1,1}(\R)^\uparrow \{\1, R_1\}), \] 
where $R_1 = \diag(1,-1,1,\ldots, 1)$. We thus obtain (i).

(ii) follows from the fact that the wedge 
reflection $R_{01}$ commutes with $P(d)_{W_R}$.
 
(iii) That $P(d)^\uparrow$ acts transitively follows from the fact that 
the stabilizer $P(d)_{W_R}$ contains the reflection $R_1$ 
satisfying $R_1 V_+ = -V_+$. 
If $d > 2$, then the stabilizer $P(d)_{W_R}$ intersects all four connected 
components of $P(d)$, so that even $P(d)^\uparrow_+$ acts transitively. 
For $d = 2$ we obtain two orbits because $\pm W_R$ lie in different orbits 
of $P(2)^\uparrow_+$. 
 
It remains to verify the equivalence of (a), (b) and (c). 
From \eqref{eq:stabgrp} we derive that (a) implies (b) and hence (c). 
That $g = (b,a)$ commutes with $R_{01}$ is equivalent to 
$b \in E_R$  and $a = a_1 \oplus a_2$ with 
$a_1$ acting on the first two coordinates and $a_2$ on $E_R$. 
That, in addition, $g$ commutes with $b_0$ restricts $a_1 \in \OO_{1,1}(\R)^\uparrow$ 
to an element of $\SO_{1,1}(\R)^\uparrow$. Finally, we observe that, 
if $g$ commutes with $b_0$, then the eigenspace decomposition 
of $\ad b_0$ on $\fp(d)$ implies that $g = (b,a)$ with 
$b \in E_R$ and $a = a_1 \oplus a_2$ with $a_1 \in \SO_{1,1}(\R)^\uparrow$. 

(iv) The first part follows from the equivalence of (a) and (b) in (iii). 
For the second part, we observe with  (iii) above that 
the stabilizer of $W_R$ is a subgroup of index $2$ in the centralizer 
of $r_{W_R}$. 

(v) follows from the equivalence of (a) and (c) in (iii). 

(vi) The centralizer of $r_{W_R}$ in $P(d)$ 
is the subgroup $E(d-2) \times \OO_{1,1}(\R)$ in which 
the stabilizer group $P(d)_{W_R}$ is open. This means that 
$(P(d), P(d)_{W_R})$ is a symmetric pair. 

(vii) For a closed convex subset $C \subeq \R^d$, its recession cone 
\[ \lim(C) 
:= \{ x \in \R^d \: x + C \subeq C \}  
=  \{ x \in \R^d \: (\exists c \in C)\, c + \R_+ x \subeq C \} \] 
is a closed convex cone (\cite[Prop.~V.1.6]{Ne00}), and each 
affine map $g = (b,a) \in \break{\R^d \rtimes \GL_d(\R)} \cong \Aff(\R^d)$ satisfies 
\begin{equation}
  \label{eq:lim-cone}
\lim(gC) = \lim(aC) = a\lim(C).
\end{equation}

If $g = (b,a) \in S_{W_R}$, then $g$ maps $\oline{W_R}$ into itself, 
so that $b = g(0) \in \oline{W_R}$. Further \eqref{eq:lim-cone} implies that 
$\oline{W_R} \supeq \lim(gW_R) = a \oline{W_R}$, and hence $aW_R \subeq W_R$. 
It follows that $aE_R \subeq E_R$, so that $aE_R = E_R$ as $a$ is injective 
and $\dim E_R < \infty$. This in turn implies that $a$ commutes with 
$r_{W_R}$, so that $a = a_1 \oplus a_2$ as above, 
where $a_2 \in \OO(E_R)$ and $a_1 W_R^2 \subeq W_R^2$. 
As $a_1 W_R^2$ is a quarter plane bounded by light rays, we get 
$a_1 W_R^2 = W_R^2$, and finally $a W_R = W_R$. 
\end{proof}

\begin{definition} \mlabel{def:standard-pair} (\cite[Def.~2.7]{Le15}) 
A $d$-dimensional {\it standard pair $(V,U)$  
with translation symmetry 
relative to $W\in \cW$} consists of a standard subspace 
$V \subeq \cH$  and a strongly continuous unitary positive energy 
representation $U$ of the translation group $\R^d$  (cf.~Definition~\ref{def:posen-rep}) 
such that $U_xV \subeq V$ whenever $x + W \subeq W$. 
\end{definition}

Here is the corresponding concept for von Neumann algebras: 

\begin{definition} (\cite[\S4]{BLS11}) A {\it (causal) Borchers triple}
$(\cM, U,\Omega)$ relative to the wedge $W\subeq \R^d$ consists of 
\begin{itemize}
\item[\rm(B1)] a von Neumann algebra $\cM \subeq B(\cH)$, 
\item[\rm(B2)] a positive energy representation $(U,\cH)$ of 
the translation group $\R^d$ such that $U_x \cM U_x^* \subeq \cM$ if $x + W \subeq W$, and 
\item[\rm(B3)] a $U$-invariant unit vector $\Omega \in \cs(\cM)$. 
\end{itemize}
\end{definition}

\begin{remark} Let $\cM \subeq B(\cH)$ be a von Neumann algebra,  
$U \: \R^d \to \U(\cH)$ be a continuous unitary representation 
and $\Omega \in \cH^U\cap \cs(\cM)$. We consider the corresponding standard subspace 
$V := \oline{\cM_h\Omega}$. 
Then $U_x V \subeq V$ is equivalent to 
$U_x \cM U_x^* \subeq \cM$ by Lemma~\ref{lem:4.14}. 
Therefore $(\cM, U,\Omega)$ is a Borchers triple with respect to $W$ 
if and only if $(V,U)$ is a standard pair with respect to $W$. 
\end{remark}

The following theorem can be obtained by translating \cite{Bo92} 
from the context of Borchers triples to standard pairs by 
arguing as in \S \ref{subsec:4.2}. 
We give a direct proof based on our Theorem~\ref{thm:bor-stand}.

\begin{theorem}[Borchers' standard pair Theorem] \mlabel{thm:6.2a} 
Let $(V,U)$ be a $d$-dimensional 
standard pair with translation symmetry relative to $W_R$ 
and $\gamma_{W_R} \: \R^\times \to \SO_{1,d-1}(\R)$ be the corresponding 
homomorphism with $\gamma_{W_R}(e^t) = e^{tb_0}$ and $\gamma_{W_R}(-1)= r_{W_R} = R_{01}$ 
{\rm(Lemma~\ref{lem:4.17}(iv))}. Then the antiunitary 
representation $(U^V,\cH)$ of $\R^\times$ corresponding to $V$ satisfies 
\[ U^V_t U_x U^V_{t^{-1}} = U_{\gamma_{W_R}(t)x} \quad \mbox{ for } \quad x \in \R^d, 
t \in \R^\times, \] 
so that we obtain an antiunitary representation 
of $\R^d \rtimes \SO_{1,1}(\R)$ 
by $(b, \gamma_{W_R}(t)) \mapsto U_b U^V_t$. 

Conversely, every antiunitary positive energy representation 
$(U,\cH)$ of $\R^d \rtimes \SO_{1,1}(\R)$ defines a standard pair 
$(V_{\gamma_{W_R}}, U\res_{\R^d})$.
\end{theorem}

\begin{proof} 
First we write $\R^d = \R^{1,1} \oplus \R^{d-2}$, so that 
$W_R = W_R^2 \oplus \R^{d-2}$, where $W_R^2 \subeq \R^{1,1}$ is the standard right wedge. 
For the light-like vectors $\ell_\pm := (1,\pm 1,0,\ldots, 0)$ we then have 
$W_R^2 = \R_+^\times \ell_+ - \R_+^\times \ell_-$. By assumption,
$U_{t\ell_\pm} = e^{it P_\pm}$ with $P_\pm \geq 0$. 
The strong continuity of $U$ implies 
\[ U_x V \subeq V \quad \mbox{ for all } \quad 
x \in \oline W 
= ([0,\infty)\ell_+  - [0,\infty)\ell_-) \oplus \R^{d-2}.\] 
Now Theorem~\ref{thm:bor-stand} yields 
\[ U^V_{e^t} U_{s\ell_\pm} U^V_{e^{-t}} 
= U_{e^{\pm t} s \ell_\pm} \quad \mbox{ for } \quad t, s \in \R.\] 
Further, $U_x V = V$ for $x  = (0,0,x_2, \ldots, x_{d-1})$ implies that 
$U_x$ commutes with $\Delta_V$. Combing all this, the first 
assertion follows. 

For the converse, let $(U,\cH)$ be an antiunitary 
positive energy representation of the group ${\R^d \rtimes \SO_{1,1}(\R)}$ 
and $V = V_{\gamma_{W_R}}$ be the standard subspace corresponding to 
$\gamma^V := U \circ \gamma_{W_R}$ by (Proposition~\ref{prop:3.2}). 
Since $\gamma_{W_R}$ commutes with $E_R$, the subgroup 
$U_{E_R}$ commutes with $\gamma^U$ and leaves $V$ invariant. 
That $U_x V \subeq V$ for $x \in W_R^2$ follows from the 
positive energy condition and Theorem~\ref{thm:3.8}(ii). 
\end{proof}

Here is a variant of this concept where the translation group is replaced by the Poincar\'e group: 
\begin{definition}
A $d$-dimensional {\it standard pair $(V,U)$  
with Poincar\'e symmetry relative to $W_R \in \cW$} consists of a standard subspace 
$V \subeq \cH$  and a strongly continuous unitary positive energy 
representation $U$ of the connected Poincar\'e group $P(d)^\uparrow_+$, such that 
\begin{itemize}
\item[\rm(i)] $U_gV \subeq V$ for all $g \in P(d)^\uparrow_+$ with $gW_R \subeq W_R$ and 
\item[\rm(ii)] $U_gV \subeq V'$ for all $g \in P(d)^\uparrow_+$ with $gW_R \subeq -W_R$. 
\end{itemize}
\end{definition}

\begin{lemma} \mlabel{lem:stand-pair-poincare}
Let $(U,\cH)$ be an antiunitary positive energy representation of 
$P(d)_+$ for $d \geq 3$ and $V \subeq \cH$ be the standard subspace 
corresponding to the canonical homomorphism 
$\gamma_{W_R} \: \R^\times \to P(d)_+$. 
Then $(V,U)$ is a standard pair with Poincar\'e symmetry.   
\end{lemma}

\begin{proof} If $g W_R = W_R$, then $g \in P(d)_+$ commutes with $\gamma_{W_R}$ 
by Lemma~\ref{lem:4.17}(iii), so that $U_g V = V$. 
Further $U_x V \subeq V$ for $x \in W_R$ (and hence also for 
$x \in \oline{W_R}$ by continuity) follows from the second part of 
Theorem~\ref{thm:6.2a}. In view of Lemma~\ref{lem:4.17}(vii), 
this implies that $U_g V \subeq V$ if $gW_R \subeq W_R$. 

If $g W_R \subeq - W_R$, then the element $r := R_{12} = \diag(1,-1,-1,1,\ldots, 1) 
\in P(d)_+^\uparrow$ satisfies $r g W_R \subeq r(-W_R) = W_R$, 
so that the above argument leads to 
$U_g V = U_rU_{rg} V \subeq U_rV$. Now $\gamma_{W_R}^r = \gamma_{W_R}^\vee$ 
yields $U_r V = V' = V_{\gamma_{W_R}^\vee}$, so that $U_g V \subeq V'$. 
\end{proof}

\begin{remark} \mlabel{rem:inner} (a) In \cite{BLS11}, standard pairs with Poincar\'e symmetry 
are used to obtain Borchers triples by second quantization 
(cf.~Section~\ref{sec:6}). Composing with a 
deformation process due to Rieffel, this 
construction yields non-free quantum fields in arbitrary large dimensions. 

(b) The main point of the notion of a 
Borchers triple is that they can be used to construct a representation  
of the Poincar\'e group $P(d)_+^\uparrow$ by generating it with 
modular one-parameter groups of a finite set of von Neumann algebras 
with a common cyclic separating vector (\cite{Bo96}, \cite{Wi93c, Wi98, SW00, KW01}). 

(c) For $d = 1$, we think of $\R = \R^d$ as 
the underlying space as a light ray in Minkowski space, 
so that the Poincar\'e group is replaced by the affine group $\Aff(\R)$. 
In this context unitary ``endomorphisms'' of irreducible 
one-dimensional standard pairs (Definition~\ref{def:standard-pair}) are  
unitary operators $W \in \U(\cH)$ commuting with the one-parameter 
group $U$ satisfying $WV \subeq V$. 
If $P$ is the momentum operator determined by $U_t = e^{itP}$, 
then these are precisely the operators of the form 
$W = \phi(P)$, where $\phi$ is a 
{\it symmetric inner function} on the upper half plane 
$\C_+ \subeq \C$ and  
$P$ is the momentum operator. A symmetric inner function is a bounded 
holomorphic function on $\C_+$ satisfying 
\[ \phi(p)^{-1} = \oline{\phi(p)} = \phi(-p) \quad \mbox{ for almost all } \quad 
p \in \R\] 
(\cite[Cor.~2.4]{LW11}; see also  Example~\ref{ex:hardy}). 
That these functions can be used to construct Borchers triples was 
shown by Tanimoto in \cite{Ta12}. 
\end{remark}

Much more could be said about the structured 
related to standard subspaces, half-sided modular inclusions, 
modular intersections etc.. For more details and an in depth study 
of these concepts, we refer to \cite{Bo97} and Wiesbrock's work 
\cite{Wi93c, Wi97b, Wi98}.

\subsection{Modular geometry} 
\mlabel{subsec:4.4}

In this subsection we discuss some of the geometric structures
arising from a single von Neumann algebra $\cM \subeq B(\cH)$ which has 
cyclic separating vectors. Any such vector $\xi$ leads 
to a standard subspace $V_\xi = \oline{\cM_h \xi}$ and corresponding 
modular objects $(\Delta_\xi, J_\xi)$ (Theorem~\ref{thm:tom-tak}). Fixing a cyclic separating 
vector $\Omega$, the associated natural cone provides a means to 
analyze the orbits of the group generated by 
$\U(\cM)$, $\U(\cM')$ and the modular conjugations on the data.

\begin{definition} \mlabel{def:4.20} We consider a von Neumann algebra $\cM \subeq B(\cH)$ 
for which the set $\cs(\cM)$ of cyclic and separating unit vectors 
is non-empty. We fix an element $\Omega \in \cs(\cM)$ and the 
corresponding modular objects $(\Delta, J)$ 
(Theorem~\ref{thm:tom-tak}). We recall 
the {\it natural positive cone} 
\[ \cP := \oline{\{Aj(A)\Omega\: A \in \cM\}}, \quad \mbox{ where } \quad 
j(A) := JAJ\] 
(\cite[Def.~2.5.25]{BR87}) and write 
\[ \cs(\cM)_+ := \cP \cap \cs(\cM)\] 
for the set of cyclic separating unit vectors in $\cP$. 
We further write 
\[ \mc(\cM) := \{ J_\xi \: \xi \in \cs(\cM)\} \] 
for the corresponding set of {\it modular conjugations}. We further consider the 
set 
\[ \ms(\cM) = \{ V_\xi = \oline{\cM_h \xi} \: \xi \in \cs(\cM)\} 
\subeq \Stand(\cH) \] 
of {\it modular standard subspaces for $\cM$} and 
note that $\Delta_{V_\xi} = \Delta_\xi$ and $J_{V_\xi} = J_\xi$. 
\end{definition}

We write $\cZ := \cM \cap \cM'$ for the center of $\cM$. 

\begin{proposition} \mlabel{prop:4.7} 
The following assertions hold: 
\begin{itemize}
\item[\rm(i)] {\rm(Polar decomposition of $\cs(\cM)$)} The map 
$\U(\cM') \times \cs(\cM)_+ \to \cs(\cM), (U,\xi) \mapsto U\xi$ 
is a bijection. 
\item[\rm(ii)] The unitary groups $\U(\cM)$ and $\U(\cM')$ 
both act transitively on $\mc(\cM)$ by conjugation. 
For $J \in \mc(\cM)$, the stabilizer in both groups is the discrete central subgroup 
\[ \U(\cM')_J = \U(\cM)_J = \Inv(\U(\cZ)) = \{ z \in \U(\cZ) \: z^2 = \1\}\]  
of central unitary involutions. The orbit map 
$\sigma \: \U(\cM) \to \mc(\cM), \sigma(U) := UJU^{-1}$ 
is a covering morphism of Banach--Lie groups if we identify 
$\mc(\cM)$ with the quotient $\U(\cM)/\ker\sigma$. 
\item[\rm(iii)] $\grp(\mc(\cM)) = \Comm(\U(\cM)) \mc(\cM) \{ \1,J\},$ 
where $\Comm(\U(\cM))$ is the commutator subgroup of $\U(\cM)$. 
\item[\rm(iv)] $\cs(\cM)_J := \{ \xi \in \cs(\cM) \: J_\xi = J \} 
= \Inv(\U(\cZ))\cs(\cM)_+$. 
\item[\rm(v)] The stabilizer of $J$ in the group 
$\U(\cM)U(\cM')$ is $\{ Uj(U) \: U \in \U(\cM)\}\Inv(\U(\cZ))$. 
\item[\rm(vi)] For $\xi \in \cs(\cM)$, we have 
$V_\xi = V_\Omega$ if and only if there exists a 
selfadjoint operator $Z$ affiliated with $\cZ$, i.e., 
commuting with $\U(\cM)\U(\cM')$, such that $\xi = Z\Omega$. 
\end{itemize}
\end{proposition}

\begin{proof} (i) For any $\xi \in \cs(\cM)$, there exists a 
unit vector $\tilde\xi \in \cP$ defining the same state of $\cM$ 
(\cite[Thm.~2.5.31]{BR87}). 
By the GNS Theorem, there exists a $U \in \U(\cM')$ with 
$\xi = U\tilde\xi$. 
Since the elements of $\cs(\cM)_+$ are also separating for $\cM'$, 
their stabilizer in $\U(\cM')$ is trivial. 

To verify injectivity, it remains to see that every $\U(\cM')$-orbit 
in $\cs(\cM)$ meets $\cs(\cM)_+$ exactly once. 
Let $U \in \U(\cM')$ and $\xi \in \cs(\cM)_+$ be such that 
$U\xi \in \cP$. As $J_\xi = J$ for every $\xi \in \cP$ by 
\cite[Prop.~2.5.30]{BR87}, we obtain 
$U J U^{-1} = J_{U\xi} = J$. 
Then $j(U)= U$ leads to $U \in \cM \cap \cM'$ and hence 
to $U = j(U) = U^{-1}$ (Theorem~\ref{thm:tom-tak}(c)), so that $U^2 =\1$. 
Then $\cM_{\pm 1}:= \{ M \in \cM \: UM = \pm M\}$ are ideals of $\cM$ 
and $\cM \cong \cM_+ \oplus \cM_-$ is a direct sum of von Neumann algebras. 
Now $\xi = \xi_+ \oplus \xi_-$ decomposes accordingly with 
$\xi_\pm \in \cs(\cM_\pm)$. As 
$U \xi = \xi_+ - \xi_-$ and 
$\cP = \cP_+ \oplus \cP_-$, it follows that  
$\xi_- \in \cP_- \cap - \cP_- = \{0\}$ (\cite[Prop.~2.5.28]{BR87}) 
and thus $\cM_- = \{0\}$ and $U = \1$. 

(ii) If $\Omega_1, \Omega_2 \in \cs(\cM)$, 
then \cite[Lemma~2.5.35]{BR87} implies the existence of a 
unitary element $U \in \U(\cM')$ with $U J_{\Omega_1} U^{-1} = J_{\Omega_2}.$ 
Exchanging the roles of $\cM$ and $\cM'$, it also follows that 
$\U(\cM)$ acts transitively on $\mc(\cM)$. 

For $J \in \mc(\cM)$ we have $J\cM J = \cM'$, so that, for $U \in \U(\cM)$, the relation 
$UJU^{-1} = J$ implies $U = JUJ \in \cZ$. As in (i), this leads to 
$JUJ = U^* = U^{-1}$, so that $U^2 = \1$. Conversely, any involution in 
$\U(\cZ)$ stabilizes~$J$.

Clearly, $\sigma$ is a surjective equivariant map whose kernel is discrete in the 
norm topology. As the stabilizer subgroup of $J$ in $\U(\cM)$ 
is discrete and central, the quotient $\U(\cM)/\ker\sigma$ carries a natural 
Banach--Lie group structure for which $\sigma$ 
becomes a covering homomorphism. 

(iii) We consider the group $G:= \U(\cM')$ and the representation 
of $(G \times G) \rtimes \{\1,\tau\}$ on $\cH$ given by 
$U(g,h,\tau^\eps) = g JhJ J^\eps.$ 
Then Proposition~\ref{prop:4.7} shows that 
\[ U(C_{(e,e,\tau)}) = \{ g Jg^{-1}J J \: g \in \U(\cM') \} 
= \{ g Jg^{-1} \: g \in \U(\cM') \} = \mc(\cM).\] 
Now the assertion follows from Lemma~\ref{lem:abstract-grp}. 

(iv)  We have already seen in (i) that $J_\xi = J$ for every 
$\xi \in \cs(\cM)_+$. 
If $\xi = U \tilde\xi$ for some $U \in \U(\cM')$ and 
$\tilde\xi \in\cs(\cM)_+$ as in (i), then 
$J_\xi = U J U^{-1}$ equals $J$ if and only if 
$U \in \cM \cap \cM'$ is an involution. This proves (iv).

(v) Since $J$ commutes with each operator of the form 
$J U JU = j(U)U = U j(U)$, the stabilizer contains all these elements 
and also $\Inv(\cZ)$, as we have already seen above. 
If, conversely, $U \in \U(\cM)$ and $W \in \U(\cM')$ are such that 
$UW$ commutes with $J$, then 
$UW = UJUJ (JUJ)^{-1} W$ with 
$(JUJ)^{-1}W \in \U(\cM')_J = \Inv(\U(\cZ))$ by (ii). 

(vi) We shall use the theory of KMS states (cf.~\cite{BR96}). 
We recall that, for a continuous action $\alpha \: \R \to \Aut(\cA)$ 
of $\R$ on a $C^*$-algebra $\cA$, a state $\omega$ of $\cA$ 
is called an {\it $\alpha$-KMS state} if, for every pair of hermitian elements 
$A,B \in \cA$, the function 
\[ \psi \: \R \to \C, \quad \psi(t) :=\omega(A\alpha_t(B)) \] 
extends analyticallty to a holomorphic function on the strip 
$\cS := \{ z \in \C \: 0 < \Im z < 1\}$, extends continuously 
to its closure and satisfies 
$\psi(i+t) = \oline{\psi(t)}$ for $t \in \R$. 

First we observe that $\omega(A) := \la \Omega, A \Omega\ra$ 
is a KMS state with respect to the modular automorphism 
group $\alpha_t(A) = \Delta^{it} A\Delta^{-it}$ 
(Takesaki's Theorem,  \cite[Thm.~5.3.10]{BR96}). 
Now let $\xi \in \cs(\cM)$ with $V_\xi = V_\Omega$, i.e., 
$J_\xi = J$ and $\Delta_\xi = \Delta$. Then, for the same reason, the state 
$\omega_\xi(A) := \la \xi, A \xi\ra$ is also an $\alpha$-KMS state. 
By \cite[Prop.~5.3.29]{BR96}, there exists a unique 
positive selfadjoint operator 
$T \geq 0$ affiliated with $\cZ$ such that 
\[\la \xi, A \xi \ra
=  \omega_\xi(A) 
= \omega(\sqrt{T} A \sqrt{T}) 
= \la \Omega, \sqrt{T} A \sqrt{T} \Omega \ra
= \la \sqrt{T} \Omega,  A \sqrt{T} \Omega \ra 
\quad \mbox{ for } \quad A \in \cM.\] 
Therefore $\xi$ and $\sqrt{T}\Omega$ define the same state. 
Further $\sqrt{T}\Omega$ is also contained in the natural cone $\cP$ 
(\cite[Prop.~2.5.26]{BR87}). 

As we have seen in (i), there exists a unique 
$U \in \U(\cM')$ with $U\xi \in \cs(\cM)_+$. 
As $J = J_{U\xi} = U J_\xi U^{-1} = U^{-1} J U$, it 
follows from (ii) that $U \in \Inv(\U(\cZ))$. 
As $U\xi$ and $\sqrt{T}\Omega$ define the same state and 
both are contained in $\cP$, 
\cite[Thm.~2.5.31]{BR87} yields $U\xi = \sqrt{T}\Omega$, 
i.e., $\xi = U \sqrt{T}\Omega$. Now the assertion 
follows with $Z := U \sqrt{T}$. 

Suppose, conversely, that $\xi = Z \Omega$ 
with a selfadjoint operator affiliated to~$\cZ$. 
Decomposing $\cH$, $\cM$ and $\Omega$ as a direct sum 
corresponding to bounded spectral projections of $Z$ 
(which are central in $\cM$ as well), we may w.l.o.g.\ assume that 
$Z$ is bounded. Since $\xi$ is separating, $\ker Z = \{0\}$, 
so that we may further assume that $Z$ is invertible. 
As $Z$ commutes with $J$ and $\Delta$, 
it commutes with $S = J \Delta^{1/2}$, and thus $Z$ 
leaves $V = \Fix(S)$ invariant. This shows that 
$V_\xi = Z V = V$. 
\end{proof}

\begin{remark} (a) Proposition~\ref{prop:4.7}(iv) describes
the fibers of the map 
\[ \cs(\cM) \to \mc(\cM), \quad \xi \mapsto J_\xi.\] 
This map is $\U(\cM')$-equivariant, so that the space 
$\cs(\cM)$ is a homogeneous $\U(\cM')$-bundle over the symmetric 
space $\mc(\cM)$. 

(b) Proposition~\ref{prop:4.7}(vi) describes the fibers 
of the map $\cs(\cM)\to \ms(\cM)$ in terms of the center~$\cZ$. 
If $\cM$ is a factor, i.e., $\cZ = \C\1$, then we see in particular that 
$V_\xi = V$ implies $\xi = \pm \Omega$ (because $\|\xi\|=1$). 
\end{remark}

\begin{lemma}[Stabilizer subgroup of $V = V_\Omega$]
The stabilizer of $V$ in the group $G :=   \U(\cM)\U(\cM')$ 
consists of all elements of the form 
$g = u j(u) z$ with $z \in \Inv(\U(\cZ))$ and 
$u \in \U(\cM)$ fixed by the modular automorphisms 
$\alpha_t(M) = \Delta^{it} M \Delta^{-it}$. 
\end{lemma}

\begin{proof} Since standard subspaces are completely determined by their 
modular objects, the stabilizer of $V$ in $G$ is 
\[ G_V = G_J \cap G_\Delta = \{ g \in G \: gJg^{-1} = J, g\Delta g^{-1} = \Delta \}.\] 
By Proposition~\ref{prop:4.7}(v), any $g \in G_J$ is of the form 
$g = uj(u)z$ with $u \in \U(\cM)$ and $z \in \Inv(\U(\cZ))$. 
As $z$ is central, it commutes with the modular unitaries 
$\Delta^{it}$ (Theorem~\ref{thm:tom-tak}(c)), i.e., $z \in G_V$. 
An element of the form $g = uj(u)$ is fixed by each $\alpha_t$ 
if and only if 
\[ \alpha_t(u) J \alpha_t(u) J = u J u J, 
\quad \mbox{ resp. } \quad 
u^{-1} \alpha_t(u)  = J u \alpha_t(u)^{-1} J.\] 
Then  $u^{-1} \alpha_t(u) \in \cM \cap J\cM J = \cZ$. 
To see that this implies that $\alpha_t(u) = u$, we consider the 
commutative von Neumann subalgebra $\cA \subeq \cM$ 
generated by the center $\cZ$ and $u$. As each $\alpha_t$ fixes the 
center pointwise, we have $\alpha_t(\cA) = \cA$ for every $t \in \R$. 
Then the state of $\cA$ given by 
$\omega(A) := \la\Omega, A \Omega \ra$ is a KMS state 
with respect to $\alpha_t\res_\cA$, so that 
the restrictions $\alpha_t\res_\cA$ are the unique automorphisms corresponding 
to this KMS states (\cite[Thm.~5.3.10]{BR96}). 
Since $\cA$ is abelian, the uniqueness of the automorphism group 
implies its triviality. We conclude that each $\alpha_t$ fixes~$u$ if 
$g \in G_V$. 

If, conversely, $\alpha_t(u) =u$, then $\alpha_t$ fixes $g$. 
This implies that $g$ commutes with $S = J \Delta^{1/2}$, hence 
preserves $V = \Fix(S)$. 
\end{proof}

\begin{ex} \mlabel{ex:hilb-schm} 
(a) If $\cM$ is a factor, then $\Inv(\U(\cZ)) = \{ \pm 1\}$, 
so that $\U(\cM)_J = \{\pm \1\}$ and $\mc(\cM) \cong \U(\cM)/\{\pm \1\}$. 

(b) For $\cM = B(\cK)$ acting on $\cH = B_2(\cK)$ 
by left multiplications,  we have 
$JA = A^*$ (Example~\ref{ex:3.2}(b)) and by (a), 
we have $\mc(B(\cK)) \cong \U(\cK)/\{\pm \1\}$. 
For any $0 < \Omega = \Omega^* \in \cs(\cM)$ we have 
$\cP = \{ A \in B_2(\cK) \: A \geq 0\}$. 

If $\Omega = \sqrt{D}$ holds for the trace class operator $D > 0$, 
then the centralizer of $\Delta$ in $\U(\cM)$ is  
\[ \{ g \in \U(\cM) \: g\Delta g^{-1} = \Delta\} 
\cong  \{ g \in \U(\cK) \: gDg^{-1} = D\}, \] 
and since $D$ is diagonalizable, this subgroup consists 
of those unitaries leaving all eigen\-spaces of $D$ invariant. 
In particular 
\[ \{ A \in B_2(\cK) \: A = A^*, [A,D] = 0\} \subeq V \cap V' \] 
shows that $V \cap V'$ is much larger than $\R \Omega$. 
For every elements $A = A^* \not\in \cP \cup -\cP$ 
with $\ker A = \{0\}$, we have $JA = A$ but 
$J_A \not=J$ (Proposition~\ref{prop:4.7}(iv)). 

(c) For $\cM = L^\infty(X,\fS,\mu)$, $\mu$ finite, 
acting on $\cH = L^2(X,\mu)$ by multiplication operators, 
we find for $Jf= \oline f$ that 
$\U(\cM)_J$ is the set of involutions in $\U(\cM)$.  
As the squaring map $U \mapsto U^2$ is a morphism of Banach--Lie groups, 
the Banach symmetric space 
$\mc(\cM)$ is diffeomorphic to the unitary group and we have a 
short exact sequence 
\[ \1 \to \Inv(\U(\cM)) \to \U(\cM) \to \mc(\cM) \to \1.\] 
\end{ex}

\section{Nets of standard subspaces and von 
Neumann algebras} 
\mlabel{sec:5}

In this section we briefly discuss some elementary properties of 
nets of standard subspaces 
$(V_\ell)_{\ell \in L}$ and their connection with 
antiunitary representations $(U,\cH)$. The connection 
with nets of von Neumann algebras and QFT is made in \S \ref{subsec:5.2}. 
Nets of standard subspaces are considerably simpler than 
nets of von Neumann algebras and naturally 
determined by an antiunitary representation, 
of the group generated by all subgroups $U^{V_\ell}(\R^\times) \subeq \AU(\cH)$ 
(Proposition~\ref{prop:3.2}), but this group need not be finite dimensional. 

\subsection{Nets of standard subspaces} 

Let $\cV := (V_\ell)_{\ell \in L}$ be a family of standard subspaces of the Hilbert 
space~$\cH$. We assume that the map $\ell\mapsto V_\ell$ is injective, so that 
the index set $L$ is only a notational convenience and we could equally 
well work directly with the subset $\{ V_\ell \: \ell \in L\} \subeq \Stand(\cH)$ 
(cf.\ \cite{BGL02, SW03}). 

\begin{definition}
A {\it net automorphism} is an $\alpha \in \AU(\cH)$ permuting 
the standard subspaces $V_\ell$. We write $\Aut(\cV) \subeq \AU(\cH)$ 
for the subgroup of net automorphisms. 
A net automorphism is called an {\it internal symmetry} if it preserves 
each $V_\ell$ separately. The corresponding subgroup of $\Aut(\cV)$ 
is denoted $\Inn(\cV)$. 
\end{definition}

We write $(\Delta_\ell, J_\ell)$ for the modular objects corresponding 
to $V_\ell$ and consider the {\it modular symmetry group} 
\[ \cJ := \grp(\{J_\ell \: \ell \in L\}) \subeq \AU(\cH) \] 
generated by the modular conjugations. 
A natural assumption enriching the underlying geometry is 
the condition of {\it geometric modular action} 
\begin{description}
\item[\rm(CGMA)] $\cJ \subeq \Aut(\cV)$ 
\end{description}
(see \cite{BS93, BDFS00} for the von Neumann context). 
The condition 
\begin{description}
\item[\rm(MS)] $\Delta_{V_\ell}^{it} \in \cJ$ for all $t \in \R, \ell \in L$ 
\end{description}
is called the {\it modular stability condition} 
(\cite{BDFS00}, \cite[\S\S IV.5.6/7]{Bo00}).

From now on we assume that (CGMA) is satisfied. 
We obtain an action \break $\sigma \: \Aut(\cV) \times L \to L, 
(g,\ell) \mapsto \sigma_g(\ell)$ 
on the index set $L$ by 
\begin{equation}
  \label{eq:modact1}
 g V_\ell  = V_{\sigma_g(\ell)}.
\end{equation}
This further implies 
\begin{equation}
  \label{eq:modact2}
g J_\ell g^{-1} = J_{\sigma_g(\ell)} \quad \mbox{ and }\quad 
g \Delta_\ell g^{-1} = \Delta_{\sigma_g(\ell)}. 
\end{equation}
In particular, (CGMA) implies that  
\[ \cS := \{ J_\ell \: \ell \in L\} \] 
is a conjugation invariant set of generators of $\cJ$. 
This fact opens the door to the construction 
of geometric structures from the group $\cJ$ and its generating set 
$\cS$ in specific situations. 

\begin{lemma} \mlabel{lem:5.2.1} 
The subgroup of $\cJ$ of those elements acting trivially on $L$ 
is its center 
\[ \cZ := \{ g \in \cJ \: (\forall \ell \in L)\, gV_\ell = V_\ell \} 
= \Inn(\cV) \cap \cJ = \ker \sigma.\] 
\end{lemma}

\begin{proof}
For $g \in \cJ$, the relation  $\sigma_g = \id_L$ implies that $g$ commutes with 
every element $J_\ell \in \cS$ by~\eqref{eq:modact2}. 
As $\cJ$ is generated by $\cS$, the assertion follows. 
\end{proof}

By Lemma~\ref{lem:5.2.1}, 
the action of $\cJ$ on the index set describes this group as a 
central extension of the group $\cJ/\cZ$ that acts faithfully on the set $L$ 
which is supposed to carry geometric information (cf.~\cite{BDFS00} 
and Remark~\ref{rem:5.5} below). 

An immediate consequence of (CGMA) is that the net is invariant 
under the passage to the symplectic orthogonal space 
$V_\ell \mapsto V_\ell' = J_\ell V_\ell$ (cf.~\S \ref{subsec:3.2}). 
In particular, we have a duality map $\ell \mapsto \ell' := \sigma_{J_\ell}(\ell)$ 
on~$L$. 
We also have a natural order structure $\leq$ on $L$ by 
\[ \ell_1 \leq \ell_2 \quad \mbox{ if }  \quad V_{\ell_1} \subeq V_{\ell_2}.\] 

\begin{remark} \mlabel{rem:order-axioms} The key properties of the triple 
$(L, \leq, ')$, given by the partial order $\leq$ and the {\it duality map} 
$\ell \mapsto \ell'$ are that 
\begin{itemize}
\item[\rm(A1)] $\ell_1 \leq \ell_2$ implies $\ell_2' \leq \ell_1'$, and 
\item[\rm(A2)] $\ell_1 \leq \ell_2'$ if and only if $\ell_2 \leq \ell_1'$. 
\end{itemize}
From (A2) we immediately derive $\ell \leq  \ell''$ and by combining 
this with (A1), we obtain $\ell' = \ell'''$ 
for every $\ell \in L$, i.e., the duality map restricts to an 
involution on its range. 
\end{remark}

\begin{exs} \mlabel{ex:caus-comp} 
(a) For a subset $S$ of the Minkowski space $\R^{1,d-1}$, we 
define the {\it causal complement} by 
\[ S' := \{ x \in \R^{1,d-1} \: (\forall y \in S) [x-y,x-y] < 0\}.\] 
Then $S' = \bigcap_{s \in S} \{s\}'$, which immediately leads to 
(A1/2). Here $S \subeq T'$ means that $S$ and $T$ are {\it space-like separated},  
and $S''$ is the {\it causal completion of $S$}. 
For $g \in P(d)$ and $S \subeq \R^{1,d-1}$, 
we have $(gS)' = gS'$. 

For the standard right wedge $W_R \subeq \R^{1,d-1}$, we have $W_R' = - \oline{W_R}$ 
and $W_R'' = W_R$,  and for the positive light cone $V_+$ we have $V_+' = \eset$ and 
$V_+'' = \R^{1,d-1}$. For $x-y \in V_+$, the causal completion 
\[ \{x,y\}'' = (x - \oline{V_+}) \cap (y + \oline{V_+}) = \oline{\cO_{x,y}} \] 
is the closure of the double cone $\cO_{x,y}$ 
(cf.\ Remark~\ref{rem:doubco}(b)). 

(b) If $(X,\leq)$ is a partially ordered space, then we define 
\[ \{x\}' := \{ y \in X \: x \not\leq y, y \not\leq x\} 
\quad \mbox{ and } \quad 
S' := \bigcap_{s \in S} \{s\}'.\] 
Then the set $L$ of subsets of $X$, endowed with the inclusion order, 
satisfies (A1/2). 

(c) For a complex Hilbert space $\cH$, the set of real subspace 
$V \subeq \cH$, endowed with the inclusion order 
and the symplectic orthogonal space $V' = i V^{\bot_\R}$ 
satisfies (A1/2). 

(d) For a complex Hilbert space $\cH$, the set of von Neumann 
subalgebras $\cM \subeq B(\cH)$,  endowed with the inclusion order 
and the commutant map $\cM \mapsto \cM'$ 
satisfies (A1/2). 
\end{exs}

\begin{remark} \mlabel{rem:5.5} (a) In QFT, one expects that the structure 
$(L, \leq,')$ encodes physical 
information and one would like to recover information 
on the geometry of spacetime from this structure. 
In this context, causal complements, resp., the notion of being space-like 
separated, appears more fundamental than the causal 
order if we want to recover a spacetime $M$ from the triple 
$(L, \leq, ')$, where $L$ consists of certain subsets of $M$ 
but does not contain one-point sets (cf.\ \cite{Ke96}). 

If the modular stability condition is satisfied, i.e., if $\cJ$ contains also the 
modular unitaries, this group is supposed to encode the dynamics of 
the quantum theory, the isometry group of the corresponding 
spacetime and a (projective) unitary representation 
of this group (\cite[\S 6.4]{Su05}). This connects naturally 
with the approach of Connes and Rovelli who ``construct'' the dynamics 
of a quantum statistical system by a modular one-parameter group $\Delta^{it}$ 
(\cite{CR94}). 

Here an interesting result concerning the detection of known group from this 
viewpoint is the characterization of the Poincar\'e group in terms of 
structure preserving maps on the set $\cW$ of wedges in $\R^{1,d-1}$ 
(\cite{BDFS00}, \cite[\S IV.5]{Bo00}, Lemma~\ref{lem:4.17}). 

(b) In \cite[\S 4]{Ke98} the relation of the causal structure 
on spacetime and how it can be determined by data encoded in 
nets of $C^*$-algebras is discussed very much in the spirit of this 
section. We refer to \cite{Ra17} for an approach to quantum field theory 
based on modular theory of operator algebras that does not 
assume an a priori given underlying spacetime. Instead, one would like 
to generate the spacetime geometry from operator theoretic data.  

(c) In \cite{SW03} this program 
is carried out to a large extent by specifying a set of axioms 
formulated in terms of the modular conjugations 
$J_\ell$, such that the index set $L$ 
corresponds to the set $\cW$ of wedges in three-dimensional Minkowski space 
$\R^{1,2}$, $J_W$ corresponds to the  the orthogonal 
reflection $r_W \in P(3)_+$ in the edge of $W$ and $\cJ \cong P(3)_+$ 
(cf.~Lemma~\ref{lem:4.17}). 

In this case, the subset $\SO_{1,2}(\R)^\uparrow W_R\subeq \cW$ 
identifies naturally with the anti-deSitter space 
$\AdS^2  \cong \SO_{1,2}(\R)^\uparrow/\SO_{1,1}(\R)^\uparrow$, which can 
be realized as an adjoint orbit in the Lie algebra 
$\so_{1,2}(\R)\cong \fsl_2(\R)$ 
(cf.\ Lemma~\ref{lem:4.17}(v)). 
\end{remark} 

The following construction is of fundamental importance in our 
approach. It is inspired by the modular localization approach to QFT developed in 
\cite[Thm.~2.5]{BGL02}: 

\begin{proposition} {\rm (Nets of standard subspaces from antiunitary representations; 
the BGL construction)} \mlabel{prop:antiunirep-stand}
Let $(U,\cH)$ be an antiunitary representation of the Lie group pair $(G,G_1)$ and 
associate to $\gamma \: (\R^\times, \R^\times_+) \to (G,G_1)$ 
the standard subspace $V_\gamma$ with $U^{V_\gamma} = U \circ \gamma$ 
{\rm(Proposition~\ref{prop:3.2})}. 
Then, for every non-empty subset $\Gamma \subeq \Hom((\R^\times, \R^\times_+), (G,G_1))$ 
invariant under conjugation with elements of $G$ and 
under inversion, we thus obtain a net 
$(V_\gamma)_{\gamma \in \Gamma}$ of standard subspaces 
which satisfies {\rm(CGMA)} for the group 
$\cJ$ which is the image under $U$ of the subgroup 
of $G$ generated by the conjugation invariant set of involutions 
$\{\gamma(-1) \: \gamma \in \Gamma\}$. 
\[ \mat{
\Gamma & \mapright{\ev_{-1}} & \Inv(G\setminus G_1) \\ 
\mapdown{\gamma \mapsto V_\gamma} & & \mapdown{U} \\ 
\Stand(\cH) &\mapright{V \mapsto J_V} & \Conj(\cH).}\] 
\end{proposition}

\begin{proof} This follows from the observation that, 
for $\gamma^g(t) := g\gamma(t)g^{-1}$, we have 
$U_g V_\gamma = V_{\gamma^g}$ for $g \in G_1$ and 
$U_g V_\gamma = V_{\gamma^g}'$ for $g \not\in G_1$, 
so that $G$ acts naturally by automorphisms on the net 
$(V_\gamma)_{\gamma \in \Gamma}$. 
\end{proof}

\begin{remark}
(a) Evaluation in $-1$ leads to a fibration  
\[ \ev_{-1} \: \Hom((\R^\times, \R^\times_+), (G,G_1)) 
\to \Inv(G\setminus G_1). \] 
An involution $r \in G\setminus G_1$ is contained in the image 
if and only if there exists an $x \in \g$ fixed by $\Ad(r)$, 
i.e., if $\g^{\Ad(r)} = \ker(\Ad(r)-\1)\not=\{0\}$. This is always the case if 
$\g$ is non-abelian, i.e., if $-\id_\g$ is not an automorphism. Then the fiber 
over $r$ can be identified with the Lie subalgebra $\g^{\Ad(r)}$.   
  
(b) In many situations one considers minimal sets 
\[ \Gamma = C_{\gamma} \cup C_{\gamma^\vee}, \quad \mbox{ where } \quad 
C_\gamma := \{ \gamma^g \: g \in G\}. \] 
Then $\ev_{-1}(C_\gamma) = C_r$ is the conjugacy class of the 
involution $r := \gamma(-1) \in G \setminus G_1$, 
hence in particular a symmetric space (cf.~Appendix~\ref{app:a.2}). 
An important example in QFT is $\gamma = \gamma_{W_R}$ for $G = P(d)_+$ (Lemma~\ref{lem:4.17}). 

(c) In the context of Proposition~\ref{prop:antiunirep-stand}, the relation 
$V_\gamma' = V_{\gamma^\vee}$ shows that duality is naturally built into the construction. 
However, in general it may not be so easy to determine when 
$V_{\gamma_1} \subeq V_{\gamma_2}$. In \cite[Thm.~3.4]{BGL02} it is shown that, 
for $G = P(d)_+$ and homomorphisms $(\gamma_W)_{W \in \cW}$ corresponding to wedges,  
the relation $W_1 \subeq W_2$ is equivalent to 
$V_{\gamma_{W_1}} \subeq V_{\gamma_{W_2}}$ if and only if $U$ is a positive energy representation. 
\end{remark}

The preceding discussion suggests a closer 
look at conjugacy classes of involutions 
$\tau \in G \setminus G_1$. 
We write $C_\tau \subeq G$ for the conjugacy class of $G$. 

\begin{lemma} \mlabel{lem:conjug-gen}
Let $G_1$ be a connected Lie group with 
Lie algebra $\g$ and $\tau \in \Aut(G_1)$ be an involutive automorphism. 
Then the conjugacy class $C_\tau \subeq G := G_1 \rtimes \{\1,\tau\}$ 
generates the subgroup $\grp(C_\tau) = B \{\1,\tau\}$, where 
$B$ is the integral subgroup whose Lie algebra is the ideal 
$\fb := \g^{-\tau} + [\g^{-\tau}, \g^{-\tau}]$. In particular, 
$C_\tau$ generates $G$ if and only if $\g^\tau = [\g^{-\tau}, \g^{-\tau}]$. 
\end{lemma}

\begin{proof} Let $H = \grp(C_\tau) \subeq G$ be the subgroup generated 
by $C_\tau$. As $\tau \in H$, we have 
$H = B \{e,\tau\}$ for $B := H \cap G_1$. 
Then $B$ is generated by 
the elements of the form $g \tau(g)^{-1}$, $g \in G_1$, hence in particular 
arcwise connected. 
For $x \in \g^{-\tau}$, we therefore obtain 
$\exp(2x) = \exp(x) \tau(\exp(-x)) \in B$, so that 
the Lie algebra $\fb$ of $B$ contains $\g^{-\tau}$ and hence also 
$\fb = \g^{-\tau} + [\g^{-\tau},\g^{-\tau}]$, which is an ideal of $\g_1$. 

Let $\tilde G_1$ denote the universal covering group of 
$G_1$. Then $\tilde B := \la \exp_{\tilde G_1} \fb \ra$ is a normal integral 
subgroup of $\tilde G_1$, hence closed. As $\tau$ acts trivially 
on the quotient group $\tilde G_1/\tilde B$, 
all elements of the form $g\tau(g)^{-1}$ are contained in $\tilde B$. 
Therefore $B := \la \exp_G \fb \ra$ contains the arcwise
connected subgroup $H \cap G_1$, and thus 
$H \cap G_1 = B$. This implies the lemma. 
\end{proof}

\begin{ex} \mlabel{exs:5.8}
$G = \Aff(\R)$ and $\tau = (x,-1)$ with $C_\tau = \R \times \{-1\}$, 
and this conjugacy class generates~$G$. 
\end{ex} 

\begin{lemma}
 \mlabel{lem:5.8} 
Let $\tau = r_W \in P(d)_+$ be a wedge reflection for some 
$W \in \cW$. Then 
\begin{itemize}
\item[\rm(i)] The conjugacy class $C_\tau$ of $\tau$ generates 
$P(d)_+$ if and only if $d > 2$. 
\item[\rm(ii)] The conjugacy class $C_\tau$ of $\tau$ in the conformal 
group $\SO_{2,d}(\R)$ generates the whole group for any $d > 0$. 
\end{itemize}
\end{lemma}

\begin{proof} (i) Since all wedges 
$W \in \cW$ are conjugate to the standard right wedge $W_R$, it suffices to consider
$\tau = r_{W_R} = R_{01} = \diag(-1,-1,1,\ldots,1)$. 
If $d = 2$, then $R_{01} = - \1$, so that $\grp(C_\tau) 
= \R^2 \rtimes \{ \pm \1\}$ is a proper subgroup of $P(2)_+$. 

The case $d = 2$ already implies 
that $\grp(C_\tau)$ contains all translations in the directions of all Lorentzian 
$2$-planes, hence all translations. Therefore it suffices to show that the 
conjugacy class of $R_{01}$ in $\SO_{1,d}(\R)$ generates the whole group. 
In view of Lemma~\ref{lem:conjug-gen}, this follows from 
the simplicity of the real Lie algebra 
$\g = \so_{1,d-1}(\R)$. 

(ii) We consider $\SO_{2,d}(\R)$ as a group acting on 
$\R^{1,d-1}$ by rational maps (cf.~\cite[\S17.4]{HN12}). 
We have already seen above that the group $\grp(C_\tau)$ generated by 
the conjugacy class $C_\tau$ in $\SO_{2,d}(\R)$ 
contains the Poincar\'e group $P(d)_+$, which is a parabolic 
subgroup of $\SO_{2,d}(\R)$ and it intersects both connected components. 
By the same argument, it contains the opposite 
parabolic subgroup, and both subgroups generate $\SO_{2,d}(\R)$ because 
it has only two connected components (cf.~\cite{Be96}). 
\end{proof}

If $d$ is odd, then $\SO_{2,d}(\R) \cong \OO_{2,d}(\R)/\{\pm \1\}$ 
is the full conformal 
group of $\R^{1,d-1}$, but if $d$ is even, then 
the kernel $\{\pm \1\}$ of the action of $\OO_{2,d}(\R)$ 
is contained in the identity component, so that 
$\Conf(\R^{1,d}) \cong \OO_{2,d}(\R)/\{ \pm \1\}$ 
has four connected components (\cite[\S17.4]{HN12}). 
Therefore the conjugacy class of a wedge reflection does not 
generate the whole conformal group. 

\begin{remark} In \cite[Thm.~4.7]{BGL02},
Brunetti, Guido and Longo describe a 
one-to-one correspondence between antiunitary 
positive energy representations of $P(d)_+$ and certain 
nets of closed real subspaces $V_\cO$ indexed by certain 
open subsets $\cO \subeq \R^d$, for which the subspaces $(V_W)_{W \in \cW}$ 
corresponding to wedges are standard and the modular covariance condition
\[ \Delta_W^{-it/2\pi} V_\cO = V_{\gamma_W(t)\cO} \] 
 holds for the homomorphisms $\gamma_W \: \R^\times \to P(d)_+$ 
and the modular unitaries of $V_W$. 

The uniqueness of the local net, 
once the unitary representation is given, 
is discussed in \cite[Rem.~4.8]{BGL02} (see also \cite{BGL93}). For the converse, i.e.,  
the uniqueness of the unitary representation, once the local net 
is given, we refer to \cite{BGL93}. 
In \cite{Mu01}, Mund shows that, for any representation 
$(U,\cH)$ of $P(d)^\uparrow_+$ that is a finite direct sum of 
irreducible representations of strictly positive mass, there 
is only one covariant net of standard subspaces; which therefore 
coincides with the one obtained in Proposition~\ref{prop:antiunirep-stand} 
from any antiunitary extension of $U$ to $P(d)_+$. 
\end{remark}

\begin{ex} (Nets arising from a single von Neumann algebra) \\
(a) Let $\cM \subeq B(\cH)$ be a von Neumann algebra 
for which $\cs(\cM) \not=\eset$ and consider the 
corresponding set 
\[ \cV := \ms(\cM) = \{ V_\xi \: \xi \in \cs(\cM)\} \] 
of standard subspaces (Definition~\ref{def:4.20}). 
Fix fix a cyclic separating vector $\Omega$ and the corresponding 
modular objects $(\Delta, J)$ and consider the group 
\[ G := \U(\cM)\U(\cM')\{\1,J\} \subeq \AU(\cH).\] 
It is easy to see that this group 
permutes the standard subspaces in~$\cV$. 
From Proposition~\ref{prop:4.7}(ii) we  derive that 
\[ \mc(\cM) =  \{ gJg^{-1} \: g \in G \} \] 
is the conjugacy class of $J$ in $G$. 
We also note that the $G$-orbit 
$\{ g \cM g^{-1} \: g \in G \} = \{\cM,\cM'\}$ 
of $\cM$ in the set of von Neumann subalgebras of $B(\cH)$ consists only of two 
elements.

(b) Consider the group 
\[ G^\sharp := \U(\cM)\U(\cM') \gamma(\R^\times) 
\quad \mbox{ for  } \quad 
\gamma(-1) := J\quad \mbox{ and } \quad \gamma(e^t) := \Delta^{-it/2\pi}.\]  
That $G^\sharp$ is a group follows from the fact that $\gamma(\R^\times_+)$ normalizes 
$\U(\cM)$ and $\U(\cM')$, whereas conjugation by $J = \gamma(-1)$ exchanges both. 
This group is strictly larger than $\U(\cM)\U(\cM')\{\1,J\}$ 
if the modular automorphisms 
$\alpha_t(M) := \Delta^{it} M \Delta^{-it}$ of $\cM$ are not inner.  

If $\xi \in \cs(\cM)$ is different from $\Omega$, 
then Connes' Radon Nikodym Theorem\begin{footnote}{
See \cite[Thm.~III.4.7.5]{Bla06}, \cite[Thm.~5.3.34]{BR96}, 
and in particular \cite{Fl98} for a quite direct proof.}\end{footnote}
implies the existence of a strongly continuous path of unitaries 
$(u_t)_{t\in\R}$ in $\U(\cM)$ such that the corresponding 
modular automorphism group 
$\alpha_t^\xi(M) = \Delta_\xi^{it} M \Delta_\xi^{-it}$ 
satisfies 
\[ \alpha_t^\xi(M) = u_t\alpha_t(M)u_t^* 
\quad \mbox{ for } \quad M \in \cM,  t \in \R.\] 
This implies that 
$\Delta_\xi^{-it} u_t \Delta^{it} \in \U(\cM')$, so that 
$G^\sharp$ also contains the operators $\Delta_\xi^{it}$. 
Hence the net of standard subspaces of $\cH$ 
specified by the conjugacy class of the antiunitary 
representation $\gamma \in \Hom(\R^\times, G^\sharp)$ coincides with the orbit 
$G^\sharp V = GV\subeq \Stand(\cH)$. 
\end{ex} 

\subsection{Nets of von Neumann algebras} 
\mlabel{subsec:5.2}

The context that actually motivates
the consideration of families of standard subspaces  
are families $(\cM_\ell)_{\ell \in L}$ of von Neumann algebras on some 
Hilbert space~$\cH$. 
In the theory of algebras of local observables, one considers 
$\ell$ as indicating the ``laboratory'' in which observables corresponding 
to $\cM_\ell$ can be measured, and then $L$ is the set of laboratories 
(cf.~\cite{Ha96, Ar99, Bo97}). 

We write $\cM\subeq B(\cH)$ 
for the von Neumann algebra generated by all the algebras~$\cM_\ell$. 
We shall discuss several properties of these families and 
relate them to antiunitary representations and some results in 
Algebraic Quantum Field Theory (AQFT). 
Our first assumption is the {\it Reeh--Schlieder property}: 
\begin{itemize}
\item[\rm(RS)] There exists a unit vector $\Omega$ that is 
cyclic and separating for each $\cM_\ell$. 
\end{itemize}
By the Tomita--Takesaki Theorem, (RS) 
leads to a family of standard subspaces 
given 
\[ V_\ell := \oline{\cM_{\ell,h}\Omega}\] 
and the map $\ell \mapsto V_\ell$ is injective if and only if the map 
$\ell \mapsto \cM_\ell$ is injective 
(Lemma~\ref{lem:4.14}). This leads us to the setting of the 
preceding subsection, so that everything said there applies in particular 
here. As each $J_\ell$ fixes $\Omega$, it is fixed by the whole group~$\cJ$. 
For $g \in \cJ$ and 
$\cM_\ell^g := g\cM_\ell g^{-1}$, we therefore have 
$g V_\ell = \oline{\cM^g_{\ell,h}\Omega},$ so that Lemma~\ref{lem:4.14} implies that 
$gV_\ell = V_{\tilde\ell}$ for some $\tilde\ell \in L$ 
is equivalent to $g\cM_\ell g^{-1} = \cM_{\tilde\ell}$. 
Hence the {\it condition of geometric modular action} (CGMA) from the 
preceding section is equivalent to the following (\cite{BDFS00}): 
\begin{description}
\item[\rm(CGMA)] Conjugation with elements of the group 
$\cJ$ permutes the von Neumann algebras $(\cM_\ell)_{\ell \in L}$. 
\end{description} 
The relation $J_\ell \cM_\ell J_\ell = \cM_\ell'$ 
then implies that the net $(\cM_\ell)_{\ell \in L}$ is invariant under the 
passage to the commutant.

\begin{remark} \mlabel{rem:doubco} (a) In quantum field theory, where 
$L$ often is the set $\cW$ of wedges in Minkowski space $\R^{1,3}$, 
 \cite[Thm.~5.2.6]{BDFS00} asserts that (CGMA) 
basically is equivalent to the duality condition 
$\cM(W') = \cM(W)'$ for every $W \in \cW$.
Then one obtains an antiunitary 
representation of the Poincar\'e group $P(4)_+$ fixing 
$\Omega$ and acting covariantly on the net. 
Further, $U_{r_W} = J_W$ (cf.~Lemma~\ref{lem:4.17}) 
and the spectrum of the translation subgroup is either 
contained in $\oline{V_+}$ or in $-\oline{V_+}$, 
i.e., we either have positive or negative energy representations. 

(b) For $x,y \in \R^{1,3}$ and $x-y \in V_+$, the open causal interval 
\[ \cO_{x,y} := (x - V_+) \cap (y + V_+) \] 
is called a {\it double cone}. 
There are various Reeh--Schlieder Theorems,  
that provide sufficient conditions for the vacuum vector to be 
cyclic and separating for an algebras $\cM(\cO)$ of 
local observables attached to an open subset of Minkowski space 
(\cite{Bo92, RS61, Bo68}). The most classical results 
concern the cyclicity of the vacuum for 
double cone algebras $\cM(\cO_{x,y})$. 
Since every wedge contains double cones, the vacuum is also 
cyclic and separating for wedge algebras $\cM(W)$, $W \in \cW$. 
This leads to modular objects $(\Delta_W, J_W)$, so that 
the condition (RS) in \S \ref{subsec:5.2} holds for the 
index set~$L = \cW$. 

(c) For nets of von Neumann algebras $\cM(\cO)$ of local oberservables associated 
to regions $\cO$ in some spacetime $M$, it is important to specify those 
regions behaving well with respect to our assumptions. 
In \cite{Sa97} they are called {\it test regions}. This requires in particular 
that the vacuum vector $\Omega$ should be cyclic for $\cM(\cO)$ 
(the Reeh--Schlieder property) and that a suitable duality holds 
$\cM(\cO)' = \cM(\cO')$, where $\cO'$ is the (interior of the) causal complement of 
$\cO$. Prototypical examples of test domains are wedges $W$ in Minkowski space (or its conformal 
completion) \cite[Thm.~2.5]{BGL02}, but in many situations larger classes also have these properties, 
such as double cones or {\it spacelike cones}, i.e., 
translates of convex cones $\R_+ \cD$, where $\cD$ is a double cone not containing~$0$. 
In this context the CGMA is a natural additional requirement for test regions 
that ties the corresponding modular structure to spacetime geometry.  

(d) For a Haag--Kastler net $\cA(\cO)$ (as in the introduction), 
the (CGMA) for the net $\pi_\omega(\cA(\cO))''$ of von Neumann algebras 
specified by a state $\omega$ of $\cA$ can be seen as a requirement 
that selects states which are particularly natural 
(cf.~\cite[p.~485]{BDFS00}). 
\end{remark} 

Under certain assumptions on the corresponding 
net of local observables, the Bisog\-nano--Wichmann Theorem (\cite{BW76, So10}) 
asserts that the antiunitary representation 
$(U,\cH)$ of $P(4)_+$ obtained from the PCT Theorem, where 
$\Theta = U_{-\1}$ is the antiunitary PCT operator, 
has the property that the boost generator $b_0$ 
from  \eqref{eq:bostgen-d} in Example~\ref{ex:one-par}  
satisfies 
\[ \Delta_{W_R}^{-it/2\pi} = U(e^{tb_0}) \quad \mbox{ and } \quad 
J_{W_R} = U(r_{W_R}) = \Theta U(\diag(1,1,-1,-1)).\] 
The first relation is called the {\it modular covariance relation} 
(cf.\ \cite[p.~911]{Mu01}). 
In \cite[Props.~2.8,2.9]{GL95} Guido and Longo show that 
modular covariance implies covariance of the corresponding modular 
conjugations which in turn implies the PCT Theorem. 
In the context of standard subspaces, 
the Bisognano--Wichmann Theorem and the PCT Theorem 
were derived in by Mund (\cite[Thm.~5]{Mu01}). 

\begin{exs} \mlabel{exs:5.16} (Conformal invariance) 
Beyond the Bisognano--Wichmann Theorem, 
the following geometric implementations of modular automorphism 
groups are known: 

(a) In \cite{Bu78}, Buchholz shows that, for a free scalar massless field on 
$\R^{1,d-1}$ (which automatically enjoys conformal symmetry), for $d > 2$ 
the dilation group $\gamma_{V_+}(a)(x) = a x$, $a \in \R^\times$, 
corresponds to the modular objects of the light cone algebra $\cM(V_+)$. 
As we shall see in (b) below,  the light cone is conformally equivalent to 
the right wedge $W_R$. Therefore 
$\gamma_{V_+}$ is conjugate in the conformal group to the homomorphism 
\break $\gamma_{W_R} \:\R^\times \to \Conf(\R^{1,d-1})$, 
 corresponding to the right wedge $W_R$, which occurs 
in the Bisognano--Wichmann Theorem. 

(b) In \cite{HL82}, Hislop and Longo obtain similar results for 
double cones in the context of 
massless scalar fields by conjugating them conformally to light cones 
and then apply \cite{Bu78}. More concretely, the 
{\it relativistic ray inversion} 
\[ \rho \: 
x = (t,\bx) \mapsto \frac{1}{[x,x]}(t,\bx), \quad [x,x] = t^2 - \bx^2 \] 
(which is an involution), exchanges the translated right wedge 
\[ W_R + \frac{r}{2} e_1 
= \Big\{ (x_0, \bx) \: x_1 > \frac{r}{2} + |x_0| \Big\} \] 
with the double cone 
\[ \cO_{\frac{e_0-e_1}{r}, \frac{-e_0-e_1}{r}} = 
\Big(\frac{e_0 - e_1}{r} - V_+\Big) \cap \Big(-\frac{e_0 + e_1}{r} + V_+\Big).\] 
It also exchanges the double cone 
$\cO_{r e_0, 0} = (r e_0 - V_+) \cap V_+$ and the light cone  
$\frac{e_0}{r} + V_+$ 
(see \cite[p.~111]{Gu11}).
With these explicit transformations, one also obtains the corresponding 
one-parameter groups of automorphisms and the corresponding 
conformal involutions. For the light cone $V_+$, we know from (a) that 
the corresponding automorphism group is given by the dilations 
$\gamma_{V_+}(t)x = t x$. So it follows in particular, that it is conformally 
conjugate to the Lorentz boosts $\gamma_W \: \R^\times \to P(4)_+$ 
corresponding to a wedge $W$ (Lemma~\ref{lem:4.17}). 

As a consequence of this discussion, the modular automorphism groups 
corresponding to the local observable algebras associated to 
double cones, light cones and wedges are conjugate under 
the conformal group $\Conf(\R^{1,d-1}) \cong \OO_{2,d}(\R)/\{\pm \1\}$. 
In particular, they  
correspond to a single conjugacy class of homomorphism 
$\gamma \: \R^\times \to \Conf(\R^{1,d-1})$ which is most simply 
represented by $\gamma_{V_+}$. 
\end{exs}

\begin{ex} (cf.\ Example~\ref{ex:proj-grp}) 
In the one-dimensional Minkowski space 
$\R$, the order intervals are represented by the open interval 
$(-1,1)$ transformed by the Cayley map $c(x) := \frac{1 + x}{1-x}$ to $(0,\infty) = V_+$ 
and the involution $\sigma(x) := x^{-1}$ maps $(-1,1)$ to its (conformal) 
complement. These are the geometric transformations corresponding to the modular 
operators on the double cone algebra $\cM(\cO_{1,-1})$ for $d = 1$. 
\end{ex}

\begin{ex} Interesting examples of nets of von Neumann algebras 
with (CGMA) arise from \cite[Thm.~4.3.9]{BDFS00}, 
where the index set is the set $\cW$ of wedges in $\R^{1,3}$. 
Under suitable continuity assumptions, one obtains a 
continuous antiunitary representation of $P(4)_+$ with 
\[ U_{r_W} = J_W \quad \mbox{ and } \quad \cJ = U_{P(4)_+}.\] 
Here a key point is that $P(4)_+$ 
is generated by the conjugacy class of the wedge reflection $r_{W_R}$ 
(Lemma~\ref{lem:5.8}). 
\end{ex}

\begin{remark}
A key observation in the work of Borchers and Wiesbrock 
is that von Neumann algebras of local observables corresponding to two 
wedges having a light ray in common define modular intersections 
(\cite[Prop.~7]{Wi98}). That one can deal with them as pairs without 
any direct reference to the intersection (cf.\ Theorem~\ref{thm:wies3-standard}) 
is crucial because the modular group of the intersection need not be implemented 
geometrically \cite{Bo96}. This is of particular interest for QFT on de Sitter 
space $\dS^d$ whose isometry group $\OO_{1,d}(\R)$ has no positive energy 
representations for $d > 2$. 
\end{remark}

\section{Second quantization and modular localization} 
\mlabel{sec:6}

In this section we explain how 
Second Quantization, i.e., the passage from a (one-particle) Hilbert space 
$\cH$ to the corresponding Fock spaces $\cF_\pm(\cH)$ 
(bosonic and fermionic) provides for 
each standard subspace $V \subeq \cH$ 
pairs $(\cR^\pm(V),\Omega)$, where $\cR^\pm(V)$ is a von Neumann algebra 
on $\cF_\pm(\cH)$ and the vacuum vector 
$\Omega \in \cF_\pm(\cH)$ is cyclic. 

Let $\cH$ be a complex Hilbert space and let  
\[ \cF(\cH) := \hat\bigoplus_{n = 0}^\infty \cH^{\hat\otimes n}  \] 
be the full {\it Fock space over $\cH$}. 
We write $\cF_+(\cH)$ for the subspace of symmetric tensors, 
the {\it bosonic Fock space}, and 
$\cF_-(\cH)$ for the subspace of skew-symmetric tensors, 
the {\it fermionic Fock space}. 
Both spaces carry a natural representation 
\[ \Gamma_\pm \: \AU(\cH) \to \AU(\cF_\pm(\cH)) \] 
of the antiunitary group $\AU(\cH)$  given by 
\[ \Gamma_+(U)(v_1 \vee \cdots \vee v_n) 
:= Uv_1 \vee \cdots \vee Uv_n, \qquad 
 \Gamma_-(U)(v_1 \wedge \cdots \wedge v_n) 
:= Uv_1 \wedge \cdots \wedge Uv_n.\] 
Moreover, the bosonic Fock space 
carries a unitary representation of the Heisenberg
group $\Heis(\cH)$ (\S \ref{subsec:7.1}) and its subgroups 
can be used to derive a net of von Neumann algebras 
on $\cF_+(\cH)$. A similar construction can be carried out 
for the fermionic Fock space in terms of the 
natural representation of the $C^*$-algebra $\CAR(\cH)$, a 
$C^*$-algebra defined by the {\it canonical anticommutation operators}. 
Both constructions are functorial and associate to every 
antiunitary representation $(U,\cH)$ of $(G,G_1)$ on $\cH$ a 
covariant family $(\cM_\gamma)_{\gamma \in \Gamma}$ 
of von Neumann algebras on $\cF_\pm(\cH)$, 
where $\Gamma$ is as in Proposition~\ref{prop:antiunirep-stand}.

\subsection{Bosonic Fock space} 
\mlabel{subsec:7.1}

We start with the construction of the von Neumann algebras on the bosonic Fock space. 
For $v_1, \ldots, v_n \in \cH$, we define 
\[  v_1 \cdots v_n := v_1 \vee \cdots \vee v_n := 
\frac{1}{\sqrt{n!}} \sum_{\sigma \in S_n} v_{\sigma(1)} \otimes \cdots \otimes v_{\sigma(n)} \] 
and $v^n := v^{\vee n}$, so that 
\begin{eqnarray} 
  \label{eq:symprod}
\la v_1 \vee \cdots \vee v_n, w_1 \vee \cdots \vee w_n \ra 
&=&   \sum_{\sigma \in S_n} \la v_{\sigma(1)}, w_1 \ra 
\cdots \la v_{\sigma(n)}, w_m \ra. 
\end{eqnarray} 
For every $v \in \cH$, the series 
$\Exp(v) := \sum_{n = 0}^\infty \frac{1}{n!} v^n$ 
defines an element in $\cF_+(\cH)$ and the scalar product of two 
such elements is given by 
\[ \la \Exp(v), \Exp(w) \ra 
= \sum_{n = 0}^\infty \frac{n!}{(n!)^2} \la v, w\ra^n = e^{\la v, w \ra}.\] 
These elements span a dense subspace of $\cF_+(\cH)$, and therefore 
we have for each $x \in \cH$ a unitary operator on $\cF_+(\cH)$ determined by the 
relation 
\begin{equation}
  \label{eq:Ux-ops}
 U_x \Exp(v) = e^{ -\la x, v\ra - \frac{\|x\|^2}{2}} \Exp(v+x) \quad 
\mbox{ for } \quad x,v \in \cH.   
\end{equation}
A direct calculation then shows that 
\begin{equation}
  \label{eq:comm-rel-U}
U_x U_y = e^{-i\Im \la x, y \ra} U_{x+y} \quad \mbox{ for  } \quad 
x, y \in \cH.
\end{equation}
To obtain a unitary representation, we have to replace the 
additive group of $\cH$ by the {\it Heisenberg group}
\[ \Heis(\cH) := \T\times \cH \quad \mbox{ with } \quad 
(z,v)(z',v') := (zz' e^{-i\Im \la v,v' \ra}, v + v').  \] 
For this group, we obtain with \eqref{eq:comm-rel-U} a unitary representation 
\[ U \: \Heis(\cH) \to \U(\cF_+(\cH)) \quad \mbox{ by } \quad U_{(z,v)} := z U_v.\] 
In this physics literature, all this is expressed in terms of the 
so-called {\it Weyl operators} 
\[ W(v) := U_{iv/\sqrt{2}}, \qquad v \in \cH \] 
satisfying the {\it Weyl relations} 
\begin{equation}
  \label{eq:weyl}
  W(v) W(w) = e^{-i \Im \la v,w \ra/2} W(v+w), \qquad v,w \in \cH.
\end{equation}

\begin{definition} To each real subspace $V \subeq \cH$, we assign the 
von Neumann algebra $\cR(V) := \cR^+(V) := W(V)'' \subeq B(\cF_+(\cH))$ on the bosonic Fock 
space of $\cH$.   
\end{definition}

\begin{lemma} \mlabel{lem:6.3} We have 
  \begin{itemize}
  \item[\rm(i)] $\cR(\cH) = B(\cF_+(\cH))$, resp., the representation 
of $\Heis(\cH)$ on $\cF_+(\cH)$ is irreducible. 
  \item[\rm(ii)] $\cR(V) \subeq \cR(W)'$ if and only if 
$V \subeq W'$ (locality). 
  \item[\rm(iii)] $\cR(V) = \cR(\oline V)$. 
\item[\rm(iv)] $\Omega= \Exp(0) \in \cF_+(\cH)$ is cyclic for $\cR(V)$ if and only 
if $V + i V$ is dense in $\cH$.
  \item[\rm(v)] $\Omega  \in \cF_+(\cH)$ is separating 
for $\cR(V)$ if and only if $\oline V \cap i \oline V = \{0\}$. 
\item[\rm(vi)] $\Omega \in \cs(\cR(V))$ if and only if $\oline V$ is standard. 
  \end{itemize}
\end{lemma}

\begin{proof} (i) is well-known (\cite[Prop.~5.2.4(3)]{BR96}). 

(ii) follows directly from the Weyl relations \eqref{eq:weyl}. 

(iii) follows from the fact that $\cH \to B(\cF_+(\cH)), v \mapsto W_v$ 
is strongly continuous and $\cR(V)$ is closed in the weak operator topology. 

(iv) Assume that $\cK:=\overline{V+iV}\not= \cH$. 
Then $\cR(V)\Omega \subeq \cF_+(\cK)$, so that $\Omega$ cannot be cyclic. 

Suppose, conversely, that $\cK = \cH$ and that $f \in (\cR(V)\Omega)^\bot$. 
Then the holomorphic function $\hat f(v) := \la f, \Exp(v) \ra$ on 
$\cH$ vanishes on $V$, hence also on $V + i V$, and since this subspace is 
dense in $\cH$, we obtain $f = 0$ because $\Exp(\cH)$ is total in $\cF_+(\cH)$. 
We conclude that $\Omega$ is cyclic. 

(v) In view of (iii), we may assume that $V$ is closed. 
Let $0 \not= w \in V \cap i V$. To see that $\Omega$ is not separating 
for $\cR(V)$, it suffices to show that, for the one-dimensional Hilbert space 
$\cH_0 := \C w$, the vector $\Omega$ is not separating for 
$\cR(\C w) = B(\cF_+(\C w))$ (which follows from the irreducibility 
of the representation of $\Heis(\C w)$ on $\cF_+(\C w)$). 
This is obviously the case because $\dim \cF_+(\C w) > 1$. 

Suppose that $\cK = \{0\}$. As $\cK = V'' \cap (iV'') = (V' + i V')'$, 
it follows that $V' + i V'$ is dense in $\cH$. By (ii), 
$\Omega$ is cyclic for $\cR(V')$ which commutes with $\cR(V)$. Therefore 
$\Omega$ is separating for $\cR(V)$. 

(vi) follows from (iv) and (v). 
\end{proof}

\begin{remark} (a) $\cR(V)$ is commutative if and only if 
$V \subeq V'$. For a standard subspace $V$ the relation 
$V' = JV$ shows that this is equivalent to $V = V'$, respectively 
to $\Delta = \1$ (Lemma~\ref{lem:stand-factorial}). 

(b) The imaginary 
part $\omega(\xi,\eta) := \Im \la \xi,\eta\ra$ turns $\cH$ into a symplectic manifold 
$(\cH,\omega)$. From this perspective, we may consider the algebras 
$\cR(V)$ as ``quantizations'' of the algebra of measurable functions on the 
Lagrangian subspace $E := \cH^J$. If $V = V'$, then 
$\cF_+(\cH) \cong L^2(E^*, \gamma)$, where 
$\gamma$ is a Gaussian probability measure on the algebraic dual space $E^*$ of $E$, 
endowed with the smallest $\sigma$-algebra for which all evaluation maps are 
measurable. Then the commutative von Neumann algebra $\cR(V)$ is isomorphic to 
$L^\infty(E^*,\gamma)$. In general, if $V \not\subeq V'$, then $\cR(V)$ is non-commutative 
and the degree of non-commutativity depends on the non-degeneracy of $\omega$ on $V$. 
It is ``maximal'' if $V \cap V' = \{0\}$, which implies that $\cR(V)$ is a factor 
by the following theorem. 
\end{remark}

\begin{theorem} \mlabel{thm:araki-1} {\rm(\cite[Thm.~1]{Ar63})} 
For closed real subspaces $V,W, V_j$ of $\cH$, the following assertions hold: 
\begin{itemize}
\item[\rm(i)] $\cR(V) \subeq \cR(W)$ if and only if $V \subeq W$ (isotony). 
\item[\rm(ii)] $R\big(\bigvee_{j \in J} V_j\big) = \bigvee_{j \in J} \cR(V_j)$, 
where $\bigvee_{j \in J} V_j$ denotes the closed subspace generated by the 
$V_j$ and $\bigvee_{j \in J} \cR(V_j)$ denotes the von Neumann algebra generated by 
the $\cR(V_j)$. 
\item[\rm(iii)] $R\big(\bigcap_{j \in J} V_j\big) = \bigcap_{j \in J} \cR(V_j)$. 
\item[\rm(iv)] $\cR(V)' = \cR(V')$ (duality). 
\item[\rm(v)] $\cR(V) \cap \cR(V') = \cR(V \cap V')$. In particular, 
the algebra $\cR(V)$ is a factor if and only if $V \cap V' =\{0\}$. 
\end{itemize}
\end{theorem}

\subsection{Fermionic Fock space} 
\mlabel{subsec:7.2}

On the fermionic Fock space, the construction of the von Neumann algebras is slightly 
different but similar in spirit. 

For $v_1, \ldots, v_n \in \cH$, we define 
\[  v_1 \wedge \cdots \wedge v_n := 
\frac{1}{\sqrt{n!}} \sum_{\sigma \in S_n} \sgn(\sigma) 
v_{\sigma(1)} \otimes \cdots \otimes v_{\sigma(n)}, \] 
so that 
\begin{eqnarray} 
  \label{eq:altprod}
\la v_1 \wedge \cdots \wedge v_n, w_1 \wedge  \cdots \wedge w_n \ra 
&=&   \sum_{\sigma \in S_n} \sgn(\sigma) \la v_{\sigma(1)}, w_1 \ra 
\cdots \la v_{\sigma(n)}, w_m \ra
\end{eqnarray} 
In $\cF_-^0(\cH) \cong \C$ we pick a unit vector $\Omega$, called the {\it vacuum}.

\begin{definition}
The {\it CAR-algebra} $\CAR(\cH)$ of $\cH$ is a $C^*$-algebra, 
together with a continuous antilinear map
$a \: \cH \to \Car(\cH)$ satisfying the
{\it canonical anticommutation relations}
\begin{equation}
  \label{eq:car}
\{a(f), a(g)^*\}  = \la f,g \ra \1
\quad \hbox{ and } \quad  \{a(f),a(g)\} = 0
\quad \mbox{ for } \quad f, g \in \cH
\end{equation}
and which has the universal property that, 
for any $C^*$-algebra $\cA$ and any antilinear map 
$a' \:  \cH \to \cA$ satisfying the above anticommutation relations, 
there exists a unique homomorphism 
$\phi \: \CAR(\cH) \to \cA$ with $\phi \circ a  =a'$. 
This determines the pair $(\Car(\cH),a)$ 
up to isomorphism (\cite[Thm.~5.2.8]{BR96}). 
We  write $a^*(f) := a(f)^*$ and observe that this defines a complex 
linear map $a^* \: \cH \to \Car(\cH)$. 
\end{definition} 

\begin{remark} \mlabel{rem:10.1} 
The $C^*$-algebra $\Car(\cH)$ has an irreducible 
representation 
$(\pi_0, \cF_-(\cH))$ 
on the fermionic Fock space $\cF_-(\cH)$ (\cite[Prop.~5.2.2(3)]{BR96}). 
The image $c(f) := \pi_0(a(f))$ 
acts by $c(f)\Omega=0$ and 
\[  c(f)(f_1 \wedge \cdots\wedge f_n) 
= \sum_{j = 1}^n (-1)^{j-1} \la f, f_j \ra 
f_1 \wedge \cdots \wedge f_{j-1} \wedge f_{j+1} \wedge \cdots \wedge f_n. \] 
Accordingly, we have 
$$ c^*(f)\Omega = f \quad \mbox{ and } \quad 
c^*(f)(f_1 \wedge \cdots \wedge f_n) 
= f \wedge f_1 \wedge \cdots \wedge f_n. $$
\end{remark}
Consider the hermitian operators 
\begin{equation}
  \label{eq:6.1}   
b(f) := c(f) + c^*(f) \in \CAR(\cH) 
\end{equation}
and note that 
\begin{equation}
  \label{eq:6.2}
\{b(f), b(g)\} = \{c(f), c^*(g)\} + \{c^*(f), c(g)\} 
= \la f,g \ra \1 + \la g, f \ra \1 = 2 \beta(f,g) \1,
\end{equation}
where 
\[ \beta(f,g) = \Re \la f, g \ra \quad \mbox{ for } \quad f,g \in \cH\] 
is the real scalar product on $\cH$.

\begin{definition}
Let $\cH= \cH_{\oline 0} \oplus \cH_{\oline 1}$ be a $2$-graded Hilbert space. 
Accordingly, $B(\cH)$ inherits a grading and therefore a 
{\it Lie superbracket} which on homogeneous elements is given by 
\[ [A,B]_\tau := AB - (-1)^{|A| |B|} BA, \] 
where $|A|$ denotes the degree of a homogeneous element~$A$.
For a subset $E \subeq B(\cH)$, we accordingly define the {\it super-commutant} by 
\[ E^\sharp := \{ A \in B(\cH) \: (\forall M \in E)\, [A,M]_\tau = 0\}.\] 
\end{definition}

For each homogeneous $M \in B(\cH)$, the operator $D_M(A) := [M,A]_\tau$ is a 
superderivation of the $\Z_2$-graded associative algebra $B(\cH)$ in the sense that 
\begin{equation}
  \label{eq:6.3}
D_M(AB) = D_M(A)B + (-1)^{|M| |A|} A D_M(B).
\end{equation}
It follows in particular that, if $E$ is spanned by homogeneous elements, 
then $E^\sharp$ is a von Neumann algebra adapted to the 
$2$-grading of $B(\cH)$. 
Let $Zv = (-1)^{|v|} v$ ($|v| \in \{0,1\}$) 
denote the parity operator on $\cH$ and 
$\tilde Z v = (-i)^{|v|} v$ (also known as 
the {\it Klein twist} $\tilde Z = \frac{1 + iZ}{1+i\1}$) 
which satisfies $\tilde Z^2 =Z$. For $A$ and $M$ odd we then have 
\[ [\tilde Z^{\pm 1} A \tilde Z^{\mp 1}, M] = \pm i Z\{A,M\} = -i Z[A,M]_\tau.\] 
This leads to 
\begin{equation}
  \label{eq:6.4}
E^\sharp = \tilde Z E' \tilde Z^{-1} = \tilde Z^{-1} E' \tilde Z\end{equation}
for any graded subspace  $E \subeq B(\cH)$.

As in \cite{Fo83}, we associated to every real 
linear subspace $V \subeq \cH$ a von Neumann subalgebra 
\[ \cR(V) :=  \cR^-(V) := b(V)'' \subeq B(\cF_-(\cH)).\] 
We list some properties of this assignment (cf.\ \cite[Prop.~2.5]{Fo83} for (iv) and (v)): 

\begin{lemma} \mlabel{lem:ferm-dual} We have 
  \begin{itemize}
  \item[\rm(i)] $\cR(\cH) = B(\cF_-(\cH))$, resp., the representation 
of $\Car(\cH)$ on $\cF_-(\cH)$ is irreducible. 
  \item[\rm(ii)] $\cR(V) = \cR(\oline V)$. 
\item[\rm(iii)] $\cR(V)$ and $\cR(W)$ super-commute if and only if $V \bot_\beta W$ 
(twisted duality). 
\item[\rm(iv)] The vacuum $\Omega$ is cyclic for $\cR(V)$ if and only if 
$V + i V$ is dense in $\cH$. 
\item[\rm(v)] The vacuum $\Omega$ is separating for $\cR(V)$ if and only if 
$\oline V \cap i \oline V = \{0\}$. 
\item[\rm(vi)] $\Omega \in \cs(\cR(V))$ if and only if $\oline V$ is standard. 
  \end{itemize}
\end{lemma}

\begin{proof} (i) is well-known (\cite[Prop.~5.2.2(3)]{BR96}). 

(ii) follows from the fact that 
$b \: \cH \to B(\cF_-(\cH))$ is continuous. 

(iii) follows immediately from \eqref{eq:6.2}. 

(iv) We explain how this can be derived from 
\cite[Prop.~3.4]{BJL02}, where a different setting is used: 
Consider a conjugation $\Gamma$ on a complex 
Hilbert space $\cK$ and a corresponding basis projection $P$, i.e., 
$\Gamma P \Gamma = \1 - P$. For $v \in \cK^\Gamma$ we then have 
the orthogonal decomposition $v = Pv + (\1 - P)v$, where both summands 
are exchanged by $\Gamma$, hence have the same length. 
Therefore the map 
\[ \Phi \: \cK^\Gamma \to P\cK, \quad 
\Phi(v) = \sqrt 2 Pv \] 
is an isometry between the real Hilbert space $\cK^\Gamma$ and the 
complex Hilbert space $\cH := P\cK$. 
The antilinear map 
\[ a \: \cK \to \CAR(\cH), \qquad 
a(f) := c^*(P\Gamma f) + c(Pf) \] 
then satisfies 
\[ a(\Gamma f) = a(f)^*\quad \mbox{ for }\quad f \in \cK\] 
and $a$ is the unique antilinear extension of the map 
$a\res_{\cK^\Gamma} = b \circ P \: \cK^\Gamma \to \CAR(\cH).$ 

For any $\Gamma$-invariant subspace $\cV \subeq \cK$, we therefore 
have 
\begin{equation}
  \label{eq:ac-rel}
a(\cV) = b(P\cV^\Gamma)_\C = b(\Phi(\cV^\Gamma))_\C 
\end{equation}
and thus, for the real subspace $V := \Phi(\cV^\Gamma) = P(\cV^\Gamma) \subeq \cH$,  
\begin{equation}
  \label{eq:ac-rel2}
a(\cV^\Gamma)'' = a(\cV)'' = b(\Phi(\cV^\Gamma))'' = \cR(V)''. 
\end{equation}
As $V + i V = P(\cV^\Gamma_\C) = P(\cV),$ 
\cite[Prop.~3.4]{BJL02} implies that 
$P(\cV)$ is dense in $P(\cH)$ if and only if $\Omega$ is $\cR(V)$-cyclic, and 
(iv) follows. 

(v) In view of (ii), we may assume that $V$ is closed. 
Let $0 \not= w \in W := V \cap i V$. To see that $\Omega$ is not separating 
for $\cR(V)$, it suffices to show that, for the one-dimensional Hilbert space 
$\cH_0 := \C w$, the vector $\Omega$ is not separating for 
$\cR(\C w) = B(\cF_-(\C w))$. This follows from the irreducibility 
of the representation of $\Car(\C w) \cong M_2(\C)$ on $\cF_-(\C w)\cong \C^2$ 
which has no separating vector (see (i)). 

Suppose, conversely, that $W = \{0\}$. As $W = (V^\bot+ i V^\bot)^\bot$, 
the subspace  $V^\bot + i V^\bot$ is dense in $\cH$. By (iii), 
$\Omega$ is cyclic for $\cR(V^\bot)$ which anticommutes with $\cR(V)$. Therefore 
$\Omega$ is separating for $\cR(V)$. 
\end{proof}

The following theorem is the fermionic version of 
the duality result in Theorem~\ref{thm:araki-1}(iii) 
(\cite[Thm.~7.1]{BJL02}, \cite[Thm.~2.4(v)]{Fo83}). 

\begin{theorem}[Fermionic Duality Theorem] \mlabel{thm:6.9} 
\[ \cR(V^{\bot_\beta}) = \cR(V)^\sharp  = \{ A \in B(\cF_-(\cH)) \: (\forall v \in V) [A, b(v)]_\tau = 0\} 
= \tilde Z^{-1} \cR(V)' \tilde Z \] 
for every real linear subspace $V \subeq \cH$. 
\end{theorem}

To match our notation with Foit's in \cite{Fo83}, we note that Foit's operator \break 
${V := \frac{1}{\sqrt 2}(\1 - i Z)}$ satisfies $V = e^{-\pi i/4}\tilde Z^{-1}$, 
so that $\tilde Z^{-1} A \tilde Z = V A V^*$ for every operator~$A$ on $\cF_-(\cH)$.

\subsection{From antiunitary representations to local nets} 

For a closed real subspace $V$ of the Hilbert space $\cH$, 
we write $\cR^\pm(V) \subeq B(\cF_\pm(\cH))$ for the associated 
von Neumann algebras on the bosonic and fermionic Fock space. 

\begin{proposition} \mlabel{prop:6.9}
For a closed real subspace $V \subeq \cH$, 
the vacuum $\Omega$ is cyclic and separating for the 
von Neumann algebras $\cR^\pm(V)\subeq B(\cF_\pm(\cH))$ if and only if 
$V$ is a standard subspace of $\cH$. 
The corresponding modular objects $(\Delta_V^\pm, J_V^\pm)$ on $\cF_\pm(\cH)$ 
are obtained by second quantization from the modular objects 
$(\Delta_V, J_V)$ associated to~$V$, in the sense that 
\begin{equation}
  \label{eq:6.1b}
\Delta_V^\pm= \Gamma_\pm(\Delta_V), \quad 
J_V^+ = \Gamma_+(J_V) \quad \mbox{ and }\quad 
J_V^- = \tilde Z \Gamma_-(i J_V).
\end{equation}
\end{proposition}

\begin{proof} The first assertion follows from 
Lemmas~\ref{lem:6.3} and \ref{lem:ferm-dual}. 
For the identification of the modular objects, 
we refer to \cite[Thm.~1.4]{FG89} (see also \cite{EO73}) 
in the bosonic case and to 
\cite[Prop.~2.8]{Fo83} for the fermionic case 
(see also \cite[Cor.~5.4]{BJL02}, \cite[Thm.~4.13]{Lle09}). 
\end{proof}

\begin{remark} (a) The twists arising in Theorem~\ref{thm:6.9} and Proposition~\ref{prop:6.9} 
arise from the fact that the fermionic situation has to take the 
$2$-grading  on $\cF_-(\cH)$ into account. In particular Theorem~\ref{thm:6.9} takes its most natural 
form $\cR(V^{\bot_\beta}) = \cR(V)^\sharp$ if the commutant is defined in terms of the 
super bracket.

(b) If $\cM$ is a $\Z_2$-graded von Neumann algebra on the $\Z_2$-graded Hilbert 
space $\cH = \cH_0 \oplus \cH_1$ and $\Omega \in \cH_0$ is a cyclic separating vector, then 
the theory of Lie superalgebras suggests to consider the antilinear involution 
$(x_0 + x_1)^\sharp := x_0^* - i x_1^*$ instead of the operator adjoint. Then the 
corresponding {\it unitary Lie superalgebra} is 
\[ \fu(\cM) = \{ x \in \cM \: x^\sharp = - x\} 
= \{ x = x_0 + x_1 \in \cM \: x_0^* = -x_0, x_1^* = -ix_1\}.\] 
Accordingly, modular theory can be based on the unbounded antilinear operator defined by 
$\tilde S(M\Omega) := M^\sharp \Omega = \tilde Z S(M\Omega)$  for $M \in \cM$. 
The polar decomposition $\oline{\tilde S} = \tilde J \Delta^{1/2}$ 
results in the pair $(\tilde J, \Delta)$ of modular objects, where 
$\Delta$ is unchanged, but $\tilde J = \tilde Z J$. This leads to the relation 
\[ \tilde J \cM \tilde J = \tilde Z \cM' \tilde Z^{-1} = \cM^\sharp,\] 
which is a super version of $J\cM J = \cM'$. 

We also obtain with \eqref{eq:6.1b} 
\[ \tilde {J_V^-} = \tilde Z^2 \Gamma_-(iJ_V) 
= Z \Gamma_-(i J_V)= \Gamma_-(-i J_V).\] 
To obtain a situation where the modular objects on $\cF^-(\cH)$ are simply given 
by second quantization, one may consider the von Neumann algebras 
$\tilde \cR^-(V) := \cR^-(\zeta V)$ for $\zeta := e^{\pi i/4}$ instead. The standard subspace 
$\tilde V := \zeta V$ satisfies $\Delta_{\tilde V} =\Delta_V$ and 
$J_{\tilde V} = i J_V$, so that the modular conjugation corresponding to 
$\tilde \cR^-(V)$ is 
\[ \tilde {J_{\zeta V}^-} = \Gamma_-(-i J_{\zeta V}) = \Gamma_-(J_V).\] 
\end{remark}

\begin{remark} Let $(U,\cH)$ be an antiunitary representation of 
$(G,G_1)$ on $\cH$ and \break 
$\gamma \: \R^\times \to G$ be a homomorphism with 
$\gamma(-1) \not\in G_1$, so that it specifies 
a standard subspace $V_\gamma \subeq \cH$ 
(Proposition~\ref{prop:antiunirep-stand}). 
Consider the antiunitary representation 
\[ \Gamma_\pm \: \AU(\cH) \to \AU(\cF_\pm(\cH)) \] 
of the antiunitary group of $\cH$ on the corresponding Fock spaces. 
Then $\Gamma_\pm \circ U$ is an antiunitary representation 
of $(G, G_1)$ on $\cF_\pm(\cH)$, so that we also obtain 
standard subspaces $V_\gamma^\pm \subeq \cF_\pm(\cH)$. 
The pair $(\cR^+(V_\gamma), \Omega)$ then satisfies 
\[ V_\gamma^+ = \oline{\cR^+(V_\gamma)_h \Omega},\] 
and in the fermionic case the pair $(\cR^-(V_\gamma), \Omega)$ leads to the 
correct modular operator $\Delta_{V_\gamma}^-$, but to the modular conjugation 
$\tilde Z \Gamma_-(i J_{V_\gamma}).$ 
\end{remark}

\section{Perspectives} 
\mlabel{sec:7} 

For a detailed exposition of the results mentioned below, 
we refer to the forthcoming paper \cite{NO17}.

\subsection{The Virasoro group}

In $\Diff(\bS^1)$ we consider the involution on 
$\bS^1 \cong \T \subeq \C$, given by 
$r(z) = \oline z$. We consider the group $G := \Diff(\bS^1) 
\cong \Diff(\bS^1)_0  \rtimes \{\1,r\}$. 
One can show that all projective unitary positive energy 
representations of $\Diff(\bS_1)_0$ extend naturally to projective 
antiunitary representations of $G$. To obtain 
antiunitary representations, one has to replace $G$ by 
a central extension $\Vir \rtimes \{\1,r\}$, where 
$\Vir$ is the simply connected Virasoro group. 

Another closely related  ``infinite dimensional'' group that occurs in the context of 
modular localization is the free product 
$\PSL_2(\R) *_{\Aff(\R)_0}  \PSL_2(\R)$ of two copies of $\PSL_2(\R)$ 
over the connected affine group (\cite{GLW98}).

\subsection{Euclidean Jordan algebras} 
\mlabel{subsec:7.2b}

Minkowski spaces are particular examples of simple  euclidean 
Jordan algebras, namely those of rank~$2$ (cf.~\cite{FK94}).
Many of the geometric structures of Minkowski spaces and their conformal 
completions are also available for general simple euclidean 
Jordan algebras, where the role of the lightcone $V_+$ is played 
by the open cone of invertible squares. There also exists a 
natural causal compactification $\hat V$ which carries a causal 
structure. The corresponding conformal group $G := \Conf(V)$ 
has an index $2$ subgroup $G_1$ preserving the causal structure on $\hat V$; 
other group elements reverse it. 
In $\hat V$, the set $\cW^c := \{ g V_+ \: g \in \Conf(V)\}$ 
specializes for Minkowski spaces to the set of conformal wedge 
domains, which include in particular the light cone and double cones 
(cf.~Example~\ref{exs:5.16}). Moreover, the homomorphism 
\[ \gamma_{V_+} \:  \R^\times \to \GL(V), \quad 
\gamma(t)v = tv\]
is naturally specified because $\gamma_{V_+}(\R^\times_+)$ is central 
in the identity component of the stabilizer $G_{V_+}$. 
Therefore any antiunitary positive energy representation of $G$ yields a net 
of standard subspaces indexed by $\cW^c$. 
In \cite{NO17} we obtain a classification of these representations 
along the lines of \S \ref{subsec:2.2}. 
The subsemigroups $S_{V_+} := \{ g \in G \: gV_+ \subeq V_+\} \subeq G$ 
also leads to a natural generalization of Borchers pairs in this context.

\subsection{Hermitian groups} 
\mlabel{subsec:7.3}

The conformal group $\Conf(V)$ of a euclidean Jordan algebra 
can be identified with the group $\AAut(T_{V_+})$ of 
holomorphic and antiholomorphic automorphisms 
of the corresponding tube domain $T_{V_+} = V_+ + i V$. 
This suggests that some of the crucial structure relevant for 
antiunitary representations can still be obtained for the 
groups $G := \AAut(\cD)$ of all holomorphic and 
antiholomorphic automorphisms of a bounded symmetric domain $\cD$. 
The irreducible antiunitary positive energy 
representations can also be parametrized in a natural way 
by writing $G = G_1 \rtimes \{\id,\sigma\}$, where 
$\sigma$ is an antiholomorphic involution of $\cD$ (\cite{NO17}).
Here there are many homomorphisms $\gamma \: (\R^\times, \R^\times_+) \to (G,G_1)$ 
with $\gamma(-1) = \sigma$, but one cannot expect 
$\gamma(\R^\times_+)$ to be central in 
$G^\sigma$, which can be achieved 
for tube type domains 
(coming from euclidean Jordan algebras). 

\subsection{Analytic extension} 
\mlabel{subsec:7.3b}

We have seen that, for antiunitary representations 
of $\Aff(\R)$, the positive energy condition appears quite naturally 
from the order structure on the set of standard subspaces. 
If $(U,\cH)$ is an antiunitary representation of $G$ 
containing copies of $\Aff(\R)$ coming from 
half-sided modular inclusions, it follows that 
the closed convex cone 
\[ C_U := \{ x \in \g \: -i\dd U(x) \geq 0\} \] 
is non-trivial. 
This further leads to an analytic extension of the representation to the domain 
$G \exp(i C_U)$ (see \cite{Ne00} for details on this process). 

On the other hand, antiunitary representations of 
$\R^\times$ correspond to modular objects $(\Delta, J)$ and the 
orbit maps of elements $v \in V$ extend to the strip 
$\{ z \in \C \: 0 \leq \Im z \leq \pi\}$ 
(Remark~\ref{rem:anaext}). 
Composing families of homomorphisms $\gamma \: \R^\times \to G$ 
with an antiunitary representation, 
we therefore expect analytic continuation of $U$ to natural 
complex domains containing $G$ in their boundary. 

It would be very interesting to combine these two types of 
analytic continuations in a uniform manner, in the same spirit 
as the KMS condition is a generalization of the ground state condition 
(corresponding to positive energy) (\cite{BR96}).  One may further expect 
that this leads to ``euclidean realizations'' 
of antiunitary representations of $G$ by unitary representations 
of a Cartan dual group in the sense of the theory of reflection positivity 
developed in in \cite{NO14, NO16}; see also \cite{Sch06} for relations 
with modular theory. Maybe it can even be combined 
with the analytic extensions to the crown of a Riemannian 
symmetric space (\cite{KS05}). 

\subsection{Geometric standard subspaces} 
\mlabel{subsec:7.3c}

In QFT, the algebras $\cR(V)$ are supposed to correspond to regions in some 
spacetime $M$. Therefore one looks for standard subspaces $V(\cO)$ that 
are naturally associated to a domain in some spacetime, such as Hardy spaces 
(Example~\ref{ex:hardy}) or the standard subspaces $K(\cO)$ constructed in \cite{FG89} 
for free fields. From the perspective of antiunitary group representations, 
a natural class of representations of a pair $(G,G_1)$ are those realized 
in spaces $\cH_D \subeq C^{-\infty}(M)$ of distributions on a manifold $M$ on which $G_1$ acts. 
Here $\cH_D$ is the Hilbert space completion of the space $C^\infty_c(M)$ of test functions 
with respect to the scalar product given by a positive definite distribution $D$ on $M \times M$ 
via 
\[ \la \xi, \eta \ra_D = \int_{M \times M} \oline{\xi(x)}\eta(y)\, dD(x,y)\] 
(cf.\ \cite{NO14}). 
We associate to each open subset $\cO \subeq M$ a subset $\cH_D(\cO)$ 
generated by the space $C^\infty_c(\cO)$ of test functions supported in $\cO$. 
In this context it is an interesting problem to find natural 
antiunitary extensions of the representation of $G_1$ to $G$ such 
that some of the corresponding standard subspaces 
(Proposition~\ref{prop:3.2}) have natural geometric descriptions.  
In this context the detailed analysis of KMS conditions for unitary representations 
of $\R$ in  \cite{NO16} should be a crucial tool because one typically expects 
standard subspaces to be described in terms of analytic continuations of 
distributions on some domain $\cO \subeq M$ 
to a complex manifold containing $\cO$, 
which links this problem to~\S\ref{subsec:7.3b} (cf.\ \cite{NO17b}). 
As one also wants the modular 
unitaries to act geometrically on the manifold $M$, the case $G_1 = \R$ acting 
by translations on $\R$ considered in \cite{NO16} is of key importance. \\

Conversely, one may also consider Hilbert spaces $\cH$ of holomorphic functions 
on a complex manifold $M$ on which $G$ acts in such a way that $G_1$ acts 
by holomorphic maps and $G \setminus G_1$ by antiholomorphic ones. Then 
any $\gamma \in \Hom((\R^\times,\R^\times_+),(G,G_1))$ leads to a standard subspace 
of $\cH$. Many natural examples of this type arise from \S\S\ref{subsec:7.2b} and~\ref{subsec:7.3}. 
In particular, the representation of $\AU(\cH)$ on $\cF_+(\cH)$ is of this type 
if we identify $\cF_+(\cH)$ with the Hilbert space of holomorphic functions on $\cH$ 
with the reproducing kernel $K(\xi,\eta) = e^{\la \xi,\eta\ra}$ (cf.~\cite{Ne00}).

\subsection{Dual pairs in the Heisenberg group} 
\mlabel{subsec:7.4}

Let $\cH$ be a complex Hilbert space and $V \subeq \cH$ be a real linear 
subspace. We consider the  corresponding 
subgroup $\Heis(V) := \T \times V \subeq \Heis(\cH)$ (\S \ref{subsec:7.1}). 
The centralizer of this subgroup in $\Heis(\cH)$ coincides with 
$\Heis(V')$. If $V$ is closed, we thus obtain a 
{\it dual pair} $(\Heis(V), \Heis(V'))$ of subgroups in 
$\Heis(\cH)$ in the sense that both subgroups are their mutual centralizers. 

\begin{remark} For a closed linear subspace 
$V \subeq \cH$, we have 
$\Herm(\cH) = \oline{\Heis(V)\Heis(V')}$ if and only if 
$V + V'$ is dense in $\cH$, which is equivalent to $(V + V')' = V \cap V' = \{0\}$ 
(cf.~Lemma~\ref{lem:stand-factorial}). 
If this is the case, then $\cR(V) \subeq B(\cF_+(\cH))$ is a factor 
by Theorem~\ref{thm:araki-1}. Accordingly, the restriction 
of the irreducible Fock representation $(U, \cF_+(\cH))$ of 
$\Heis(\cH)$ to $\Heis(V)$ is a factor representation 
and the same holds for $\Heis(V')$. We thus obtain many 
interesting types of factor representations of Heisenberg 
groups of the type $\Heis(V)$ simply by restricting an 
irreducible representation of $\Heis(\cH)$. 
In \cite{vD71} this approach is used to realize quasi-free representations 
of $\Heis(V)$ in a natural way.  
\end{remark}

\begin{remark} (a) Suppose that $G$ is a group which is the product $G = G_1G_2$ 
of two subgroups $G_1$ and $G_2$ such that $G_1 = Z_G(G_2)$ and $G_2 = Z_G(G_1)$. 
Then $G_1 \cap G_2 = Z(G)$ and every unitary representation 
$(U,\cH)$ of $G$ restricts to factor representations of the subgroups 
$G_j$. 

(b) A typical example arises from  a von Neumann algebra 
$\cM \subeq B(\cH)$ in symmetric form, i.e., there exists a conjugation 
$J$ with $J\cM J = \cM'$ (Definition~\ref{def:symform}). 
Then $G := \U(\cM) \U(\cM')$ is a product of two subgroups 
$G_1 := \U(\cM)$ and $G_2 := \U(\cM')$ satisfying this condition. 
The representation of $G$ on $\cH$ is multiplicity free 
because $G' = \cM \cap \cM'$ is the center of $\cM$, hence abelian. 
It is irreducible if and only if $\cM$ is a factor, and then 
the representations of the subgroups $G_1$ and $G_2$ are factor representations. 
Note that the representation of $G$ extends to an antiunitary 
representation of $G \rtimes \{\1,j\}$, where 
$j(g) = JgJ$.
\end{remark}

\begin{remark}
Similar structures also arise for infinite dimensional 
Lie groups such as $\Diff(\bS^1)$, (doubly extended) loop groups 
and oscillator groups because modular objects provide information on restrictions 
of irreducible representations to factorial representations of subgroups 
(cf.~\cite{Wa98} for loop groups). 
So one should also try to develop the theory of modular localization 
for antiunitary representations of infinite dimensional Lie groups. 
\end{remark}

\subsection{A representation theoretic perspective on 
modular localization} 

The analysis of ordered families of von Neumann algebras 
with a common cyclic separating vector carried out by Borchers in 
\cite{Bo97} should also have a natural counterpart in the context 
of standard subspaces, in the spirit of the translation mechanism 
described in Subsection~\ref{subsec:4.2}. It would be interesting 
to see if the corresponding results can be formulated entirely 
in group theoretic terms, concerning multiplicative 
one-parameter groups of some pair $(G,G_1)$ 
(cf.~Proposition~\ref{prop:antiunirep-stand}). 
As we have seen in \S\S\ref{subsec:3.3} and \ref{subsec:modint}, this works 
perfectly well for 
half-sided modular inclusions and modular intersections, 

The same could be said about Wiesbrock's program, concerning 
the generation of Haag--Kastler nets from finite configurations 
of von Neumann algebras with common cyclic separating vectors 
(\cite{Wi93c, Wi97b, Wi98, KW01}). 

\appendix 

\section{Appendices} 

In this short appendix we collect some general lemmas used in the main 
text. 

\subsection{A lemma on von Neumann algebras} 
\mlabel{app:a.3} 

\begin{lemma} \mlabel{lem:app.1}
Let $\cM\subeq \cH$ be a von Neumann algebra, 
$\alpha \: \cM \to \cM$ a real-linear weakly continuous automorphism and $U \in \U(\cM)$ be a unitary element. Then 
the following assertions hold: 
\begin{itemize}
\item[\rm(a)] If $\alpha$ is complex linear and $\alpha(U)= U$, 
then there exists a $V \in \U(\cM)$ with $\alpha(V) = V$ and $V^2 = U$.
\item[\rm(b)] If $\alpha$ is complex linear and $\alpha(U)= U^{-1}$ 
with $\ker(U+\1) = \{0\}$, then there exists a $V \in \U(\cM)$ with 
$\alpha(V) = V^{-1}$ and $V^2 = U$.
\item[\rm(c)] If $\alpha$ is antilinear and $\alpha(U)= U^{-1}$, 
then there exists a $V \in \U(\cM)$ with $\alpha(V) = V^{-1}$ and $V^2 = U$.
\item[\rm(d)] If $\alpha$ is antilinear and $\alpha(U)= U$  
with $\ker(U+\1) = \{0\}$, then there exists a $V \in \U(\cM)$ with 
$\alpha(V) = V$ and $V^2 = U$.
\end{itemize}
\end{lemma}

\begin{proof} (a) Let $P$ be the spectral measure of $U$ on the circle 
$\T \subeq \C$ with $U = \int_{\T} z \, dP(z)$. As $P$ is uniquely determined 
by $U$ and $\alpha$ is complex linear, we obtain 
$\alpha(P(E)) = P(E)$ for every measurable subset $E \subeq \T$. 
For every measurable function $S \:  \T \to \T$ with $S(z)^2 = z$ for $z\in \T$, 
we obtain by $V := \int_\T S(z)\, dP(z)$ a square root of $U$ fixed by $\alpha$. 

(b) Now our assumptions implies that $P(\{-1\}) = 0$, 
so that we find a unique spectral measure $\tilde P$ on the open 
interval $(-\pi, \pi)$ with $U = \int_{-\pi}^\pi e^{i\theta}\, d\tilde P(\theta)$. 
From $\alpha(U) = U^{-1}$ we derive $\alpha(\tilde P(E)) = \tilde P(-E)$ 
for every measurable subset $E \subeq (-\pi,\pi)$. Therefore 
$V := \int_{-\pi}^\pi e^{i\theta/2}\, dP(\theta)$ is a square root of $U$ satisfying 
$\alpha(V) = V^{-1}$. 

(c) We consider the von Neumann algebra $\cN \subeq \cM$ generated by 
$U$. Then $\cN$ is abelian and $\alpha$-invariant. Therefore 
$\beta(N) := \alpha(N^*)$ defines a complex linear automorphism of $\cN$. 
It satisfies $\beta(U) = U$, so that the existence of $V$ follows from (a). 

(d) We argue as in (c), but now (b) applies. 
\end{proof}

\begin{ex} For $\cM = \C$, $\alpha = \id_\C$ and $U = - \1$ there exists no 
unitary $V \in \U(\cM)$ with $V^2 = U$ and $V = \alpha(V) = V^{-1}$. 
This shows that the extra assumption in (b) is really needed. 
\end{ex}

\subsection{Two  lemmas on groups} 

\begin{lemma} \mlabel{lem:abstract-grp} 
Let $G$ be a group and $G^\sharp := (G \times G) \rtimes \{\1,\tau\}$, 
where $\tau(g,h) = (h,g)$ is the flip automorphism. We consider the subset 
$D := \{ (g,g^{-1}) \: g \in G \}$. 
Then the following assertions hold: 
\begin{itemize}
\item[\rm(a)] $\grp(D) = (G' \times \{e\}) D,$ 
where $G' \subeq G$ is the commutator subgroup. 
\item[\rm(b)] The conjugacy class $C_{(e,e,\tau)}$ 
of the involution $(e,e,\tau) \in G^\sharp$ coincides with 
$D\times \{\tau\}$ and 
\[ \grp(C_{(e,\tau)}) = \big((G' \times \{e\})D\big) \times \{\1,\tau\}.\] 
\end{itemize}
\end{lemma} 

\begin{proof} (a) For $g,h \in G$, the relation 
\[ (g,g^{-1})(h,h^{-1}) =  (gh,g^{-1}h^{-1}) 
= (ghg^{-1}h^{-1},e)(hg, (hg)^{-1}) \] 
implies that $G' \times \{e\} \subeq \grp(D)$. 
Conversely, it is easy to see that $(G' \times \{e\}) D$ is a subgroup 
of $G \times G$. 

(b) The first assertion follows from 
$(g,h,\1)(e,e,\tau)(g,h,\1)^{-1} = (gh^{-1}, hg^{-1}, \tau)$. 
The second assertion follows from (a). 
\end{proof}

\begin{lemma} \mlabel{lem:ext-homo} Let 
$G_1 \subeq G$ be a subgroup of index two, 
$r \in G \setminus G_1$, 
and $\phi \: G_1 \to H$ be a group homomorphism. 
If $h \in H$ satisfies $h^2 = \phi(r^2)$ and 
$h\phi(g)h^{-1} = \phi(rgr^{-1})$ for $g \in G_1$, 
then $\hat\phi(gr) := \phi(g) h$ and $\hat\phi(g) := \phi(g)$ 
for $g \in G_1$ defines an extension of 
$\phi$ to a homomorphism $\hat\phi \: G \to H$. 
\end{lemma}

\begin{proof} First we observe that 
$\hat\phi(gu) = \hat\phi(g)\hat\phi(u)$ obviously holds for $g \in G_1$ and $u \in G$. 
For $g \in G_1$ we further have 
\[ \hat\phi(rg) = \hat\phi(rgr^{-1}r) = \phi(rgr^{-1}) h 
= h\phi(g) h^{-1} h 
= h\phi(g) = \hat\phi(r) \hat\phi(g).\] 
Finally, we note that, for $g \in G_1$, 
\[ \hat\phi(rgr) = \hat\phi(rgr^{-1} r^2) 
= \phi(rgr^{-1})\phi(r^2) 
= h\phi(g) h^{-1} \phi(r^2) 
= h\phi(g) h = \hat\phi(r)  \hat\phi(gr).\] 
This implies that $\hat\phi$ is a group homomorphism.
\end{proof}

\subsection{Symmetric spaces} 
\mlabel{app:a.2}

\begin{definition} (a) Let $M$ be a set and 
\[ \mu \: M \times M \to M, \quad (x,y) \mapsto x \cdot y =: s_x(y) \] 
be a map with the following properties: 
\begin{itemize}
\item[\rm(S1)] $x \cdot x=x$ for all $x \in M$, i.e., $s_x(x) = x$.
\item[\rm(S2)] $x \cdot (x \cdot y) =y$ for all $x,y \in M$, i.e., $s_x^2 = \id_M$. 
\item[\rm(S3)] $s_x(y \cdot z) = s_x(y)\cdot s_x(z)$ for all $x,y \in M$, 
i.e., $s_x \in \Aut(M,\mu)$. 
\end{itemize}
Then we call $(M,\mu)$ a {\it reflection space}. 

(b) If $M$ be a smooth manifold and $\mu \: M \times M \to M$ is a smooth map 
turning $(M,\mu)$ into a reflection space, then it is called a 
{\it symmetric space} (in the sense of Loos) if, in addition, 
each $x$ is an isolated fixed point of $s_x$ (\cite{Lo69}). 
\end{definition}

\begin{ex} (a) Any group $G$ is a reflection space 
with respect to the product
\[  g \bullet h := s_g(h) := gh^{-1}g.\] 
The subset $\Inv(G)$ of involutions in $G$ is a reflection subspace 
on which the product takes the form $s_g(h) := ghg = ghg^{-1}$. 
More generally, the subset 
\[ G_2 := \{ g \in G \: g^2 \in Z(G)\} \] 
is a reflection subspace of $G$ because $g^2, h^2 \in Z(G)$ implies 
\[ (gh^{-1}g)^2 = gh^{-1}g^2 h^{-1}g = g^2 h^{-2} g^2 \in Z(G).\] 
This calculation even shows that the square map 
$G_2 \to Z(G), g \mapsto g^2$ is a homomorphism of reflection  spaces. 

(b) Suppose that $G$ is a Lie group and $\tau \in \Aut(G)$ an involution. 
For any open subgroup $H\subeq G^\tau = \Fix(\tau)$, we obtain on the coset space 
$M = G/H$ the structure of a symmetric space by 
\[ xH \bullet y H := x\tau(x)^{-1} \tau(y) H.\]
\end{ex}

\subsection*{Acknowledgments}

K.-H.~Neeb acknowledges supported by DFG-grant NE 413/9-1. 
The research of G.~\'Olafsson was supported by NSF grant  DMS-1101337.
Both authors thank Arthur Jaffe for supporting a visit at Harvard 
University, where some of this work was done.


\begin{thebibliography}{aaaaaaa}


\bibitem[Ar63]{Ar63} Araki, H., {\it 
A lattice of von Neumann algebras associated with the 
quantum theory of a free Bose field}, 
J. Math. Phys. {\bf 4} (1963), 1343--1362 
 
\bibitem[Ar64]{Ar64} ---, {\it von Neumann algebras of local observables 
for free scalar field}, J. Mathematical Phys. {\bf 5} (1964),  1--13

\bibitem[Ar99]{Ar99} ---,  ``Mathematical Theory of Quantum Fields,'' 
Int. Series of Monographs on Physics, Oxford Univ. Press, Oxford, 1999 

\bibitem[AW63]{AW63} Araki, H., and E. J. Woods, {\it 
Representations of the canonical commutation relations describing 
a nonrelativistic infinite free Bose gas}, 
J. Math. Phys. {\bf 4} (1963), 637--662 

\bibitem[AW68]{AW68} ---, {\it A classification of factors}, 
Publ. RIMS, Kyoto Univ. Ser. A {\bf 3} (1968), 51--130

\bibitem[AZ05]{AZ05} Araki, H., and L.~Zsid\'o, {\it 
Extension of the structure theorem of Borchers and its application to 
half-sided modular inclusions}, Rev. Math. Phys. {\bf 17:5} (2005), 491--543

\bibitem[Ba64]{Ba64} Bargmann, V., {\it Note on Wigner's Theorem on Symmetry Operations}, 
Journal of Mathematical Physics {\bf 5:7} (1964), 862--868 

\bibitem[BJL02]{BJL02} Baumg\"artel, H., Jurke, M., and F. Lledo, 
{\it Twisted duality of the CAR-algebra}, J. Math. Physics {\bf 43:8} (2002), 
4158--4179

\bibitem[BCL10]{BCL10} Bertozzini, P., R. Conti, and 
W. Lewkeeratiyutkul, {\it Modular theory, Non-commutative Geometry 
and Quantum Gravity}, in ``Symmetry, Integrability and Geometry: 
Methods and Applications,'' SIGMA {\bf 6} (2010), 067; 47pp 

\bibitem[Be96]{Be96} Bertram, W., {\it Un th\'eor\`eme de Liouville pour 
les alg\`ebres de Jordan}, Bull. Soc. Math. France {\bf 124:2} (1996), 299--327 

\bibitem[BW76]{BW76} Bisognano, J. J., and E. H. Wichmann,  {\it On the duality 
condition for quantum fields}, J. Math. Phys. {\bf 17} (1976), 303--321 

\bibitem[Bla06]{Bla06}
  Blackadar, B., ``Operator Algebras,''
Encyclopaedia of Mathematical Sciences Vol. {\bf 122},
Springer-Verlag, Berlin, 2006.

\bibitem[Bo68]{Bo68} Borchers, H.-J., {\it On the converse of the 
Reeh--Schlieder theorem}, Comm. Math. Phys. {\bf 10} (1968), 
269--273 

\bibitem[Bo92]{Bo92} ---, {\it The CPT-Theorem in two-dimensional 
theories of local observables}, Comm. Math. Phys. {\bf 143} (1992), 315--332 

\bibitem[Bo95]{Bo95} ---, {\it On the use of modular groups in 
quantum field theory}, Ann. Ins. H. Poincar\'e {\bf 63:4} (1995), 331--382 

\bibitem[Bo96]{Bo96} ---, {\it Half-sided modular inclusion 
and the construction of the Poincar\'e group}, Commun. Math. Phys. 
{\bf 179} (1996), 703--723 

\bibitem[Bo97]{Bo97} ---, {\it On the lattice of subalgebras associated 
with the principle of half-sided modular inclusion}, 
Lett. Math. Phys. {\bf 40:4} (1997), 371--390 

\bibitem[Bo98]{Bo98} ---, {\it Half-sided translations and the type 
of von Neumann algebras},  Lett. Math. Phys. {\bf 44} (1998), 283--290 

\bibitem[Bo00]{Bo00} ---,  {\it On revolutionizing quantum field theory 
with Tomita's modular theory}, J. Math. Phys. {\bf 41} (2000), 3604--3673; 
extended version with complete proofs available under \\
ftp://ftp.theorie.physik.uni-goettingen.de/pub/papers/lqp/99/04/esi773.ps.gz


\bibitem[BR87]{BR87} Bratteli, O., and D.~W.~Robinson, ``Operator Algebras and Quantum Statistical
Mechanics I,'' 2nd ed.,
Texts and Monographs in Physics, Springer-Verlag, 1987 

\bibitem[BR96]{BR96} ---, ``Operator Algebras and Quantum Statistical
Mechanics II,'' 2nd ed.,
Texts and Monographs in Physics, Springer-Verlag, 1996 

\bibitem[BtD85]{BtD85} Br\"ocker, T., and T.~tom Dieck, ``Representations of Compact 
Lie Groups,'' Springer, Graduate Texts in Mathematics {\bf 98}, 1985 


\bibitem[BGL93]{BGL93} Brunetti, R., Guido, D., and R. Longo, {\it 
Modular structure and duality in conformal quantum field theory}, 
Comm. Math. Phys. {\bf 156} (1993), 210--219 

\bibitem[BGL02]{BGL02} ---, {\it 
Modular localization and Wigner particles}, Rev. Math. Phys. {\bf  14} (2002), 759--785


\bibitem[Bu78]{Bu78} Buchholz, D., {\it On the structure of local quantum 
fields with non-trivial interaction}, in ``Proceedings of the 
Int. Conf. on Operator Algebras, Ideals and their Application in Theoretical 
Physics,'' Baumg\"artel, Lassner, Pietsch, Uhlmann (eds.), 146--153; 
Teubner Verlagsgesellschaft, Leipzig 

\bibitem[BDFS00]{BDFS00} Buchholz, D., O.~Dreyer, M. Florig, and S. J. Summers, 
{\it Geometric modular action and spacetime symmetry groups}, Rev. Math. Phys. {\bf 12:4} (2000), 475--560 

\bibitem[BLS11]{BLS11} Buchholz, D., Lechner, G., and S. J. Summers, {\it 
Warped convolutions, Rieffel deformations and the construction of quantum 
field theories}, Comm. Math. Phys. {\bf 304:1} (2011), 95--123 

\bibitem[BS93]{BS93} Buchholz, D., and S. J. Summers, 
{\it An algebraic characterization 
of vacuum states in Minkowski space}, Comm. Math. Phys. {\bf 155:3} (1993), 449--458

\bibitem[Co73]{Co73} Connes, A., {\it Une classification des facteurs de type
III}, Annales scientifiques de l’\'E.N.S. 4i\'eme s\'erie {\bf 6:2}
(1973), 133--252

\bibitem[Co74]{Co74} ---, {\it  Caract\'erisation des espaces 
vectoriels ordonn\'es sous-jacents aux alg\`ebres de von Neumann}, 
Annales de l'institut Fourier {\bf 24:4} (1974), 121--155 

\bibitem[CR94]{CR94} Connes, A., and C. Rovelli, {\it 
Von Neumann Algebra Automorphisms and Time-Thermodynamics Relation 
in General Covariant Quantum Theories}, Class. Quant. Grav. {\bf 11} 
(1994), 2899--2918 

\bibitem[vD71]{vD71} van Daele, {\it Quasi-equivalence of quasi-free states 
on the Weyl algebra}, Commun. Math. Phys. {\bf 21} (1971), 171--191


\bibitem[EO73]{EO73} Eckmann, J.-P., and K.~Osterwalder, {\it An application of Tomita's theory 
of modular Hilbert algebras: duality for free Bose Fields}, J. Funct. Analysis {\bf 13:1} (1973), 1--12 

\bibitem[FK94]{FK94} Faraut, J., and A. Koranyi, ``Analysis on symmetric cones,'' 
Oxford Math.\ Monographs, Oxford University Press, 1994 

\bibitem[FG89]{FG89} Figliolini, F., and D. Guido, {\it The Tomita operator for the free 
scalar field}, Ann. Inst. H. Poincar\'e, Phys. Th\'eor. {\bf 51} (1989), 419--435 

\bibitem[Fl98]{Fl98} Florig, M., {\it On Borchers' Theorem}, Lett. Math. Phys. 
{\bf 46} (1998), 289--293 

\bibitem[Fo83]{Fo83} Foit, J., {\it Abstract twisted duality for quantum free Fermi fields}, 
Publ. Res. Inst. Math. Sci. {\bf 19:2} (1983), 729--741

\bibitem[Gu11]{Gu11} Guido, D., {\it Modular theory for the von Neumann 
algebra of local quantum physics}, Contemp. Math. {\bf 534} (2011), 97-120 

\bibitem[GL95]{GL95} Guido, D., and R. Longo, {\it An algebraic spin and statistics 
theorem}, Commun. Math. Physics {\bf 172} (1995), 517--533 


\bibitem[GLW98]{GLW98} Guido, D., Longo, R., and H.-W. Wiesbrock, 
{\it Extensions of conformal nets and superselection structures}, 
Commun. Math. Physics {\bf 192} (1998), 217--244 


\bibitem[Ha96]{Ha96} Haag, R., ``Local Quantum Physics. Fields, Particles, Algebras,'' 
Second edition, Texts and Monographs in Physics,  Springer-Verlag, Berlin, 1996

\bibitem[HN12]{HN12} Hilgert, J., and K.-H. Neeb, 
``Structure and Geometry of Lie Groups'', Springer, 2012 

\bibitem[H{\'O}96]{HO96} Hilgert, J., and G. {\'O}lafsson. {\it Causal Symmetric Spaces, Geometry and Harmonic
Analysis}. Perspectives in Mathematics {\bf 18}, Academic Press, 1996

\bibitem[HL82]{HL82} Hislop, P., and R. Longo, {\it Modular structure of the local 
observables associated with the free massless scalar field theory}, 
Comm. Math. Phys. {\bf 84:1} (1982), 71--85 

\bibitem[JM17]{JM17} J\"akel, C., and J. Mund, {\it The 
Haag--Kastler axioms for the $P(\phi)_2$ model on the de Sitter space}, 
arXiv:math-ph:1701.08231 

\bibitem[KW01]{KW01} K\"ahler, R., and H.-W. Wiesbrock {\it Modular theory and the 
reconstruction of four-dimensional quantum field theories}, 
J. Math. Phys. {\bf 42:1} (2001), 74--86

\bibitem[Ka97]{Ka97} Kaup, W., {\it On real Cartan factors}, 
manuscripta math. {\bf 92} (1997), 191--222 

\bibitem[Ke96]{Ke96} Keyl, M., {\it Causal spaces, causal complements and their 
relations to quantum field theory}, 
Rev. Math. Phys. {\bf 8:2} (1996), 229--270 

\bibitem[Ke98]{Ke98} ---, {\it How to describe the space-time structure 
with nets of $C^*$-algebras}, Internat. J. of Theoretical Physics {\bf 37:1} 
(1998), 375--385 

\bibitem[Kl11]{Kl11} Klotz, M., {\it Banach Symmetric Spaces}, 
Preprint, arXiv:math.DG:0911.2089 

\bibitem[KS05]{KS05} Kr\"otz, B., and R. J. Stanton, {\it Holomorphic extension 
of representations: (II) geometry and harmonic analysis}, Geom. Funct. Anal. {\bf 15} 
(2005), 190--245 

\bibitem[Le15]{Le15} Lechner, G., {\it Algebraic constructive quantum field theory: 
Integrable models and deformation techniques}, 
Preprint, arXiv:math.ph:1503.03822
 
\bibitem[LL14]{LL14} Lechner, G., and R. Longo, {\it Localization in Nets of Standard 
Spaces}, Comm. Math. Phys., to appear  

\bibitem[Lle09]{Lle09} 
Lled\'o, F., {\it Modular theory by example}, 
Preprint, arXiv:math.OA:0901.1004v1 

\bibitem[Lo82]{Lo82} Longo, R., {\it Algebraic and modular structure 
of von Neumann algebras of physics}, in ``Operator algebras and 
applications, Part 2 (Kingston, Ont., 1980),'' pp. 551--566, 
Proc. Sympos. Pure Math., 38, Amer. Math. Soc., Providence, R.I., 1982

\bibitem[Lo08]{Lo08} ---, {\it Real Hilbert subspaces, modular theory, SL(2, R) and CFT} 
in ``Von Neumann Algebras in
Sibiu'', 33-91, Theta Ser. Adv. Math. {\bf 10}, Theta, Bucharest, 

\bibitem[LW11]{LW11} Longo, R., and E. Witten, {\it An algebraic construction of 
boundary quantum field theory}, Comm. Math. Phys. {\bf 303:1} (2011), 213--232 

\bibitem[Lo69]{Lo69} Loos, O., ``Symmetric spaces I: General theory,'' 
W. A. Benjamin, Inc., New York, Amsterdam, 1969

\bibitem[Mo17]{Mo17} Morinelli, V., {\it The Bisognano--Wichmann property on nets 
of standard subspaces, some sufficient conditions}, 
arXiv:math-ph:1703.06831 

\bibitem[Mu01]{Mu01} Mund, J., {\it The Bisognano--Wichmann theorem for massive 
theories}, Ann. Henri Poincar\'e {\bf 2:5} (2001),  907--926 

\bibitem[MSY06]{MSY06} Mund, J., Schroer, B., and J.~Yngvason, {\it 
String-localized quantum fields and modular localization}, Comm. Math. Phys. 
{\bf 268:3} (2006), 621--672 

{\it The Bisognano--Wichmann theorem for massive 
theories}, Ann. Henri Poincar\'e {\bf 2:5} (2001),  907--926 

\bibitem[Ne00]{Ne00} Neeb, K.-H., ``Holomorphy and Convexity in Lie Theory,'' 
Expositions in Mathematics {\bf 28}, de Gruyter Verlag, Berlin, 2000 

\bibitem[N\'O14]{NO14} Neeb, K.-H., G. \'Olafsson,  {\it Reflection 
positivity and conformal symmetry}, J. Funct. Anal.   266 (2014), 2174--2224

\bibitem[NO15]{NO15} ---,  {\it Reflection 
positive one-parameter groups and dilations}, Complex 
Analysis and Operator Theory {\bf 9:3} (2015), 653--721 

\bibitem[N\'O16]{NO16} ---, {\it 
KMS conditions, standard subspaces and reflection positivity on the circle 
group}, arXiv:math-ph:1611.00080; submitted

\bibitem[N\'O17]{NO17} ---, {\it Antiunitary representations of hermitian groups}, 
in preparation 

\bibitem[N\'O17b]{NO17b} ---, {\it Reflection positivity on spheres and hyperboloids}, 
in preparation 

\bibitem[OM16]{OM16} Oppio, M., and V. Moretti, {\it 
Quantum theory in real Hilbert space: How the complex Hilbert space 
structure emerges from Poincar\'e symmetry}, 
Preprint arXiv:math-ph:1611.09029v1 28 Nov 2016 

\bibitem[Ra17]{Ra17} Raasakka, M., {\it Spacetime-free approach to 
quantum theory and effective spacetime structure}, 
arXiv:1605.03942v2 [gr-qc] 24 Jan 2017 

\bibitem[RS73]{RS73} Reed, S., and B. Simon, 
``Methods of Modern Mathematical Physics I: Functional Analysis,'' 
Academic Press, New York, 1973

\bibitem[RS61]{RS61} Reeh, H., and S. Schlieder, {\it Bemerkungen zur 
Unit\"ar\"aquivalenz von Lorentzinvarianten Feldern}, Nuovo Cimento 
{\bf 22} (1961), 1051--1068 

\bibitem[Sa97]{Sa97} Salehi, H., {\it Problems of dynamics in generally covariant 
quantum field theory}, Internat. J. Theoret. Phys. {\bf 36:1} (1997), 143--155
 
\bibitem[Sch97]{Sch97} Schroer, B., {\it Wigner representation theory of the Poincar\'e group, 
localization, statistics and the S-matrix},  Nuclear Phys. {\bf B 499-3} (1997), 519--546

\bibitem[Sch06]{Sch06} ---, {\it Positivity and Integrability 
(Mathematical Physics at the FU-Berlin)}, 
Preprint, arXiv:hep-th/0603118 

\bibitem[Sch09]{Sch09} ---, {\it Localization and the interface between quantum 
mechanics, quantum field theory and quantum gravity II. The search of the interface 
between QFT and QG}, Stud. Hist. Philos. Sci. B Stud. Hist. Philos. Modern 
Phys. {\bf 41:4} (2010), 293--308

\bibitem[SW00]{SW00} Schroer, B., and H.-W.\ Wiesbrock, 
{\it Modular theory and geometry}, Rev. Math. Phys. {\bf 12:1} (2000), 139--158 

\bibitem[So10]{So10} Solveen, C., {\it The Bisognano-Wichmann Theorem and nets 
on $\R^4$}, manuscript from http://math.mit.edu/~eep/CFTworkshop

\bibitem[StVa02]{StVa02} Stalder, Y., and A. Valette, 
{\it Le lemme de Schur pour les repr\'esentations orthogonales}, 
Expo. Math. {\bf 20} (2002), 279--285 

\bibitem[Su05]{Su05} Summers, S., {\it Tomita--Takesaki modular theory}, 
arXiv:0511.034v1 [math-ph] 9 Nov 2005 

\bibitem[SW03]{SW03} Summers, S., and R. K. White, {\it On deriving space-time 
from quantum observables and states}, Commun. Math. Phys. {\bf 237} (2003), 203--220 

\bibitem[Ta12]{Ta12}
Tanimoto, Y., {\it Construction 
of wedge local nets of observables through Longo--Witten 
endomorphisms}, Comm. Math. Phys. {\bf 314:2} (2012), 443--469 

\bibitem[Tr97]{Tr97} Trebels, S., {\it \"Uber die geometrische Wirkung 
modularer Automorphismen}, PhD Thesis, Univ. G\"ottingen, 1997 

\bibitem[Va85]{Va85} Varadarajan, V. S., ``Geometry of Quantum Theory,'' 
Springer Verlag,  1985 

\bibitem[Wa98]{Wa98} Wassermann, A., {\it 
Operator algebras and conformal field theory. 
III. Fusion of positive energy representations
of LSU(N) using bounded operators}, Invent. Math. {\bf 133} (1998), 467--538

\bibitem[Wi92]{Wi92} Wiesbrock, H.-W., {\it A comment on a recent work of Borchers}, 
Lett. Math. Phys. {\bf 25} (1992), 157--159 

\bibitem[Wi93]{Wi93} ---, {\it Half-sided modular inclusions of von 
Neumann algebras}, Commun. Math. Phys. {\bf 157} (1993), 83--92 

\bibitem[Wi93b]{Wi93b} ---, {\it Conformal quantum field theory 
and half-sided modular inclusions of von Neumann algebras}, 
Comm. Math. Phys. {\bf 158:3} (1993), 537--543 

\bibitem[Wi93c]{Wi93c} ---, {\it Symmetries and half-sided modular 
inclusions of von Neumann algebras}, 
Lett. Math. Phys. {\bf 28:2} (1993), 107--114

\bibitem[Wi97]{Wi97} ---, {\it Half-sided modular inclusions of von 
Neumann algebras, Erratum}, Commun. Math. Phys. {\bf 184} (1997), 683--685 

\bibitem[Wi97b]{Wi97b} ---, {\it Symmetries and modular intersections of 
von Neumann algebras},  Lett. Math. Phys. {\bf 39:2} (1997), 203--212

\bibitem[Wi98]{Wi98} ---, {\it  Modular intersections of von Neumann algebras 
in quantum field theory}, Comm. Math. Phys. {\bf 193:2} (1998), 269--285 

\bibitem[Wig59]{Wig59} Wigner, E., ``Group Theory and its Applications 
to the Quantum Mechanics of Atomic Spectra,'' 
Academic Presse, 1959 

\end{thebibliography}
\end{document}